\documentclass{article}

\usepackage[colorlinks,citecolor=blue,linkcolor=blue,urlcolor=blue,backref=page]{hyperref}
\usepackage{amsmath,amssymb,amsthm,asymptote,dsfont,enumerate,float,mathrsfs,resmes,soul,stmaryrd,upref,xcolor}
\DeclareMathAlphabet\mathbb{U}{msb}{m}{n}
\usepackage[ngerman,main=american]{babel}\addto\extrasamerican{\languageshorthands{ngerman}\useshorthands{"}}
\usepackage[tmargin=2.5cm,lmargin=2.5cm,rmargin=2.5cm,bmargin=2.5cm]{geometry}

\newcommand\1{{\mathds{1}}}
\renewcommand\ae{{a.\@e.\@}}
\newcommand\B{\mathrm{B}}
\newcommand\BV{\mathrm{BV}}
\newcommand\C{\mathrm{C}}
\renewcommand\c{\mathrm{c}}
\newcommand*{\coleq}{\mathrel{\vcenter{\baselineskip0.5ex\lineskiplimit0pt\hbox{\normalsize.}\hbox{\normalsize.}}}=}
\newcommand*{\colequiv}{\mathrel{\vcenter{\baselineskip0.5ex\lineskiplimit0pt\hbox{\normalsize.}\hbox{\normalsize.}}}\equiv}
\newcommand\D{\mathrm{D}}
\renewcommand\d{\mathrm{d}}
\renewcommand\div{\mathrm{div}}
\newcommand\dx{{\,\d x}}
\newcommand\eps{\varepsilon}

\newcommand\ip{\boldsymbol{\cdot}}
\newcommand\ka{\textup{(}}
\newcommand\kz{\textup{)}}
\renewcommand\L{\mathrm{L}}
\newcommand\llangle{\big\langle\!\!\big\langle}
\newcommand\LN{\mathcal{L}^N}
\newcommand\loc{\mathrm{loc}}
\renewcommand\H{{\mathcal{H}}}
\newcommand\N{{\mathds{N}}}
\renewcommand\P{\mathrm{P}}
\newcommand\p{\varphi}
\newcommand\qq{\qquad}
\newcommand\R{{\mathds{R}}}
\newcommand\rrangle{\big\rangle\!\!\big\rangle}

\newcommand\s{\mathrm{s}}
\DeclareMathOperator\TV{TV}
\newcommand\W{\mathrm{W}}

\DeclareSymbolFont{extraup}{U}{zavm}{m}{n}
\DeclareMathSymbol{\vardiamond}{\mathalpha}{extraup}{87}

\newcommand{\cupdot}{\ensuremath{\,\mathaccent\cdot\cup\,}}
\newcommand{\ol}{\overline}
\newcommand\dt{{\,\d t}}
\renewcommand\div{\mathrm{div}}

\newtheorem{thm}{Theorem}[section]
\newtheorem{cor}[thm]{Corollary}
\newtheorem{lem}[thm]{Lemma}
\newtheorem{prop}[thm]{Proposition}

\theoremstyle{definition}
\newtheorem{assum}[thm]{Assumption}
\newtheorem{defi}[thm]{Definition}
\newtheorem{examp}[thm]{Example}
\newtheorem{rem}[thm]{Remark}

\newcommand\finf{{f^\infty}}
\newcommand\fpers{{\,\overline{\!f}}}
\newcommand\Om{\Omega}
\DeclareMathOperator{\RM}{RM}

\newcommand\MF{\mathcal{F}}
\newcommand\Omd{{\Omega_\lozenge}}
\newcommand\nud{{\nu}_{\Omega_\lozenge}}
\newcommand\wdi{{w_\lozenge}}
\newcommand\udi{{u_\lozenge}}
\newcommand\mud{{\mu_\lozenge}}
\newcommand\Ds{{\D^\s}}
\newcommand\A{{\mathcal{A}}}
\renewcommand\c{{\mathrm{c}}}

\newcommand\pOm{{\Omega^\prime}}

\numberwithin{equation}{section}

\begin{document}

\title{\vspace{-5ex}Existence theory for linear-growth variational integrals\\with signed measure data}
	
\author{
	Eleonora Ficola\footnote{Fachbereich Mathematik, Universit\"at Hamburg, Bundesstr. 55, 20146 Hamburg, Germany.\newline
	Email address: \href{mailto:eleonora.ficola@uni-hamburg.de}{\tt eleonora.ficola@uni-hamburg.de}.
	}
	\qq\qq
	Thomas Schmidt\footnote{Fachbereich Mathematik,
	Universit\"at Hamburg, Bundesstr. 55, 20146 Hamburg, Germany.\newline
	Email address: \href{mailto:thomas.schmidt.math@uni-hamburg.de}{\tt thomas.schmidt.math@uni-hamburg.de}.
	URL: \href{http://www.math.uni-hamburg.de/home/schmidt/}{\tt http:/\!/www.math.uni-hamburg.de/home/schmidt/}.
	}
}

\date{February 5, 2026}

\maketitle

\vspace{-4ex}

\begin{abstract}
  We develop a semicontinuity-based existence theory in $\BV$ for a general class of scalar linear-growth variational integrals with additional signed-measure terms. The results extend and refine previous considerations for anisotropic total variations and area-type cases, and they pave the way for a variational approach to the corresponding Euler-Lagrange equations, which involve the signed measure as right-hand-side datum.
\end{abstract}

\medskip


\tableofcontents
	
\section{Introduction}

\noindent\textbf{\boldmath Linear-growth functionals with measure.}
In this paper we are concerned with minimization among scalar functions $w\colon\Om\to\R$ of first-order variational integrals of type
\begin{equation}\label{eq:formal_functional}
  \int_\Om f(\,.\,,\nabla w)\dx+\int_\Omega w\,\d\mu\,,
\end{equation}
where the given data are a dimension $N\in\N$, a bounded Lipschitz domain $\Om\subseteq\R^N$, a continuous integrand $f\colon\ol{\Omega}\times\R^N\to[0,\infty)$, and a finite signed Radon measure $\mu$ on $\Omega$. Our motivation partially stems from the corresponding Euler-Lagrange equation, in which $\mu$ takes the role of the right-hand side. This equation --- valid for minimizers $u\colon\Omega\to\R$ at least in suitably differentiable cases --- reads
\begin{equation}\label{eq:EL}
  \div\,\big[\nabla_\xi f(\,.\,,\nabla u)\big]=\mu
  \qq\qq\text{in }\Om\,.
\end{equation}
Specifically, we now focus on the case that $f$ is convex in its second variable $\xi\in\R^N$ and satisfies the linear growth condition
\begin{equation}\label{eq:lin_growth_intro}
  \alpha|\xi|\le f(x,\xi)\le\beta(|\xi|+1)
  \qq\qq\text{for all }(x,\xi)\in\Om\times\R^N
\end{equation}
with constants $0<\alpha\le\beta<\infty$. In this situation, the natural energy space in which general existence results for minimizers of \eqref{eq:formal_functional} can be approached is the space $\BV(\Om)$ of functions of bounded variation on $\Om$. However, some care is needed already in giving a good definition of the integral \eqref{eq:formal_functional} for arbitrary $w\in\BV(\Om)$, since the first term should take into account the derivative measure $\D w\in\RM(\Om,\R^N)$, while the second term requires suitable $\mu$-\ae{} evaluation of $w$.

\medskip

\noindent\textbf{\boldmath First-order convex terms on $\BV$.}
In case of the first term in \eqref{eq:formal_functional}, the common $\BV$ interpretation goes back to \cite{GofSer64,GMS79} and consists in using the convex functional of measures (compare Sections \ref{subsec:integrands} and \ref{subsec:func_meas})
\[
  \int_\Om f(\,.\,,\D w)
  \coleq\int_\Om f(\,.\,,\nabla w)\,\d\LN
  +\int_\Om\finf\left(.\,,\frac{\d\D^\s w}{\d|\D^\s w|}\right)\,\d|\D^\s w|
  \qq\qq\text{for }w\in\BV(\Om)\,.
\]
Here, $\D w=(\nabla w)\LN+\D^\s w$ is the Lebesgue decomposition of $\D w$ with respect to $\LN$, and in our convex case the recession function $\finf\colon\ol\Om\times\R^N\to[0,\infty)$ is given by $\finf(x,\xi)\coleq\lim_{t\to0^+}tf(x,\xi/t)$ and will also be assumed continuous. Moreover, if one adds (as we will always do) a Dirichlet boundary condition, conveniently specified in terms of some $u_0\in\W^{1,1}(\R^N)$, then in the $\BV$ setting this condition needs to be recast in the sense of an additional boundary penalization term, that is, the first term in \eqref{eq:formal_functional} is replaced with
\[
  \int_\Om f(\,.\,,\D w)
  +\int_{\partial\Om}\finf(\,.\,,(w{-}u_0)\nu_\Om)\,\d\H^{N-1}
  \qq\qq\text{for arbitrary }w\in\BV(\Om)\,,
\]
where $\nu_\Om$ denotes the inward normal at $\partial\Om$ (see again Sections \ref{subsec:integrands} and \ref{subsec:func_meas} for more notation and details).

\medskip

\noindent\textbf{\boldmath Zero-order measure terms on $\BV$.}
For what concerns the second term in \eqref{eq:formal_functional}, a first plausible approach is to understand it as
\begin{equation}\label{eq:mu_w*}
  \int_\Om w^\ast\,\d\mu
  \qq\qq\text{for }w\in\BV(\Om)\,,  
\end{equation}
where the $\H^{N-1}$-\ae{} defined precise representative $w^\ast$ takes the Lebesgue value in the Lebesgue points of $w$ and the average of the two jump values in its jump points. The well-posedness of the term for arbitrary $w\in\BV(\Om)$ is then equivalent with the vanishing of $|\mu|$ on all $\H^{N-1}$-negligible sets plus a mild boundedness requirement on $|\mu|$ (compare with Definition \ref{defi:mu}), and we here call such measures $\mu$ the admissible ones. In fact, the admissibility conditions are very natural and have already been around in the literature in order to characterize either the continuity of the embedding $\BV(\Om)\hookrightarrow\L^1(\Om\,;\mu)$, or the inclusion $\mu\in\BV(\Om)^\ast$, or the divergence structure $\mu=\div\,\sigma$ with some $\sigma\in\L^\infty(\Om,\R^N)$, the latter also an obvious necessary condition for solving \eqref{eq:EL}; compare \cite{MeyZie77,PhuTor08,PhuTor17} for full-space results and \cite{Schmidt25,ComLeo25,FS} for results on domains $\Om$. Recently, however, it has been observed that, for general admissible $\mu$, one cannot expect lower semicontinuity and existence results when using \eqref{eq:mu_w*}, but rather the results of \cite{Schmidt25,LeoCom24_v5,FS} underpin that the proper understanding of the measure term for matters of variational existence theory is
\begin{equation}\label{eq:mu_wmp}
  \llangle\mu_\pm\,;w^\mp\rrangle
  \coleq\int_\Om w^-\,\d\mu_+-\int_\Om w^+\,\d\mu_-
  \qq\qq\text{for }w\in\BV(\Om)\,.
\end{equation}
Here, $\mu=\mu_+{-}\mu_-$ (with $\mu_\pm\ge0$ singular to each other) is the Jordan decomposition of $\mu$, and $w^+$ takes the larger, $w^-$ the smaller of the two jump values in jump points of $w$, while both $w^+$ and $w^-$ still take the Lebesgue value in Lebesgue points. We remark that, specifically for $w\in\W^{1,1}(\Om)$, the choice \eqref{eq:mu_wmp} still reduces to simply $\llangle\mu_\pm\,;w^\mp\rrangle=\int_\Om w^\ast\,\d\mu$ and that we have $\llangle H_\pm\LN\,;w^\mp\rrangle=\int_\Om wH\,\d\LN$ for $w\in\BV(\Om)$ and a measure $H\LN$ with $H\in\L^1(\Om)$. Thus, for the precise convention in \eqref{eq:mu_wmp} to matter, $w\in\BV(\Om)\setminus\W^{1,1}(\Om)$ needs to meet a measure $\mu$ with non-vanishing singular part\footnote{More precisely, the truly relevant cases are only those with $|\mu|(\mathrm{J}_w)>0$ for the approximate jump set $\mathrm{J}_w$ of $w\in\BV(\Om)$.}.

\medskip

\noindent\textbf{Literature context, isoperimetric conditions, and main results.}
All in all, the preceding considerations indicate that, for admissible $\mu$ and arbitrary $w\in\BV(\Om)$, the reasonable extension of the functional in \eqref{eq:formal_functional} is given by
\begin{equation}\label{eq:MF_intro}
  \MF_{u_0}^\mu[w]
  \coleq\int_\Om f(\,.\,,\D w)
  +\int_{\partial\Om}\finf(\,.\,,(w{-}u_0)\nu_\Om)\,\d\H^{N-1}
  +\llangle\mu_\pm\,;w^\mp\rrangle\,.
\end{equation}
This functional $\MF_{u_0}^\mu$ is the concern of our main results, and indeed these results essentially parallel and generalize the ones obtained by Leonardi--Comi \cite{LeoCom24_v5} and the first author \cite{Schmidt25} for the case of prescribed-mean-curvature measures (i.\@e.\@ for $f(x,\xi)=\sqrt{1+|\xi|^2}$ and a parametric counterpart of the corresponding functional) as well as the ones of our predecessor paper \cite{FS} for the case of a possibly anisotropic total variation with measure (i.\@e.\@ for $f(x,\xi)$ additionally homogeneous of degree $1$ in $\xi$, in other words for $\finf=f$); compare also \cite{Ziemer95,MerSegTro08,SchSch16,DaiTruWan12,DaiWanZho15} for previous related solution theory of the equation \eqref{eq:EL} with right-hand-side measure in these specific cases. For general functionals with measure of type \eqref{eq:formal_functional}, to the state of our knowledge the sole results available are those of Carriero--Leaci--Pascali \cite{CarLeaPas85,CarLeaPas86,CarLeaPas87} and Pallara \cite{Pallara91}, who in certain $p$-coercive cases with $p\ge1$ establish lower semicontinuity restricted to $\W^{1,p}(\Om)$ (or subspaces thereof). Hence, in these works the convention \eqref{eq:mu_w*} suffices and the finer choice in \eqref{eq:mu_wmp} is irrelevant. However, in our case with $p=1$, these results come with the decisive drawback that they remain restricted to $\W^{1,1}(\Om)$, which is not an adequate space for deducing any existence results.

A crucial ingredient in our theory are (linear) isoperimetric conditions (ICs),
which are more quantitative versions of the admissibility conditions for the
measure $\mu$. These conditions depend on the integrand $f$ through $f^\infty$
only and essentially take the form
\begin{equation}\label{eq:ICs_intro}
  {-}C\int_{\partial A}\finf(\,.\,,\nu_A)\,\d\H^{N-1}
  \le\mu(A)
  \le C\int_{\partial A}\finf(\,.\,,{-}\nu_A)\,\d\H^{N-1}
  \qq\qq\text{for all smooth }A\Subset\Om
\end{equation}
with inward normal $\nu_A$ and with a constant $C\in[0,\infty)$. Conditions of
similar nature previously occurred in the calculus of variations in
\cite{BomGiu73,Miranda74,Giaquinta74a,Giaquinta74b,Gerhardt74,Steffen76a,Steffen76b,Giusti78a,Duzaar93,Ziemer95,DuzSte96,DaiTruWan12,DaiWanZho15,Schmidt25,LeoCom24_v5,FS},
for instance, and despite an inherent technical flavor are nowadays
well-understood. The prototype examples captured are (sums of)
$(N{-}\gamma)$-dimensional measures $\theta\H^{N-\gamma}\resmes S$ with
$\gamma\in[0,1]$, a Borel set $S$, and a density
$\theta\in\L^1(S\,;\H^{N-\gamma})$. Given such a measure $\mu$ the ICs can often
be grasped and verified by geometrically interpreting the integral terms in
\eqref{eq:ICs_intro} as anisotropic perimeters of $A$ (which in case
$f^\infty(x,\xi)=|\xi|$ reduce to the standard isotropic perimeter). Alternatively,
one can sometimes establish \eqref{eq:ICs_intro} in an elegant manner by finding
a calibration $\sigma\in\L^\infty(\Om,\R^N)$ with
$\|(\finf)^\circ(\,.\,,\sigma)\|_{\L^\infty(\Om)}\le C$ (in the standard
case reduced to $\|\sigma\|_{\L^\infty(\Om,\R^N)}\le C$) such that $\div(\sigma)=\mu$
holds in the sense of distributions; this calibration criterion is, in fact,
necessary and sufficient and has been established in \cite[Theorems 4.2,
4.6]{FS} in full generality. Furthermore, a simple class of concrete examples in
the most relevant borderline case $\gamma=1$ is provided by
\cite[Proposition 8.1]{Schmidt25} where it is recorded that the ICs
\eqref{eq:ICs_intro} with $f^\infty(x,\xi)=|\xi|$ are valid for
$C\H^{N-1}\resmes\partial^{(\ast)}\!K$ whenever $K$ is a bounded
(pseudo)convex set in $\R^N$; compare the discussion following
\cite[Definition 3.1]{FS} for the extension to general anisotropies. For a more
thorough treatment of ICs in similar framing and the more extensive discussion
of examples, we generally refer to \cite[Sections 7, 8]{Schmidt25} and
\cite[Section 4]{ComLeo25} in case $f^\infty(x,\xi)=|\xi|$ and to \cite[Sections
3, 4, 5]{FS} in general. Here, we limit ourselves to additionally recording, as
a central indication of significance, the necessity of \eqref{eq:ICs_intro} for
solving the Euler-Lagrange equation \eqref{eq:EL}. Indeed, whenever $u$ is a
smooth solution of \eqref{eq:EL} such that $\nabla_\xi f(\,.\,,\nabla
u)\in\C^0(\Omega)$, the divergence theorem and the convexity-based inequality
${\pm}\nabla_\xi f(x,\xi)\ip\nu\le\finf(x,{\pm}\nu)$ for $x\in\Om$ and
$\xi,\nu\in\R^N$ imply
\[
  {\mp}\mu(A)
  ={\pm}\int_{\partial A}\nabla_\xi f(\,.\,,\nabla u)\ip\nu_A\,\d\H^{N-1}
  \le\int_{\partial A}\finf(\,.\,,{\pm}\nu_A )\,\d\H^{N-1}
  \qq\qq\text{for all smooth }A\Subset\Om
\]
and thus confirm the validity of \eqref{eq:ICs_intro} with the specific constant
$C=1$.

The first and tentatively most decisive of our general results (Theorem \ref{thm:LSC}) asserts that reversely the ICs \eqref{eq:ICs_intro} with $C=1$ imply the lower semicontinuity of $\MF_{u_0}^\mu$ on $\BV(\Om)$ with respect to $\L^1(\Om)$-convergence. This is interesting, since the measure term $\llangle\mu_\pm\,;w^\mp\rrangle$ alone is \emph{not} lower semicontinuous with respect to this convergence, and thus the result involves IC-governed compensation effects between the different terms of $\MF_{u_0}^\mu$. In fact, the result even applies for measures $\mu_\pm$ which do not arise by Jordan decomposition of a signed measure $\mu$, but we defer the treatment of this more incidental generalization to the later sections. Moreover, if the ICs \eqref{eq:ICs_intro} hold even with $C<1$, then it is comparably straightforward to check also coercivity of $\MF_{u_0}^\mu$ and deduce our second main result (Theorem \ref{thm:exist_inhom}) on the existence of at least one minimizer of $\MF_{u_0}^\mu$. Correspondingly, we demonstrate with the later Example \ref{ex:anis_area} that coercivity and existence indeed do not extend to the borderline case $C=1$. Finally, we complement the semicontinuity result with a construction of recovery sequences (Theorem \ref{thm:rec_seq}), which does \emph{not} require a quantitative IC assumption like \eqref{eq:ICs_intro}. Anyhow, \emph{if} \eqref{eq:ICs_intro} with $C=1$ is at hand, then the combination of our results also identifies $\MF_{u_0}^\mu$ as the natural extension by semicontinuity from $\W^{1,1}_{u_0}(\Om)\coleq u_0+\W^{1,1}_0(\Om)$ to $\BV(\Om)$ (Corollary \ref{cor:rel_func}). We believe that the results just described can be naturally complemented with duality considerations, which are partially similar to those in \cite{Anz86,Bildhauer03,BecSch15,SchSch18,LeoCom24_v5} and which in particular yield a weak form of the Euler-Lagrange equation \eqref{eq:EL} for the $\BV$ minimizers of our theory. However, we leave details of such considerations for further work.

\medskip

\noindent\textbf{Assumptions on the integrand and exemplary cases.}
For the integrand $f$, in addition to the standard hypotheses already mentioned (that is, the growth condition \eqref{eq:lin_growth_intro} plus convexity in $\xi$ and continuity in $(x,\xi)$ of both $f(x,\xi)$ and $\finf(x,\xi)$; cf.\@ Assumption \ref{assum:f}) our approach requires imposing the mild extra requirement
\begin{equation}\label{assum:H4_intro}
  f(x,\xi)\ge\finf(x,\xi)-M
  \qq\qq\text{for all }(x,\xi)\in\Om\times\R^N
\end{equation}
with a constant $M\in\R$. In fact, we tend to believe that \eqref{assum:H4_intro} is \emph{not} truly necessary for any of our results, but also that with our taken approach its elimination would require quite some extra effort (see Remarks \ref{rem:H4_for_LSC} and \ref{rem:H4_for_coercivity} for finer discussion). Additionally, being \eqref{assum:H4_intro} satisfied in most relevant model cases and examples, as discussed next, we conclude that this assumption seems reasonable at least, and we keep it here.

In particular, all our assumptions are satisfied for the standard total variation (i.\@e.\@ $f(x,\xi)=|\xi|$) and more generally for the anisotropic total variation cases already covered in \cite{FS}. In addition, we now include the prominent model case of the area integrand
\[
  f(x,\xi)=\sqrt{1+|\xi|^2}
\]
and in this case refine some of the results in \cite{LeoCom24_v5} by treating measures $\mu$ of general form and extending semicontinuity to the borderline case $C=1$ of the ICs. Beyond that, further natural model cases covered here are Finslerian area integrands and specifically Riemannian area integrands
\[
  f(x,\xi)=\sqrt{\nu_0^2+\p(x,\xi)^2\,}
  \qq\qq\text{and}\qq\qq
  f(x,\xi)=\sqrt{\nu_0^2+g_x(\xi,\xi)}
\]
with $\nu_0\in(0,\infty)$, a $\C^0$ Finsler metric $\p\in\C^0(\ol\Om\times\R^N)$, and a $\C^0$ Riemannian metric $g$ on $\ol\Omega$. Some further $x$-independent examples included in our treatment are 
\[
  f(x,\xi)=(1+|\xi|^p)^{1/p}\,,\qq\qq
  f(x,\xi)=\begin{cases}\frac1p|\xi|^p&\text{if }|\xi|\le1\\|\xi|-\frac{p{-}1}p&\text{if }|\xi|\ge1\end{cases}\,,\qq\qq
  f(x,\xi)=|\xi|\arctan|\xi|
\]
with parameter $p\in(1,\infty)$ (where the latter two integrands need $M\ge\frac{p-1}p$ and $M\ge1$ in \eqref{assum:H4_intro}, respectively). Clearly, one may think of similar integrands with additional $x$-dependent coefficients as well. However, there exist also negative examples which fulfill all other assumptions relevant here, but are ruled out by failure of \eqref{assum:H4_intro} only. Two among such cases closely related to each other are
\[
  f(x,\xi)=1+|\xi|-{\big(1+|\xi|\big)}^\theta
  \qq\qq\text{and}\qq\qq
  f(x,\xi)=\big(|\xi|-|\xi|^\theta\big)_++1
\]
with parameter $\theta\in(0,1)$.

\medskip

\noindent\textbf{On the main semicontinuity proof.}
At this stage let us briefly comment on the underlying idea of the main semicontinuity proof in this paper. Indeed, we will rewrite the general functional $\MF_{u_0}^\mu$ --- up to harmless remainder terms --- as an anisotropic total variation ($\TV$) functional with measure and thus will essentially reduce to a case covered by the framework of \cite{FS}. The approach, by which we achieve this reduction, is passing from competitors $w\in\BV(\Om)$ to new competitors $\wdi\in\BV(\Omd)$ on $\Omd\coleq(0,1)\times\Om\subseteq\R^{N+1}$, defined by
\begin{equation}\label{eq:wdi_intro}
  \wdi(x_0,x)\coleq x_0+w(x)
  \qq\qq\text{for }(x_0,x)\in\Omd
\end{equation}
with an extra variable $x_0\in(0,1)$. With these choices, for arbitrary $w\in\BV(\Om)$, we will additively split
\[
  \MF_{u_0}^\mu[w]=\widehat\Phi[\wdi]+\mathcal{R}[w]\,,
\]
where the remainder $\mathcal{R}$ is a harmless zero-order functional and is even continuous with respect to $\L^1(\Om)$-convergence. The decisive terms instead go into $\widehat\Phi$, an anisotropic $\TV$ functional with measure, defined on $W\in\BV(\Omd)$ by
\[
  \widehat\Phi[W]=\int_\Omd\p\bigg(.\,,\frac{\d\D W}{\d|\D W|}\bigg)\,\d|\D W|+\int_{\partial\Omd}\p(\,.\,,(W{-}{u_0}_\lozenge)\nu_\Omd)\,\d\H^N+\llangle\mathcal{L}^1\otimes\mu_\pm\,;W^\mp\rrangle\,,
\]
where the new integrand $\p\colon\ol{\Omd}\times\R^{N+1}\to[0,\infty)$ is a sort of 1-homogeneous extension of $f$ (closely related to what is occasionally called perspective function). Specifically, the extension of $f(x,\xi)=\sqrt{1+|\xi|^2}$ in the previous sense is $\p((x_0,x),\Xi)=|\Xi|$, that is, the area case (in $N$ variables) is reduced by our strategy to the standard total variation case (in $N{+}1$ variables). In any case, with the rewriting at hand and after some further considerations (in particular on carrying over the ICs to the product measure $\mathcal{L}^1\otimes\mu$), in the end we will deduce lower semicontinuity of $\MF_{u_0}^\mu$ from lower semicontinuity of $\widehat\Phi$ guaranteed by our previous result in \cite[Theorem 3.5]{FS}.

Finally, we find it worth pointing out that $\wdi$ from \eqref{eq:wdi_intro} has the derivative $\D\wdi=\mathcal{L}^1\otimes(\LN,\D w)\in\RM(\Omd,\R^{N+1})$, where after \cite{GMS79} the usage of the second factor $(\LN,\D w)\in\RM(\Om,\R^{N+1})$ has become quite standard in the analysis of variational problems on $\BV(\Om)$. Still, we believe that our additional small trick of recasting the functional via the extra variable $x_0$ and the formula \eqref{eq:wdi_intro} on the level of the functions themselves has not yet been around in the literature and might eventually find further applications.

\medskip

\noindent\textbf{Remarks on the vector-valued case.} Before closing this
introduction we add some comments on reasons for restricting the considerations
of this paper to functionals which depend on \emph{scalar} functions. Indeed,
this limitation is inherent to our semicontinuity proof inasmuch as we build on
the predecessor results of \cite{FS} and these results in turn depend on the
Fleming-Rishel coarea formula for \emph{scalar} $\BV$ functions. Still, at first
glance it seems plausible that a modified approach may yield a semicontinuity
result analogous to our central Theorem \ref{thm:LSC} also for \emph{vectorial}
functionals
\begin{equation}\label{eq:vectorial}
  \MF_{u_0}^\mu[w]
  \coleq\int_\Om f(\,.\,,\D w)
  +\int_{\partial\Om}\finf(\,.\,,(w{-}u_0){\otimes}\nu_\Om)\,\d\H^{N-1}
  +\int_\Om w^\mu\ip\d\mu
  \qq\text{for }w\in\BV(\Om,\R^d)
\end{equation}
with arbitrary $d\ge2$, with the first term understood as before, with a finite
$\R^d$-valued Radon measure $\mu$ on $\Om$, with a suitable $\mu$-dependent
representative $w^\mu$ of $w$ which is (at least) $|\mu|$-\ae{} defined, and
with $\otimes$ and $\ip$ denoting the dyadic product and the inner product,
respectively. However, we will now show that lower semicontinuity of
\eqref{eq:vectorial} with respect to $\L^1(\Om,\R^2)$-convergence fails
already in the basic model case where the first-order term is the parametric or
non-parametric length of functions $w\in\BV({({-}1,1)},\R^2)$ with boundary
datum $u_0({\pm}1)=\big(\begin{smallmatrix}{\pm}1\\0\end{smallmatrix}\big)$
and the measure is the $\R^2$-valued measure
$\mu\coleq\big(\begin{smallmatrix}0\\{-}2\theta\end{smallmatrix}\big)\delta_0$
on $({-}1,1)$ with the Dirac measure $\delta_0$ at $0$ and arbitrary
$\theta\in(0,1)$. In other words, for either $f(\xi)\coleq|\xi|$ or
$f(\xi)\coleq\sqrt{1+|\xi|^2}$, we consider the functional
\[
  \MF_{u_0}^\mu[w]
  =\int_{-1}^1f(\D w)
  +\big|w(1){-}\big(\begin{smallmatrix}1\\0\end{smallmatrix}\big)\big|
  +\big|w({-}1){-}\big(\begin{smallmatrix}{-}1\\0\end{smallmatrix}\big)\big|
  -2\theta w^\mu_2(0)
  \qq\text{for }w\in\BV({({-}1,1)},\R^2)\,.
\]
Now, for fixed $\eps>0$, the sequence $(u_k)_k$ in $\BV({({-}1,1)},\R^2)$
given by $u_k(x)\coleq\big(\begin{smallmatrix}\mathrm{sgn}(x)\\0\end{smallmatrix}\big)$
for $|x|>\frac1k$ and
$u_k(x)\coleq\big(\begin{smallmatrix}0\\\eps\end{smallmatrix}\big)$ for
$|x|<\frac1k$ converges in $\L^1(\Om,\R^2)$ to $u\in\BV({({-}1,1)},\R^2)$ with
$u(x)\coleq\big(\begin{smallmatrix}\mathrm{sgn}(x)\\0\end{smallmatrix}\big)$. Moreover, whenever the choice of $w^\mu$ at least preserves
constancy intervals of the components,
i.e. if, for $i\in\{1,2\}$, constancy of $w_i$ near $x_0$ guarantees $w^\mu_i(x_0)=w_i(x_0)$, then lower semicontinuity of
$\MF_{u_0}^\mu$ fails along this specific sequence as soon as one takes
$\eps<\frac{2\theta}{1-\theta^2}$. Indeed, in case $f(\xi)=|\xi|$ an explicit
computation gives
\[
  \MF_{u_0}^\mu[u_k]=2\big(\sqrt{1{+}\eps^2}-\theta\eps\big)<2=\MF_{u_0}^\mu[u]
  \qq\text{for all }k\in\N\,,
\]
and in case $f(\xi)=\sqrt{1+|\xi|^2}$ one gets the same with merely a constant
$2$ added, that is,
\[
  \MF_{u_0}^\mu[u_k]=2\big(1+\sqrt{1{+}\eps^2}-\theta\eps\big)<4=\MF_{u_0}^\mu[u]
  \qq\text{for all }k\in\N\,.
\]
Since $\theta\in(0,1)$ is arbitrary, this (counter)example includes measures
$\mu=\big(\begin{smallmatrix}0\\{-}2\theta\end{smallmatrix}\big)\delta_0$
whose components satisfy ICs of type \eqref{eq:ICs_intro} with arbitrarily small
constants $C$. In conclusion, lower semicontinuity in the vectorial case can only
hold in severely restricted situations. For instance, it trivially holds in case
that the integrand splits in the sense of
$f(x,\xi)=\sum_{i=1}^df_i(x,\xi_i)$ (where $\xi_i\in\R^N$
denotes the $i$-th row of $\xi\in\R^{d\times N}$) and consequently also the
functional
$\MF_{u_0}^\mu[w]=\sum_{i=1}^d(\MF_i)_{(u_0)_i}^{\mu_i}[w_i]$
splits additively into functionals of the single components. Moreover, as our
(counter)example pertains to the borderline case of $(N{-}1)$-dimensional
$\R^d$-valued measures $\mu$ on $\R^N$, it leaves open the question if lower
semicontinuity is or is not preserved for a variant of \eqref{eq:vectorial} in
case of general integrands $f$ and suitable $(N{-}\gamma)$-dimensional
$\R^d$-valued measures $\mu$ on $\R^N$ with $\gamma\in(0,1)$. Anyway, we do not
further pursue this last-mentioned issue in the present paper.

Finally, in contrast with the semicontinuity situation we tend to believe that
our result on recovery sequences (Theorem \ref{thm:rec_seq}) carries
over to the vectorial case of \eqref{eq:vectorial} in a reasonably general
framework. Indeed, as soon as one can lift a precise approximation result
analogous to \cite[Proposition 4.4]{FS} or the more general
\cite[Theorem 3.2]{ComLeo25} to the vectorial case with certain choices of
representative $w^\mu$, it seems likely that also our construction of recovery
sequences will carry over with a couple of adaptations. Furthermore, if one is
willing to put aside the interesting borderline case of $(N{-}1)$-dimensional
measures and rather focuses on measures $\mu$ with bounded
$(N{-}\gamma)$-dimensional density ratio
$\sup_{x\in\R^N,r>0}\big(r^{\gamma-N}\mu(\B_r(x)\cap\Om)\big)<\infty$ for some
$\gamma\in{[0,1)}$, even a more direct route is possible: In that case one
falls back to the term $\int_\Om w^\ast\!\ip\d\mu$ defined with the standard
$|\mu|$-\ae{} representative $w^\ast$ and can deduce existence of recovery sequences for
$\MF_{u_0}^\mu$ --- both in the scalar and the $\R^d$-valued case --- by the
well-known strict approximation result \cite[Lemma B.2]{Bildhauer03}
in combination with the refined strict continuity of such measure terms from
\cite[Theorem 6.10]{GmeRaiVSc21}.

\medskip

\noindent\textbf{Plan of this paper.} In the subsequent Section \ref{sec:prelim} we collect various preliminaries, while the full statements of our main results are given and are contextualized in Section \ref{sec:results}. In Section \ref{sec:LSC_exist} we carry out the semicontinuity proof (by the strategy described) and eventually deduce coercivity and existence. The counterexample to existence in the borderline case $C=1$ is addressed in Section \ref{sec:counterexample}, while the short Section \ref{sec:rec_seq} implements a construction of recovery sequences in close analogy with \cite{FS}. Finally, Appendix \ref{asec:sections_ICs} collects add-on observations on codimension-one slicing of $\BV$ functions, relevant only for very general cases of our theory.

\section{Preliminaries} \label{sec:prelim}

\subsection[General notation, measures, \texorpdfstring{$\BV$}{BV} functions, and sets of finite perimeter]{\boldmath General notation, measures, $\BV$ functions, and sets of finite perimeter}

We work in Euclidean space $\R^N$ with $N \in \N$, where $|x|$ and $x\!\ip\!y$ stand for Euclidean norm and inner product of vectors $x,y\in\R^N$, respectively. For measurable sets $A\subseteq\R^N$, we denote by $|A| \coleq \mathcal{L}^N(A)$ the $N$-dimensional Lebesgue measure of $A$. We write $\B_r(x)$ for the open ball in $\R^N$ centered in $x \in \R^N$ and with radius $r \in (0,\infty)$, and we abbreviate $\B_r$ in case $x=0$. The $N$-dimensional volume $|\B_1|$ of the unit ball $\B_1\subseteq\R^N$ is abbreviated as $\omega_N$. Moreover, we write $\mathcal{H}^M$ for the $M$-dimensional Hausdorff measure on $\R^N$, relevant mostly in the codimension-one case. The usual notations $\ol{A}$ and $\partial A$ are employed for the closure and boundary of $A$, respectively, and $\1_A$ is the indicator function of $A$. 

\smallskip

Given $m \in \N$ and an open set $U \subseteq \R^N$, we denote by $\RM_{(\loc)}(U,\R^m)$ the space of all (locally) finite $\R^m$-valued Radon measures on $U$. Two non-negative measures $\mu$, $\widehat{\mu}$ on $U$ are said to be \emph{mutually singular}, written $\mu \perp \widehat{\mu}$, if there exist disjoint $E, F \subseteq U$ such that $\mu(F)=\widehat{\mu}(E)=0$. 
Any signed Radon measure $\mu$ on $U$ allows for Jordan decomposition $\mu = \mu_+ - \mu_-$ with $\mu_\pm$ non-negative, mutually singular Radon measures on $U$. We call $\mu_+$ and $\mu_-$ the \emph{positive} and \emph{negative} part of $\mu$, respectively, and the non-negative Radon measure $|\mu|\coleq \mu_+ + \mu_-$ the \emph{variation measure} of $\mu$. However, on occasion (compare Theorems \ref{thm:LSC} and \ref{thm:exist_inhom} in particular) we eventually denote by $\mu_\pm$ some given non-negative measures which need not necessarily arise by Jordan decomposition and need not be mutually singular. For a non-negative Borel measure $\mu$ on $U$ and $\nu\in\RM_{\loc}(U,\R^m)$, we say that $\nu$ is \emph{absolutely continuous} with respect to $\mu$ and write $\nu \ll \mu$ if for every Borel set $B$ such that $\mu(B)=0$ we have $|\nu|(B)=0$ as well. Supposed $\mu$ is additionally $\sigma$-finite, we recall that the Radon-Nikod\'ym theorem (compare with \cite[Theorem 1.28]{AFP00}) yields the existence of a unique (up to $\mu$-negligible sets) $\R^m$-valued \emph{density} $g \in \L^1_\loc(U,\R^m;\mu)$ and of an $\R^m$-valued measure $\nu^\s$ singular to $\mu$ such that it holds $\nu=g\mu + \nu^\s$ as measures on $U$. In this situation, we also write $\frac{\d\nu}{\d \mu}$ for the density $g$.

\smallskip

For open $U \subseteq \R^N$, we introduce the space of \emph{functions of \ka locally\kz \ bounded variation} in $U$, written $\BV_{(\loc)}(U)$, as the space of all $u \in \L^1_{(\loc)}(U)$ such that the distributional gradient $\D u$ of $u$ exists as $\R^N$-valued Radon measure on $U$. We then define the \emph{total variation} of $u$ in a Borel set $B \subseteq U$ Borel as $\TV(u,B)\coleq |\D u|(B)$. The space $\BV(U)$ is a Banach space when endowed with the norm $\|u\|_{\BV(U)} \coleq \|u\|_{\L^1(U)} + |\D u|(U)$. However, norm convergence in $\BV$ is often too restrictive, and we will rather work with \emph{strict convergence} of a sequence $(u_k)_k$ in $\BV(U)$ to a limit $u$ in $\BV(U)$, which by definition means nothing but $u_k\to u$ in $\L^1(U)$ together with $|\D u_k|(U)\to|\D u|(U)$. Even more useful will be a variant, the \emph{area-strict convergence} of $(u_k)_k$ to $u$ in $\BV(U)$, which is reasonable in case $|U|<\infty$ and then requires $u_k\to u$ in $\L^1(U)$ together with $\int_U\sqrt{1+|\D u_k|^2}\to\int_U\sqrt{1+|\D u|^2}$ (where $\int_U\sqrt{1+|\D u|^2}$ may be equivalently defined either as a the total variation $|(\LN,\D u)|(U)$ of $(\LN,\D u)\in\RM(U,\R^{N+1})$ or as a functional of measures in the sense of Section \ref{subsec:func_meas}).
Moreover, if $E \subseteq \R^N$ is such that $\1_E \in \BV_{\loc}(U)$ holds for some open $U \subseteq \R^N$, we call $E$ a \emph{set of locally finite perimeter} in $U$, and its \emph{perimeter} in a Borel set $B \subseteq U$ is given by $\P(E,B) \coleq |\D \1_E|(B)\in[0,\infty)$. In this situation, the Radon-Nikod\'ym density $\nu_E\coleq\frac{\d\D\1_E}{\d|\D\1_E|}$ is defined, and is a unit vector field $|\D\1_E|$-\ae{} on $U$. In fact, De Giorgi's structure theorem gives $\D\1_E=\nu_E\H^{N-1}\resmes\partial^\ast\!E$ as measures in $U$ with the reduced boundary $\partial^\ast\!E$ of $E$ (compare \cite[Definition 3.54, Theorem 3.59]{AFP00}), and this allows understanding $\nu_E$ as the generalized inward unit normal of $E$ at $U\cap\partial^\ast\!E$.
Finally, a set of locally finite perimeter such that $\P(E,U)<\infty$ is called a \emph{set of finite perimeter} in $U$, and in case $U=\R^N$ we abbreviate $\P(E)\coleq\P(E,\R^N)$.

\smallskip

Given a measurable set $A \subseteq \mathbb{R}^N$ and $\theta \in {[0,1]}$, we write $A^\theta$ for the set of density-$\theta$ points for $A$, that is, 
\[
A^{\theta} \coleq \left\{ x \in \mathbb{R}^N\,:\, \lim_{r \to 0}  \frac{| \B_r(x) \cap A|}{|\B_r|} = \theta \right\}\,.
\]
Particularly relevant are the \emph{measure-theoretic interior} $A^1$ and the \emph{measure-theoretic closure} $A^+ \coleq (A^0)^\c$ of $A$. Moreover, for $ x \in U \subseteq \R^N$ and measurable $u \colon U \to \R$, we introduce the \emph{approximate upper limit} and the \emph{approximate lower limit} of $u$ at $x$ as
\[
u^+(x) \coleq \sup \left\{ t \in \R \colon \ x \in \left\{ u > t \right\}^+ \right\}\,, 
\qquad 
u^-(x) \coleq \sup \left\{ t \in \R \colon \ x \in \left\{ u > t \right\}^1 \right\}\,,
\]
and we define the \emph{precise representative} $u^\ast$ of $u$ by setting $u^\ast(x) \coleq (u^+(x)+u^-(x))/2$. 
If specifically we have $u\in\L^1_{\loc}(U)$, all the above-mentioned representatives coincide at Lebesgue points of $u$, and $u^\pm$ represent the jump values of $u$ at its jump points. For $u \in \BV_\loc(U)$, the Federer-Vol'pert theorem (see for example \cite[Theorem 3.78]{AFP00}) implies that $\H^{N-1}$-a.\@e.\@ point outside the jump set $\mathrm{J}_u$ of $u$ is a Lebesgue point, and thus in this case $u^+=u^-=u^\ast$ holds $\H^{N-1}$-\ae{} in $U\setminus\mathrm{J}_u$. Clearly, in case of $u \in \W^{1,1}_\loc(U)$, the coincidence holds even $\H^{N-1}$-\ae{} in $U$.

\subsection{Polars of positively 1-homogeneous integrands}
  \label{subsec:polar_aniso}

We say that $\p\colon\R^N\to\R$ is \emph{positively 1-homogeneous} if it satisfies $\p(t\xi)=t\p(\xi)$ for all $\xi\in\R^N$ and $t\in[0,\infty)$ (which in particular includes the requirement $\p(0)=0$). For every differentiability point $\xi\in\R^N$ of such $\p$, differentiation with respect to $t$ at $t=1$ gives Euler's relation $\xi\ip\nabla\p(\xi)=\p(\xi)$. Furthermore, a sort of dual of a positively 1-homogeneous function is given by the following well-known construction.

\begin{defi}[polar]
\label{defi:polar}
  Consider a positively 1-homogeneous function $\p \colon \R^N \to [0,\infty)$ such that\/ $\p(\xi)>0$ holds for all $\xi\in\R^N\setminus\{0\}$. Then, the \emph{polar function} $\p^\circ \colon \R^N \to [0,\infty)$ of $\p$ is defined by 
  \[
    \p^\circ(\xi^\ast)
    \coleq \sup_{\xi \in \R^N\setminus\{0\}} \frac{ \xi^\ast \ip \xi }{\p(\xi)}
    \qq\text{for }\xi^\ast\in\R^N\,.
  \]
\end{defi}

The polar $\p^\circ$ is always positively 1-homogeneous and convex in $\R^N$ with $\p^\circ(\xi^\ast)>0$ for all $\xi^\ast\in\R^N\setminus\{0\}$, thus it is also continuous and \ae{} differentiable with non-zero gradient in $\R^N$. Moreover, the definition of the polar implies the anisotropic Cauchy-Schwarz inequality
\begin{equation}\label{eq:aniso_CS}
  \xi^\ast \ip \xi \leq \p^\circ(\xi^\ast)\,\p(\xi)
  \qq\text{for all }\xi,\xi^\ast\in\R^N\,,
\end{equation}
where, for each $\xi^\ast\in\R^N$, there exists at least one $\xi\in\R^N\setminus\{0\}$ such that \eqref{eq:aniso_CS} becomes an equality. Beyond that we record:

\begin{lem}[equality cases in the anisotropic Cauchy-Schwarz inequality] \label{lem:rel_p_gradient}
  Consider a positively $1$-homoge\-neous function $\p\colon \R^N \to [0,\infty)$ such that $\p(\xi)>0$ holds for all $\xi\in\R^N\setminus\{0\}$. Then, for every differentiability point $\xi^\ast\in\R^N\setminus\{0\}$ of\/ $\p^\circ$, the corresponding $\xi\in\R^N\setminus\{0\}$ which achieve equality in \eqref{eq:aniso_CS} are fully characterized by the condition $\frac\xi{\p(\xi)} = \nabla \p^\circ(\xi^\ast)$. In particular,
  \[
    \p(\nabla\p^\circ(\xi^\ast))=1
  \]
  holds for every differentiability point\/ $\xi^\ast\in\R^N\setminus\{0\}$ of\/ $\p^\circ$ and thus for \ae{} $\xi^\ast\in\R^N$.
\end{lem}
 
\begin{proof}
  We fix a differentiability point $\xi^\ast\in\R^N\setminus\{0\}$ of $\p^\circ$. One one hand, if $\xi\in\R^N\setminus\{0\}$ achieves equality in \eqref{eq:aniso_CS}, then $\xi^\ast$  maximizes $\tau^\ast\mapsto\tau^\ast\!\ip\xi-\p^\circ(\tau^\ast)\,\p(\xi)$ on $\R^N$, and by the first-order criterion this implies
  $\xi=\p(\xi)\,\nabla\p^\circ(\xi^\ast)$, in other words $\xi/\p(\xi)=\nabla\p^\circ(\xi^\ast)$. On the other hand, if $\xi/\p(\xi)=\nabla\p^\circ(\xi^\ast)$ holds, we bring in Euler's relation for the homogeneous function $\p^\circ$ in order to get $\xi^\ast\!\ip\xi=\p(\xi)\,\xi^\ast\!\ip\nabla\p^\circ(\xi^\ast)=\p(\xi)\,\p^\circ(\xi^\ast)$ and thus achieve equality in \eqref{eq:aniso_CS}. Finally, since some $\xi\in\R^N\setminus\{0\}$ with the previous properties always exists (as recorded directly after \eqref{eq:aniso_CS}), by homogeneity of $\p$ we obtain
  $\p(\nabla\p^\circ(\xi^\ast))=\p(\xi/\p(\xi))=1$.
\end{proof}

\subsection{Anisotropic total variations and anisotropic perimeters}

Given a set $U\subseteq\R^N$ and arbitrary function $\p\colon U\times\R^N\to[0,\infty)$, we denote by $\widetilde\p\colon U\times\R^N\to[0,\infty)$ the \emph{mirrored integrand} which is defined by
\[
  \widetilde\p(x,\xi)\coleq\p(x,{-}\xi)
  \qq\text{for }(x,\xi)\in U\times\R^N\,.
\]
We will often consider integrands $\p$ which are positively 1-homogeneous and/or convex in $\xi$ (i.\@e.\@ $\xi\mapsto\p(x,\xi)$ has the respective property for every $x\in U$), have linear growth in $\xi$ (i.\@e.\@ $\alpha|\xi|\le\p(x,\xi)\le\beta|\xi|$ for all $(x,\xi)\in U\times\R^N$ with some $0<\alpha\le\beta<\infty$), or are continuous in $(x,\xi)$. For future usage we briefly record that all these properties trivially carry on from $\p$ to $\widetilde\p$.

\medskip
    
Now we recall the definition of the anisotropic total variations and anisotropic perimeters. 
	
\begin{defi}[anisotropic total variation and anisotropic perimeter] \label{defi:anis_TV}
  We consider an open set $U\subseteq\R^N$ and a Borel function $\p\colon U\times\R^N\to{[0,\infty)}$ which is positively 1-homogeneous in $\xi$. For $w\in\BV_\loc(U)$, the $\p$-anisotropic total variation of\/ $w$ on a Borel set $B\subseteq U$ is 
  \[
    \TV_\p(w,B)
    \coleq|\D w|_\p(B)\coleq\int_B\p(\,.\,,\nu_w)\,\d|\D w|\,,
  \]
  where $\nu_w\coleq\frac{\d\D w}{\d|\D w|}$ denotes the Radon-Nikod\'ym derivative. In particular, for a set $E\subseteq\R^N$ of locally finite perimeter in $U$, the $\p$-anisotropic perimeter of\/ $E$ in a Borel set $B\subseteq U$ is obtained as 
  \[
    \P_\p(E,B)
    \coleq|\D\1_E|_\p(B)=\int_B\p(\,.\,,\nu_E)\,\d|\D\1_E|
  \]
  with the generalized inward normal $\nu_E=\frac{\d\D\1_E}{\d|\D\1_E|}$. We abbreviate $\P_\p(E,\R^N)$ as $\P_\p(E)$. Moreover, by convention we understand $\TV_\p(w,B)=\infty$ if $w\notin\BV_\loc(U')$ for every open $U'$ such that $B \subseteq U' \subseteq \R^N$ and $\P_\p(E,B)=\infty$ if a $E$ does not have locally finite perimeter in any open $U'$ such that $B \subseteq U' \subseteq \R^N$. Finally, with slight abuse of notation, we identify a positively 1-homogeneous $\p\colon\R^N\to[0,\infty)$ with its extension $\p\colon\R^N\times\R^N\to[0,\infty)$ such that $\p(x,\xi)=\p(\xi)$ and tacitly adopt the previous notions to this understanding.
\end{defi}

In the situation of the definition, we have the easy relations $\TV_{\widetilde\p}(w,B)=\TV_{\p}({-}w,B)$ and $\P_{\widetilde\p}(E,B)=\P_\p(E^\c,B)$. Moreover, we stress that the usage of the inward normal (rather than the outward normal) in our definition of $\P_\p$ is not fully standard and causes an opposite sign or the occurrence of $\widetilde\p$ in some of our statements. Still, we prefer this convention which ensures $\TV_\p(\1_E,B)=\P_\p(E,B)$ without an additional minus sign. 
Clearly, for even $\p$, we have $\widetilde\p=\p$, then these details do not matter anyway.

\smallskip

The next statements are originally due to Reshetnyak \cite{Reshetnyak68} and may be read off from \cite[Theorems 2.38, 2.39]{AFP00}, essentially by specializing the assertions to the case of derivative measures $\D w\in\RM(U,\R^N)$. The first one extends the central semicontinuity property of the total variation to the anisotropic case.

\begin{thm}[Reshetnyak semicontinuity; homogeneous version]
\label{thm:Resh_LSC_hom_BV}
  For an open set\/ $U\subseteq\R^N$, assume that\/ $\p\colon U \times \R^N \to {[0,\infty)}$ is positively 1-homogeneous in $\xi$, convex in $\xi$, lower semicontinuous in $(x,\xi)$, and satisfying $\p(x,\xi)\ge\alpha|\xi|$ for all $(x,\xi)\in U\times\R^N$ and some $\alpha>0$. Then, whenever a sequence $(u_k)_k$ in $\BV(U)$ converges in $\L^1(U)$ to $u\in\BV(U)$, there holds
  \[
    \liminf_{k \to \infty}|\D u_k|_{\p}(U)
    \geq|\D u|_{\p}(U)\,.
  \]
\end{thm}

The subsequent companion result concerns the more restricted class of sequences which converge \emph{strictly} in $\BV(U)$. Along such sequences it grants continuity without any convexity assumption on the integrand.

\begin{thm}[Reshetnyak continuity; homogeneous version] 
    \label{thm:Resh_cont_hom_BV}
  For an open set\/ $U\subseteq\R^N$, assume that\/ $\p\colon U \times \R^N \to {[0,\infty)}$ is positively 1-homogeneous in $\xi$ and continuous in $(x,\xi)$ and satisfies $\p(x,\xi)\leq\beta|\xi|$ for all $(x,\xi)\in U\times\R^N$ and some $\beta\in[0,\infty)$. Then, whenever a sequence $(u_k)_k$ in $\BV(U)$ converges strictly in $\BV(U)$ to $u\in\BV(U)$, there holds
  \[
    \lim_{k \to \infty}|\D u_k|_{\p}(U)
    = |\D u|_{\p}(U)\,.
  \]
\end{thm}

\smallskip

Finally, we put on record the anisotropic isoperimetric inequality, originally established in \cite[Section 1]{Busemann49} and converted into our framework in \cite[Proposition 2.3]{AFTL97}.

\begin{thm}[$\p$-anisotropic isoperimetric inequality]  \label{thm:anis_isop_ineq}
  Fix a positively 1-homogeneous and convex function $\p \colon \R^N \to [0,\infty)$ such that\/ $\p(\xi)>0$ holds for all\/ $\xi \in \R^N\setminus\{0\}$. Then, for measurable $A \subseteq \R^N$ and $r\in[0,\infty)$ such that\/ $|A|=|\{\p^\circ<r\}|<\infty$, one has
  \[
    \P_{\widetilde\p}(A) \geq \P_{\widetilde\p}(\{\p^\circ<r\})\,.
  \]
\end{thm}

\subsection{Linear-growth integrands, recession and perspective function} \label{subsec:integrands}

For this section, we fix an open set $U\subseteq\R^N$.

\begin{defi}\label{defi:f-finf-perspective}
  A function $f \colon U \times \R^N \to \R$ has \textit{linear growth} \ka to be understood in the second variable\kz \ if 
  \begin{equation*} 
    \alpha |\xi| \leq f(x,\xi) \leq \beta(|\xi|+1)
    \qq\text{for all } (x,\xi) \in U \times \R^N
  \end{equation*}
  holds with constants $\alpha, \beta \in [0,\infty)$. If $f$ is moreover convex in its second variable $\xi\in\R^N$, we introduce the \textit{recession function} $\finf \colon U \times \R^N \to [0,\infty)$ of $f$ by setting 
  \[
    \finf(x,\xi)\coleq\lim_{t \to 0^+} t f\left(x,\frac{\xi}{t} \right)
    =\lim_{s \to \infty} \frac{f(x,s\xi)}{s}
    \qq\text{for all } x \in U\,, \ \xi \in \R^N\,.
  \]
  The \textit{perspective} \textup{(}or \textit{homogenized}\textup{)} \textit{function} of\/ $f$ is $\fpers \colon  U \times [0,\infty) \times \R^N   \to [0,\infty)$ given by
  \[ 
    \fpers(x,t,\xi)\coleq
    \begin{cases}
      t f\left(x,\frac{\xi}{t}\right) &\text{for }t > 0\\
      \finf(x,\xi) &\text{for }t = 0
    \end{cases}\,.
  \]
  By definition, $t \mapsto \fpers(x,t,\xi)$ is continuous at\/ $t=0$, and there hold\/ $ \fpers(x,1,\xi)= f(x,\xi)$ and\/ $\fpers(x,0,\xi)=\finf(x,\xi)$ for all $(x,\xi) \in U \times \R^N$. 
\end{defi}

We now state some relevant properties of the recession and perspective function.

\begin{lem}\label{lem:properties_rec_persp}
  If $f\colon U \times \R^N \to \R$ has linear growth and is convex in $\xi \in \R^N$, then we have\textup{:}
  \begin{enumerate}[{\rm(i)}]
  \item\label{item:properties_rec_persp_i}The recession function $\finf$ is well-defined, convex in $\xi$, positively 1-homogeneous in $\xi$, and it satisfies
  \begin{equation}\label{eq:lin_growth_finf}
    \alpha |\xi| \leq \finf(x,\xi) \leq \beta |\xi|
    \qq\text{for all\/ }(x,\xi) \in U \times \R^N\,.	
  \end{equation}
  Similarly, the perspective function $\fpers$ is convex in $(t,\xi)$, positively 1-homogeneous in $(t,\xi)$, and it satisfies
  \[
    \alpha |\xi| \leq \fpers(x,t,\xi) \leq \beta (t+|\xi|)
    \qq\text{for all\/ } (x,t,\xi) \in U \times [0,\infty) \times \R^N\,.
  \]
  \item\label{item:properties_rec_persp_ii}If $f$ itself is positively 1-homogeneous in $\xi$, then it holds 
  \[
    \fpers(x,t,\xi)=\finf(x,\xi)=f(x,\xi)
    \qq\text{for all\/ } (x,t, \xi) \in U \times [0,\infty) \times \R^N\,.
  \]
  \item\label{item:properties_rec_persp_iii}If $f$ is lower semicontinuous, then $\finf$ and $\fpers$ are lower semicontinuous as well. 
  \item\label{item:properties_rec_persp_iv}If both $f$ and\/ $\finf$ are continuous, then $\fpers$ is continuous. 
  \end{enumerate}
\end{lem}

We stress that in the $x$-independent case $f(x,\xi)=f(\xi)$ the convexity of $f$, $\finf$, $\fpers$ already implies their continuity. Therefore, parts \eqref{item:properties_rec_persp_iii} and \eqref{item:properties_rec_persp_iv} of Lemma \ref{lem:properties_rec_persp} are relevant only in coping with $x$-dependent cases.

\begin{proof}[Proof of Lemma \ref{lem:properties_rec_persp}]
We start with part \eqref{item:properties_rec_persp_i}. For any fixed $x \in U$ and $\xi \in \R^N$, the convexity of $f$ in $\xi$ implies that
\[
  g_{x,\xi}(s) \coleq \frac{f(x,s \xi)-f(x,0)}{s}
\]
is non-decreasing in $s\in(0,\infty)$.  This implies existence of $\finf(x,\xi)=\lim_{s\to\infty}g_{x,\xi}(s)$ and validity of \eqref{eq:lin_growth_finf}. In order to check convexity of $\fpers$ in $(t,\xi)$, we fix $x \in U$, $\xi_1, \xi_2 \in \R^N$, $t_1, t_2>0$, $\lambda \in (0,1)$, and we abbreviate $\ol{t}\coleq\lambda t_1+(1{-}\lambda)t_2$ and $\ol{\xi}\coleq\lambda\xi_1+(1{-}\lambda)\xi_2$. By convexity of $f$ in $\xi$ we find
\[
  \lambda\fpers(x,t_1,\xi_1)+(1{-}\lambda)\fpers(x,t_2,\xi_2) 
  =\ol{t}\left(\frac{\lambda t_1}{\ol{t}}f\left(x,\frac{\xi_1}{t_1}\right)
    +\frac{(1{-}\lambda) t_2 }{\ol{t}} f\left(x,\frac{\xi_2}{t_2}\right)\right)
  \ge\ol{t}f\left(x,\frac{\ol{\xi}}{\ol{t}}\right)
  =\fpers(x,\ol{t},\ol{\xi})\,,
\]
which confirms convexity of $\fpers$ in $(t,\xi)\in(0,\infty)\times\R^N$. Cases with $t=0$ and the convexity claim for $\finf$ are then reached by taking limits $t\to0^+$, and the remaining claims in part \eqref{item:properties_rec_persp_i} follow straightforwardly.

\medskip

The assertions in part \eqref{item:properties_rec_persp_ii} are direct consequences of the definition.

\medskip

In order to establish part \eqref{item:properties_rec_persp_iii}, we consider a converging sequence $(x_k, t_k, \xi_k) \to (x,t,\xi)$ in $U \times [0,\infty) \times \R^N$. In case $t>0$, lower semicontinuity of $\fpers$ along the sequence is immediate by definition of $\fpers$ and lower semicontinuity of $f$. In case $t=0$ instead, taking into account 
$\finf(x,\xi)=\lim_{s\to\infty}\left(\fpers\big(x,\frac1s,\xi\big)-\frac\beta s\right)$, for any given $\eps>0$, there exists a corresponding $s\in(0,\infty)$ such that $\finf(x,\xi)\le\eps+\fpers\big(x,\frac 1s,\xi\big)-\frac\beta s$. Bringing in lower semicontinuity of $\fpers$ in the case already treated and bounding ${-}\frac\beta s\le{-}\fpers\big(x_k,\frac1s,0\big)$, we infer
\[
  \fpers(x,0,\xi)
  = \finf(x,\xi) 
  \leq\eps + \liminf_{k \to \infty} \left( \fpers\left(x_k,\frac1s,\xi_k\right) - \fpers\left(x_k,\frac1s,0\right) \right)\,.
\]
To proceed further, we deduce from $\fpers(x_k,t,\xi_k)-\fpers(x_k,t,0)=g_{x_k,\xi_k}\big(\frac1t\big)$ for $t>0$, the monotonicity of $g_{x_k,\xi_k}$ observed earlier, and again limits $t\to0^+$ that $\fpers(x_k,t,\xi_k)-\fpers(x_k,t,0)$ is non-increasing even in $t\in[0,\infty)$. Then, since we have $t_k\le\frac1s$ for $k\gg1$, we conclude
\begin{align*}
	\fpers(x,0,\xi)
    &\le\eps + \liminf_{k \to \infty} \left( \fpers\left(x_k,t_k,\xi_k\right) - \fpers\left(x_k,t_k,0\right) \right)
    = \eps +  \liminf_{k \to \infty}\fpers(x_k,t_k,\xi_k)\,.
\end{align*}
Since $\eps>0$ is arbitrary, this proves lower semicontinuity of $\fpers$ also at points $(x,0,\xi)$. In view of $\finf(x,\xi)=\fpers(x,0,\xi)$, lower semicontinuity of $\finf$ follows as well.

\medskip

Finally, we check part \eqref{item:properties_rec_persp_iv}. In the light of part \eqref{item:properties_rec_persp_iii} it remains to prove $\limsup_{k \to \infty} \fpers(x_k,t_k,\xi_k) \le \fpers(x,t,\xi)$ whenever $(x_k,t_k,\xi_k)\to(x,t,\xi)$ in $U \times [0,\infty) \times \R^N$. By continuity of $f$ this is immediate for $t>0$. In case $t=0$ instead, the continuity assumption for $\finf$ gives
\begin{equation} \label{eq:iii_auxiliary}
  \fpers(x,0,\xi)
  = \finf(x,\xi) 
  = \lim_{k\to\infty} \finf(x_k,\xi_k)
  = \lim_{k\to\infty} \left(\fpers(x_k,0,\xi_k)-\fpers(x_k,0,0)\right)\,.
\end{equation}
Then, via the equality \eqref{eq:iii_auxiliary}, via the monotonicity property exploited already in the proof of part \eqref{item:properties_rec_persp_iii}, and via the $x$-uniform bounds $0\le\fpers(x,t,0)\le\beta t$ we arrive at
\[
  \fpers(x,0,\xi)
  \ge \limsup_{k\to\infty} \left(\fpers(x_k,t_k,\xi_k)-\fpers(x_k,t_k,0)\right)
  = \limsup_{k\to\infty} \fpers(x_k,t_k,\xi_k)\,.
\]
This completes the proof of part \eqref{item:properties_rec_persp_iv} and of the lemma.
\end{proof}

We point out that any positively 1-homogeneous and convex function $g: \R^k \to [0,\infty)$ satisfies the triangle inequalities
\begin{equation} \label{anis_eq:subadd_phi}
  g(\xi+\tau) \leq g(\xi) +g(\tau)
  \qq\text{and}\qq
  g(\xi) - g(\tau) \leq g(\xi-\tau)
  \qq\text{for all } \xi, \tau \in \R^k\,.
\end{equation}
Specifically, in view of Lemma \ref{lem:properties_rec_persp}(i), the estimates in \eqref{anis_eq:subadd_phi} apply to the mappings $\xi \mapsto \finf(x,\xi)$ and 
$(t,\xi) \mapsto \fpers(x,t,\xi)$ for any fixed $x \in U$. 

\begin{lem}\label{lem:tildef_incr}
  Suppose that $f:U \times \R^N \to \R$ has linear growth and is convex in $\xi\in\R^N$. If $f\geq \finf$ holds, then $\fpers(x,t,\xi)$ is non-decreasing in $t\in[0,\infty)$ for any fixed $(x,\xi)\in U\times\R^N$.
\end{lem}

\begin{proof}
  From the homogeneity of $\finf$ and the assumption $f\ge\finf$ we infer
  $\fpers(x,0,\xi)
  =\finf(x,\xi)
  =t\finf\big(x,\frac\xi t\big)
  \le tf\big(x,\frac{\xi}{t}\big)
  =\fpers(x,t,\xi)$ for all $t>0$. We combine the resulting inequality with the convexity property of $\fpers$ ensured by Lemma \ref{lem:properties_rec_persp}(i) to find
  \[
    \fpers(x,t_1,\xi)
    \le\left(1{-}\frac{t_1}{t_2}\right)\fpers(x,0,\xi)+\frac{t_1}{t_2}\fpers(x,t_2,\xi)
    \le\fpers(x,t_2,\xi)
    \qq\text{whenever }t_2>t_1\ge0\,.
  \]
  This proves the claimed monotonicity.  
\end{proof}

We record one more estimate which relates $f$ and its recession function.

\begin{lem}\label{lem:est_f_finf}
  Suppose that $f:U \times \R^N \to \R$ has linear growth and is convex in $\xi\in\R^N$. Then we have
  \[
    f(x,\xi+z) \leq f(x,\xi) + \finf(x,z)
    \qq\text{for all }x \in U \text{ and\/ } \xi,z \in \R^N\,.
  \]
\end{lem}

\begin{proof}
  The convexity assumption yields 
  \[
    f(x,\xi+z)
    \le(1{-}t)f\left(x,\frac{\xi}{1{-}t}\right)+tf\left(x,\frac{z}{t}\right)
    \qq\text{for all } t \in (0,1)\,,
  \]
  and the claim follows by sending $t\to0^+$.
\end{proof}

\begin{rem} \label{rem:f_finf_scalar}
  If $f:U \times \R^N \to \R$ has linear growth and is convex in $\xi\in\R^N$, then it holds
  \[
    \pm \nabla_\xi f(x,\xi) \ip \nu \leq \finf(x,\pm \nu)
    \qq\text{for all } x \in U, \text{ \ae{ }} \xi \in \R^N, \text{ and all } \nu \in \R^N. 
  \]
\end{rem}

\begin{proof}
  We combine the supporting hyperplane inequality $f(x,z)\geq f(x,\xi)+\nabla_\xi f(x,\xi) \ip (z{-}\xi)$ for $z\coleq \xi \pm \nu$ with the estimate $f(x,\xi\pm\nu)\le f(x,\xi)+\finf(x,\pm\nu)$ provided by Lemma \ref{lem:est_f_finf}.
\end{proof}

\subsection{General convex functionals of measures}\label{subsec:func_meas}

The functionals of measures we use in this paper go back to \cite{GofSer64,GMS79} and, in the first instance, are introduced in analogy with Definition \ref{defi:anis_TV}, just now with $(\LN,\nu)$ in place of $\D w$:

\begin{defi}[functionals of measures]\label{defi:meas_func}
  Given an open set $U\subseteq\R^N$ and an arbitrary Borel function $\fpers\colon U \times [0,\infty) \times \R^N \to [0,\infty)$ which is positively 1-homogeneous in its variables $(t,\xi) \in [0,\infty) \times \R^N$, we may understand $f(x,\xi)\coleq\fpers(x,1,\xi)$ for $(x,\xi)\in U\times\R^N$ and then obtain a functional of measures $\nu$ by setting
  \begin{equation}\label{eq:meas_func}
    \int_U f(\,.\,,\nu)
    \coleq\int_U \fpers\left(\,.\,,\frac{\d\LN}{\d\mu},\frac{\d\nu}{\d\mu}\right) \,\d\mu
    \qq\qq\text{for }\nu \in {\RM}_{\loc}(U,\R^N)\,,
  \end{equation}
  where $\mu$ is any non-negative Radon measure on $U$ such that $\LN\ll\mu$ and $|\nu|\ll\mu$ (for instance, $\mu=\LN+|\nu|$).
\end{defi}

It follows from the Radon-Nikod\'ym theorem and the homogeneity of $\fpers$ that the quantity defined in \eqref{eq:meas_func} does not depend on the choice of $\mu$. However, the notation is truly well-chosen only if $\fpers$ is the perspective function of an a priori given function $f\colon U \times \R^N \to [0,\infty)$, where for our purposes we assume that also $f$ is a Borel function and, for consistency with the underlying Definition \ref{defi:f-finf-perspective}, has linear growth and is convex in $\xi$. In this situation, it is well known \cite{GofSer64,GMS79} that the functional can be split in accordance with the Lebesgue decomposition $\nu=\frac{\d\nu}{\d\LN}\LN+\nu^\s$ as
\begin{equation}\label{eq:decomposition_meas_func}
  \int_Uf(\,.\,,\nu)
  =\int_Uf\left(.\,,\frac{\d\nu}{\d\LN}\right)\,\d\LN
  +\int_U\finf\left(.\,,\frac{\d\nu^\s}{\d|\nu^\s|}\right)\,\d|\nu^\s|\,.
\end{equation}
Specifically, for $\nu=g\LN$ with $g\in\L^1_\loc(U,\R^N)$, this reduces to $\int_Uf(\,.\,,\nu)=\int_Uf(\,.\,,g)\,\d\LN$. In particular, the preceding applies to the derivative measure $\nu=\D w$ of $w\in\BV_\loc(U)$. In this case, the Lebesgue decomposition takes the form $\D w=(\nabla w)\LN + \Ds w$, and the equality \eqref{eq:decomposition_meas_func} turns into
\[
  \int_Uf(\,.\,,\D w)
  = \int_{U} f\left( \,.\,, \nabla w \right) \,\d\mathcal{L}^N + \int_{U} \finf \left( .\,,\frac{\d \Ds w}{\d|\Ds w|}\right) \d|\Ds w|\,,
\]
and specifically, for $w\in\W^{1,1}_\loc(U)$, one gets $\int_Uf(\,.\,,\D w)=\int_Uf(\,.\,,\nabla w)\,\d\LN$.

\medskip

Even without any homogeneity requirement for $f$, the Reshetnyak semicontinuity and continuity theorems apply to functionals of measures in the preceding sense. As was the case for the earlier Theorems \ref{thm:Resh_LSC_hom_BV} and \ref{thm:Resh_cont_hom_BV}, this may be read off from  \cite[Theorems 2.38, 2.39]{AFP00}, now applied to $\fpers$ and  measures of type $(\LN,\D w)\in\RM(U,\R^{N+1})$ and decisively based on the 1-homogeneity of $\fpers$ in $(t,\xi)$; compare \cite[Section 2]{GMS79}, \cite[Section 4]{Delladio91}, \cite[Appendix]{KriRin10b}, \cite[Remark 2.5]{BecSch13}, for instance. Also taking into account the further properties of $\fpers$ from Lemma \ref{lem:properties_rec_persp}, this yields the following statements.

\begin{thm}[Reshetnyak semicontinuity; non-homogeneous version] \label{thm:Resh_LSC_inhom_BV}
  For an open set\/ $U\subseteq\R^N$ such that $|U|<\infty$, assume that $f\colon U \times \R^N \to [0,\infty)$ has linear growth, is convex in $\xi$, and is lower semicontinuous in $(x,\xi)$. Then, whenever a sequence $(u_k)_k$ in $\BV(U)$ converges in $\L^1(U)$ to $u\in\BV(U)$, there holds
  \[
    \liminf_{k \to \infty} \int_U f(\,.\,,\D u_k) \geq \int_U f(\,.\,,\D u)\,.
  \]
\end{thm}

\begin{thm}[Reshetnyak continuity; non-homogeneous version] \label{thm:Resh_cont_inhom_BV}
  For an open set\/ $U \subseteq \R^N$ such that $|U|<\infty$, assume that $f\colon U \times \R^N \to [0,\infty)$ has linear growth and that both $f$ and\/ $\finf$ are continuous in $(x,\xi)$. Then, whenever a sequence $(u_k)_k$ in $\BV(U)$ converges area-strictly to $u\in\BV(U)$, there holds
  \[
    \liminf_{k \to \infty} \int_U f(\,.\,,\D u_k) 
    = \int_U f(\,.\,,\D u)\,.
  \]
\end{thm}

\medskip

As mentioned in the introduction, it is nowadays well known that a Dirichlet boundary condition can be incorporated into the existence theory of linear-growth functionals via an additional boundary term. Since this involves boundary traces, we now restrict ourselves to Lipschitz domains, for which the boundary trace theorem  \cite[Theorem 3.87]{AFP00} applies. We give the following definition.

\begin{defi}[functionals of measures with boundary penalization term]\label{defi:func_meas_boundary}
  We consider an open bounded set $\Om\subseteq\R^N$ with Lipschitz boundary and a Borel function $f\colon\ol{\Om}\times\R^N\to[0,\infty)$ which has linear growth and is convex in $\xi$. Given $u_0 \in \W^{1,1}(\R^N)$, we introduce a functional $\MF_{u_0}^0$ by setting
  \[
    \MF_{u_0}^0[w] 
    \coleq\int_\Om f(\,.\,,\D w)
    +\int_{\partial\Om}\finf(\,.\,,(w{-}u_0)\nu_\Om)\,\d\H^{N-1}
    \qq\qq\text{for }w\in\BV(\Om)\,,
  \]
  where the occurrences of $w$ and $u_0$ in the boundary integral are understood as traces.
\end{defi}

The technical convenience of Definition \ref{defi:func_meas_boundary} lies partially in the fact that the additional term can be naturally incorporated into the functional of measures on an enlarged domain. In fact, for the extension $\ol{w}\coleq w\1_\Om+u_0\1_{\R^N \setminus\ol{\Om}}\in\BV(\R^N)$ of $w\in\BV(\Om)$ via the values of $u_0$, from \cite[Theorem 3.84]{AFP00} we get $\D\ol{w}=\D w$ in $\Om$ and $\D\ol{w}=(\nabla u_0)\LN$ in $\R^N\setminus\ol{\Om}$, but also $\D\ol{w}=(w{-}u_0)\nu_\Om\H^{n-1}$ at $\partial\Om$. Therefore, for an arbitrary open set $\pOm\subseteq\R^N$ such that $\Om\Subset\pOm$ and $|\pOm|<\infty$, with the understanding that $f$ suitably extends to $\pOm\times\R^N$, the definition of $\MF_{u_0}^0$ can be recast as
\begin{equation}\label{eq:F-enlarge-domain}
  \MF_{u_0}^0[w]
  = \int_\pOm f(\,.\,,\D\ol{w})
  - \int_{\pOm \setminus \ol{\Om}} f\left( \,.\,, \nabla u_0 \right) \,\d\mathcal{L}^N\,.
\end{equation}

\subsection[Product structure and slicing of \texorpdfstring{$\H^N$}{Hausdorff measure}]{\boldmath Product structure and slicing of $\H^N$}

Finally, we collect some useful observations on the product structure of Hausdorff measures.

\smallskip

In these regards, we follow the terminology of \cite[Definition 2.57]{AFP00}: We say that a set $R\subseteq\R^L$ is countably $K$-rectifiable if there holds $R\subseteq\bigcup_{i=1}^\infty F_i(\R^K)$ for countably many Lipschitz maps $F_i\colon\R^K\to\R^L$. Moreover, we say that a set $R\subseteq\R^L$ is countably $\H^K$-rectifiable if there holds $R\subseteq R'\cup R_0$ for a countably $K$-rectifiable set $R'\subseteq\R^L$ and an $\H^K$-negligible set $R_0\subseteq\R^L$.

By \cite[3.2.23]{Fed69} one has $\H^N(R\times S)=\H^K(R)\,\H^{K'}(S)$ for a countably $K$-rectifiable set $R\subseteq\R^L$, a countably $\H^{K'}$-rectifiable set $S\subseteq\R^{L'}$, and $N=K{+}K'$. For our purposes this will be relevant for $K=L=1$, $K'=N{-}1$, $L'=N$, where we recast the statement and briefly review the proof as follows.

\begin{lem}[product structure on products of rectifiable sets] \label{lem:prod_rectifiable}
  For a countably $\H^{N-1}$-rectifiable Borel set $S\subseteq\R^N$, we have
  \[
    \H^N\resmes(\R\times S)=\mathcal{L}^1\otimes(\H^{N-1}\resmes S)
    \qquad\text{as measures on }\R^{N+1}\,.
  \]
  In particular, for a Borel set $A\subset\R$, a countably $\H^{N-1}$-rectifiable Borel set $S\subset\R^N$, and a Borel function $h\colon A\times S\to{[0,\infty)}$, it is
  \[
    \int_{A\times S}h(t,x)\,\d\H^N(t,x)=\int_A\int_Sh(t,x)\,\d\H^{N-1}(x)\,\d t\,.
  \]
\end{lem}

\begin{proof}[Sketch of proof]
  By the product structure of Borel $\sigma$-algebras and Fubini's theorem it is enough to prove $\H^N(A\times S)=\mathcal{L}^1(A)\,\H^{N-1}(S)$ for a Borel set $A\subset\R$ and a countably $\H^{N-1}$-rectifiable Borel set $S\subset\R^N$. For $\H^{N-1}$-negligible $S$, this is easy to check with the definition of Hausdorff measures. Then, relying on \cite[Lemma 3.1.1]{Sim83} and possibly decomposing $S$, we can further assume $S=F(D)$ for a single one-to-one $\C^1$ mapping $F\colon\R^{N-1}\to\R^N$ and a Borel set $D\subseteq\R^{N-1}$. With $F_\vardiamond\colon\R^N\to\R^{N+1}$ given by $F_\vardiamond(t,x)\coleq(t,F(x))$, we clearly get $A\times S=F_\vardiamond(A\times D)$ and $\mathrm{J}F_\vardiamond(t,x)=\mathrm{J}F(x)$ for the Jacobians $\mathrm{J}F_\vardiamond\coleq{\det}\big(\D F_\vardiamond^\mathrm{T}\,\D F_\vardiamond\big)$ and $\mathrm{J}F\coleq{\det}\big(\D F^\mathrm{T}\,\D F\big)$. Therefore, a double application of the area formula confirms
  \[
    \H^N(A\times S)
    =\int_{A\times D}\mathrm{J}F_\vardiamond(t,x)\,\d(t,x)
    =\mathcal{L}^1(A)\int_D\mathrm{J}F(x)\,\d x
    =\mathcal{L}^1(A)\,\H^{N-1}(S)\,,
  \]
  as required.
\end{proof}

Specifically, for negligible sets $Z$, the preceding result is improved by the subsequent one, which applies even if a negligible $Z$ is \emph{not} contained in $\R\times S$ for some countably rectifiable $S$.

\begin{lem}[slicing negligible sets]\label{lem:Eilenberg}
Consider an arbitrary $\H^N$-negligible set\/ $Z\subseteq\R^{N+1}$. Then the section ${}_tZ\coleq\{x\in\R^N:(t,x)\in Z\}$ is $\H^{N-1}$-negligible for \ae{} $t\in\R$.
\end{lem}

Lemma \ref{lem:Eilenberg} follows straightforwardly from Eilenberg's inequality in the form of \cite[2.10.25]{Fed69}. Nonetheless, we here include an elementary deduction, which involves a basic version of Eilenberg's inequality for Hausdorff premeasures $\H^N_\delta$ only and thus avoids using the limit procedures for upper integrals in \cite[2.10.24]{Fed69}.

\begin{proof}[Proof of Lemma \ref{lem:Eilenberg}]
  We first observe that $\H^N(Z)=0$ implies $\H^N_\delta(Z)=0$ for the Hausdorff premeasure $\H^N_\delta$ with arbitrary $\delta>0$. For fixed $\delta>0$, we then consider an arbitrary sequence $(C_i)_i$ of sets $C_i\subseteq\R^{N+1}$ such that $\mathrm{diam}\,C_i<\delta$ and $Z\subseteq\bigcup_{i=1}^\infty C_i$. We write $\mathrm{p}(C_i)\coleq\{t\in\R:(t,x)\in C_i\text{ for some }x\in\R^N\}$ and use the basic estimates $\mathcal{L}^1(\mathrm{p}(C_i))\le\mathrm{diam}\,C_i$ and
  $\H^{N-1}_\delta({}_t(C_i))\le\omega_{N-1}\big(\frac{\mathrm{diam}\,C_i}2\big)^{N-1}$ in deriving (with the upper integral, since at this stage we do not yet have any measurability of $t\mapsto\H^{N-1}_\delta({}_tZ)$)
  \[
    \,\,\rule{0pt}{3ex}^\ast\hspace{-1.7ex}\int_\R\H^{N-1}_\delta({}_tZ)\,\d t
    \le\sum_{i=1}^\infty\,\,\rule{0pt}{3ex}^\ast\hspace{-1.7ex}\int_\R\H^{N-1}_\delta({}_t(C_i))\,\d t
    =\sum_{i=1}^\infty\,\,\rule{0pt}{3ex}^\ast\hspace{-1.7ex}\int_{\mathrm{p}(C_i)}\H^{N-1}_\delta({}_t(C_i))\,\d t\\
    \le2\omega_{N-1}\sum_{i=1}^\infty\bigg(\frac{\mathrm{diam}\,C_i}2\bigg)^N\,.
  \]
  In view of $\H^N_\delta(Z)=0$, the right-hand side of the preceding estimate can be made arbitrarily small, and we infer $\H^{N-1}_\delta({}_tZ)=0$ for \ae{} $t\in\R$. This conclusion, applied for $\delta=1/k$, $k\in\N$, gives a common $\mathcal{L}^1$-negligible set $E\subseteq\R$ such that $\H^{N-1}_{1/k}({}_tZ)=0$ for all $t\in E^\c$ and all $k\in\N$. Then we deduce $\H^{N-1}({}_tZ)=0$ for all $t\in E^\c$ and have verified the claim.
\end{proof}

\section{Statement of the main results}\label{sec:results}

As set out in the introduction, we aim at an existence theory for $\BV$-minimizers of the general functional $\MF_{u_0}^\mu$ introduced in \eqref{eq:MF_intro} and building on \eqref{eq:mu_wmp}. With the notation of Section \ref{subsec:func_meas} this functional takes the form
\begin{equation}\label{eq:MF}
  \MF_{u_0}^\mu[w]=\MF_{u_0}^0[w]+\llangle\mu_\pm\,;w^\mp\rrangle
  \qq\qq\text{for }w\in\BV(\Om)\,.
\end{equation}
In order to precisely state our results, from here on we generally understand that $\Om$ denotes an open and bounded subset of $\R^N$ with Lipschitz boundary, and we now specify the precise hypotheses, first for the integrand $f$ of $\MF_{u_0}^0$ and then for the measures $\mu_\pm$.

\begin{assum}[admissible integrands] \label{assum:f}
  For an integrand $f\colon\ol{\Om}\times\R^N\to[0,\infty)$ in the variables $(x,\xi)\in\ol{\Om}\times\R^N$, we impose the following set of assumptions\textup{:}
  
\smallskip 

\begin{enumerate}[\hspace{2ex}\rm(1)]
\renewcommand\theenumi{H\arabic{enumi}}
\item The function $f$ has linear growth in $\xi$, that is, there exist $\beta\ge\alpha>0$ such that\label{assum:H1}
\[
  \alpha|\xi| \leq f(x,\xi) \leq \beta(|\xi|+1)
  \qq\text{holds for all }(x,\xi) \in \Om \times \R^N\,.
\]
\item The function $f$ is convex in $\xi\in\R^N$.\label{assum:H2}
\item The function $f$ and the recession function $\finf$ are continuous in $(x,\xi)\in\ol{\Om}\times\R^N$.\label{assum:H3}
\item There is a constant $M\in\R$ such that\label{assum:H4}
\[
  f(x,\xi) \geq \finf(x,\xi)-M
  \qq\text{holds for all }(x,\xi) \in \Om \times \R^N\,.
\]
\end{enumerate}
\end{assum}

\smallskip

Here, \eqref{assum:H1}--\eqref{assum:H3} are standard assumptions, while the mild extra requirement \eqref{assum:H4} is less usual. In principle, we believe that one may further generalize the framework by allowing discontinuity of $f$ in $x$ (as long as some continuity for $\finf$ is kept; compare \cite{KriRin10a,BecSch15}), and by dispensing with \eqref{assum:H4} (compare the discussion in the introduction and in Remarks \ref{rem:H4_for_LSC} and \ref{rem:H4_for_coercivity}). However, for our taken approach, the above form of Assumption \ref{assum:f} seems just right.

\medskip

For the measures $\mu_\pm$, a preliminary mild requirement will be their admissibility in the following sense.

\begin{defi}[admissible measures]\label{defi:mu}
  We call a non-negative Radon measure $\mu$ on $\Om\subseteq\R^N$ admissible if it satisfies
  \[
    \mu(Z)=0
    \qq\text{for every }\H^{N-1}\text{-negligible Borel set }Z\subseteq\Om
  \]
  and
  \begin{equation}\label{eq:finite-integral}
    \int_\Om v^+\,\d\mu<\infty
  \qq\text{for every non-negative }v\in\BV(\Om)\,.
  \end{equation}
\end{defi}

In particular, we record that \eqref{eq:finite-integral} implies finiteness of $\mu$ on the bounded domain $\Om$. Moreover, we know from \cite[Proposition 4.1]{FS} that the conditions of Definition \ref{defi:mu} are equivalent with requiring an (isotropic or anisotropic) IC with arbitrarily large constant for $\mu$.

\smallskip

As foreshadowed in the introduction, our main assumptions on the measures $\mu_\pm$ then take the form of signed anisotropic ICs for a pair of measures, as introduced in \cite[Definition 3.1]{FS}. We recast the definition here for a pair $(\mu_1,\mu_2)$, but directly stress that the subsequently relevant choices are in fact $(\mu_1,\mu_2)=(\mu_\mp,\mu_\pm)$.
	
\begin{defi}[anisotropic ICs]\label{defi:IC}
Consider an open set $U \subseteq \R^N$ and a Borel function $\p \colon U \times \R^N \to [0,\infty)$ which is positively 1-homogeneous in $\xi$. A pair $(\mu_1,\mu_2)$ of finite non-negative Radon measures on $U$ satisfies the $\p$-anisotropic isoperimetric condition \ka in brief, the $\p$-IC\/\kz{} in $U$ with constant $C\in{[0,\infty)}$ if we have
  \[
    \mu_1(A^+)-\mu_2(A^1)\leq C \P_\p(A)
    \qq\qq\text{for all measurable }A\Subset U\,.
  \]
  We say that the single measure $\mu_1$ satisfies the $\p$-IC if this condition holds for $\mu_2 \equiv0$. 
\end{defi}

We observe that by \cite[Theorem 4.2]{FS}, if $U=\Om$ is bounded and Lipschitz, we can equivalently express the ICs by testing with $\BV$ functions instead of the sets $A$ above.

\medskip

Now we are ready for stating our first result on lower semicontinuity of $\MF_{u_0}^\mu$ from \eqref{eq:MF_intro}, or equivalently from \eqref{eq:MF}.

\begin{thm}[semicontinuity] \label{thm:LSC}
  We consider a bounded open set\/ $\Om \subseteq \R^N$ with Lipschitz boundary and $u_0 \in \W^{1,1}(\R^N)$. Moreover, we impose Assumption \ref{assum:f} for $f$. If $\mu_\pm$ are admissible measures on $\Om$ such that $(\mu_-,\mu_+)$ satisfies the $\finf$-IC in $\Om$ with constant $1$ and $(\mu_+,\mu_-)$ satisfies the $\widetilde{\finf}$-IC on $\Om$ with constant $1$, then $\MF_{u_0}^\mu$ is lower semicontinuous on $\BV(\Om)$ with respect to the strong topology of\/ $\L^1(\Om)$. 
\end{thm}

The proof of Theorem \ref{thm:LSC} works by reduction to the $\TV$ cases of \cite{FS} and is explicated in Section \ref{subsec:LSC}.

\begin{rem}
  If one weakens the requirement in \eqref{assum:H1} to merely $0\le f(x,\xi)\le\beta(|\xi|+1)$ and otherwise keeps the assumptions of Theorem \ref{thm:LSC}, then $\MF_{u_0}^\mu$ is still lower semicontinuous along $\L^1(\Om)$-convergent sequences $(u_k)_k$ in $\BV(\Om)$ such that additionally $\sup_{k\in\N}|\D u_k|(\Om)<\infty$. This is in fact a routine corollary, which is deduced by applying the theorem for the integrands $f_\eps(x,\xi)\coleq f(x,\xi)+\eps|\xi|$ with arbitrary $\eps>0$.
\end{rem}

As indicated in the introduction, existence of minimizers is a straightforward consequence of Theorem \ref{thm:LSC} on one hand and of a comparably straightforward coercivity property on the other hand. We emphasize, however, that coercivity indeed requires the ICs in the slightly stronger version with constant \emph{strictly} smaller than $1$ (compare the later Proposition \ref{prop:inhom_CS_coercivity}). Therefore, our existence theorem reads as follows.

\begin{thm}[existence of minima] \label{thm:exist_inhom}
  We consider a bounded open set\/ $\Om \subseteq \R^N$ with Lipschitz boundary and $u_0 \in \W^{1,1}(\R^N)$. Moreover, we impose Assumption \ref{assum:f} for $f$. If $\mu_\pm$ are admissible measures on $\Om$ such that $(\mu_-,\mu_+)$ satisfies the $\finf$-IC in $\Om$ with constant $C \in [0,1)$ and $(\mu_+,\mu_-)$ satisfies the $\widetilde{\finf}$-IC on $\Om$ with constant $C \in [0,1)$, then there exists a minimizer of\/ $\MF_{u_0}^\mu$ in $\BV(\Om)$.
\end{thm}

Full details on coercivity and the proof of Theorem \ref{thm:exist_inhom} are
given in Section \ref{subsec:coerc}.

\medskip

It seems worth pointing out that the analogous existence theorem of \cite{FS} for anisotropic $\TV$ cases with measures could be pushed to the limit case $C=1$ of the ICs at least in case of a bounded boundary datum $u_0$. The following example in dimension $N=2$ shows that there is no hope for an analogous extension of Theorem \ref{thm:exist_inhom} and thus that the present case of the general functional $\MF_{u_0}^\mu$ differs from the situations of \cite{FS}.

\begin{examp}[non-existence of minima in the borderline case] \label{ex:anis_area}
  In dimension $N=2$, we consider an arbitrary $x$-independent integrand $\p: \R^2 \to [0,\infty)$ which is positively 1-homogeneous and convex and satisfies $\p(\xi)>0$ for all $\xi\in\R^2\setminus\{0\}$. Moreover, on the ${\widetilde\p}^\circ$-unit ball\/ $\Om \coleq \left\{ {\widetilde\p}^\circ < 1 \right\}\Subset\R^2$ \ka where ${\widetilde\p}^\circ$ is the polar of\/ $\widetilde\p$ in the sense of Section \ref{subsec:polar_aniso}\kz{} we consider the anisotropic area functional
  \[
    \A_\p[w]  
    \coleq  \int_{\Om}  \sqrt{1+\p^2(\D w)}
    + \int_{\partial\Om} \p \left(w \nu_\Om  \right)\,\d\H^{1} 
    -  \int_{ \Om} \frac{w}{{\widetilde\p}^{\circ}}\dx
    \qq \text{for } w \in \BV(\Om)
  \]
  with integrand $f(\xi)=\sqrt{1+\p(\xi)^2}$, $\finf(\xi)=\p(\xi)$, measures $\mu_+\equiv0$ and $\mu_-=H\mathcal{L}^2$, where $H(x)\coleq 1/{\widetilde\p}^\circ(x)$ for $0\neq x\in\R^2$, and with zero boundary datum $u_0\equiv0$. Then, we will prove in Section \ref{sec:counterexample} that $\mu_-$ satisfies the $\p$-IC in $\R^2$ with precisely constant\/ $1$, but also that $\A_\p$ has no minimum in $\BV(\Om)$.
\end{examp}

We emphasize that Example \ref{ex:anis_area} covers in particular the standard area integrand $f(\xi)=\sqrt{1+|\xi|^2}$, and hence in case of the borderline IC it recovers non-existence phenomena in case of the $2$d prescribed-mean-curvature problem with datum $H\in\L^p(\Om)$ for all $p\in{[1,2)}$, but $H\notin\L^2(\Om)$.

\medskip

Our final result establishes a natural connection between our $\BV$-functional $\MF_{u_0}^\mu$ and its more straightforward $\W^{1,1}$-version, given by
\[
  \mathrm{F}^\mu[w]
  \coleq \int_\Om f( \,.\,, \nabla w ) \dx
  + \int_\Om w^\ast \,\d(\mu_+{-}\mu_-)\,.
\]
Evidently, $\MF_{u_0}^\mu$ coincides with $\mathrm{F}^\mu$ on the Dirichlet class $\W^{1,1}_{u_0}(\Om)\coleq u_0+\W^{1,1}_0(\Om)$. What is more, we have the following result.

\begin{thm}[existence of recovery sequences]\label{thm:rec_seq} 
  We consider a bounded open set\/ $\Om\subseteq\R^N$ with Lipschitz boundary, fix $u_0\in\W^{1,1}(\R^N)$, and recall the notation $\ol{u}\coleq\1_\Om u+\1_{\R^N\setminus\ol{\Om}}u_0$. Moreover, we impose Assumption \ref{assum:f} for $f$ with the lower bound in \eqref{assum:H1} weakened to mere non-negativity and with \eqref{assum:H4} entirely dropped. If $\mu_\pm$ are admissible measures on $\Om$ such that $\mu_+$ and $\mu_-$ are singular to each other, then, for every $u\in\BV(\Om)$, there exists a recovery sequence $(u_k)_k$ in $\W^{1,1}_{u_0}(\Om)$ such that\/ $(\ol{u_k})_k$ converges to $\ol{u}$ area-strictly in $\BV(\pOm)$, on any open $\pOm\subseteq\R^N$ such that $\Om\Subset\pOm$ and\/ $|\pOm|<\infty$, with
  \[
    \lim_{k\to\infty}\mathrm{F}^\mu[u_k]
    =\MF_{u_0}^\mu[u]\,.
  \]
  More specifically, for the single terms, we achieve
  \begin{gather*}
    \lim_{k\to\infty}\int_\Om f(\,.\,,\nabla u_k)\dx
    =\int_\Om f(\,.\,,\D u)+\int_{\partial\Om}\finf(\,.\,,(u{-}u_0)\nu_\Om)\,\d\H^{N-1}\,,\\
    \lim_{k\to\infty}\int_\Om u_k^\ast\,\d\mu_- 
    = \int_\Om u^+\,\d\mu_-
    \qq\qq\text{and}\qq\qq
    \lim_{k\to\infty}\int_\Om u_k^\ast\,\d\mu_+
    = \int_\Om u^-\,\d\mu_+\,.
  \end{gather*}
\end{thm}

The proof of Theorem \ref{thm:rec_seq} is implemented in Section \ref{sec:rec_seq} and decisively exploits that $\mu_+$ and $\mu_-$ are singular to each other, i.\@e.\@ are indeed the positive and negative part of the signed measure $\mu_+{-}\mu_-$. If this assumption were not at hand, $\mu_+$ and $\mu_-$ could partially cancel out in computing $\mu_+{-}\mu_-$, and $\MF_{u_0}^\mu$ might depend on the separate measures $\mu_+$ and $\mu_-$, while $\mathrm{F}^\mu$ generally and evidently depends on $\mu_+{-}\mu_-$ only. Therefore, the mutual singularity assumption for $\mu_\pm$ cannot be dropped from Theorem \ref{thm:rec_seq} or the subsequent Corollaries \ref{cor:consistency} and \ref{cor:rel_func}, and actually \cite[Example 3.13]{FS} shows that the conclusions may fail without this assumption.

\medskip

As a straightforward consequence of Theorem \ref{thm:rec_seq} we now record a coincidence of infimum values.

\begin{cor}[consistency] \label{cor:consistency}
  Under the assumptions of Theorem \ref{thm:rec_seq}, we have
  \[
    \inf_{\BV(\Om)} \MF_{u_0}^\mu
    = \inf_{\W^{1,1}_{u_0}(\Om)} \mathrm{F}^\mu\,.
  \]
\end{cor}

The combination of our results also identifies the \emph{relaxation} on all of $\BV(\Om)$ of $\mathrm{F}^\mu$ restricted to $\W^{1,1}_{u_0}(\Om)$, that is, the functional $(\mathrm{F}^\mu)^{\mathrm{rel}}_{u_0}$ abstractly defined by
\[
  (\mathrm{F}^\mu)^{\mathrm{rel}}_{u_0}[w]
  \coleq \inf \left\{ \liminf_{k \to \infty} \mathrm{F}^\mu[w_k] \,:\,\W^{1,1}_{u_0}(\Om) \ni \ w_k \to w \text{ in } \L^1(\Om) \right\}
  \qq \text{for }w\in\BV(\Om)\,.
\]
In fact, Theorem \ref{thm:LSC} guarantees $(\mathrm{F}^\mu)^{\mathrm{rel}}_{u_0}\ge\MF_{u_0}^\mu$ on $\BV(\Om)$, while Theorem \ref{thm:rec_seq} yields $(\mathrm{F}^\mu)^{\mathrm{rel}}_{u_0}\le\MF_{u_0}^\mu$ on $\BV(\Om)$. Therefore, we may state the following corollary.

\begin{cor}[relaxation] \label{cor:rel_func}
  We impose the general hypotheses of Theorems \ref{thm:LSC} and \ref{thm:rec_seq} \ka in particular the full Assumption \ref{assum:f} and admissibility of mutually singular measures $\mu_\pm$\kz. If $(\mu_-,\mu_+)$ satisfies the $\finf$-IC in $\Om$ with constant $1$ and $(\mu_+,\mu_-)$ satisfies the $\widetilde{\finf}$-IC on $\Om$ with constant $1$, then we have
  \[
    (\mathrm{F}^\mu)^{\mathrm{rel}}_{u_0} = \MF_{u_0}^\mu
    \qq \text{on } \BV(\Om)\,.
  \]
\end{cor}

We stress that it is quite usual to consider the lower semicontinuous envelope $(\mathrm{F}^\mu)^{\mathrm{rel}}_{u_0}$ as the natural extension by semicontinuity from $\W^{1,1}_{u_0}(\Om)$ to all of $\BV(\Om)$. Therefore, Corollary \ref{cor:rel_func} steadily underpins that our functional $\MF_{u_0}^\mu$ is a meaningful and natural choice.

\section{Lower semicontinuity and existence theory\texorpdfstring{ for \boldmath$\MF_{u_0}^\mu$}{}}\label{sec:LSC_exist}

In this section, we establish Theorems \ref{thm:LSC} and \ref{thm:exist_inhom}. To this end, we recall the general assumption that $\Om$ is a bounded open set with Lipschitz boundary in $\R^N$, and we record that we can preserve all properties of Assumption \ref{assum:f} when extending the integrand $f\colon\ol{\Om}\times\R^N\to[0,\infty)$ --- first to $f_{x_\ast} \colon(\ol{\Om}\cup U_{x_\ast})\times\R^N\to[0,\infty)$, for a suitably small neighborhood $U_{x_\ast}$ of any boundary point $x_\ast\in\partial\Om$, and then, by pasting together local extensions $f_{x_\ast}$ via a partition of unity, even to $f\colon\R^N\times\R^N\to[0,\infty)$. Therefore, for simplicity, we can and will assume in the sequel that Assumption \ref{assum:f} applies with $\Om=\R^N$ to $f$ defined on $\R^N\times\R^N$. Furthermore, since the conclusions of the theorems are not affected by adding any finite constant to the integrand $f$, we replace assumptions \eqref{assum:H1} and \eqref{assum:H4} with the following slight variants which are technically convenient for our approach:
\begin{enumerate}[\hspace{2.5ex}\rm(1)]
\renewcommand\theenumi{H\arabic{enumi}$^\prime$}
\item There exist $\beta\ge\alpha>0$ such that $\alpha\sqrt{1+|\xi|^2} \leq f(x,\xi) \leq \beta\sqrt{1+|\xi|^2}$ for all $(x,\xi) \in \R^N \times \R^N$.\label{assum:H1'}
\refstepcounter{enumi}\refstepcounter{enumi}
\item There holds $f(x,\xi) \geq \finf(x,\xi)$ for all $x,\xi \in \R^N$.\label{assum:H4'}
\end{enumerate}
Here, the passage from \eqref{assum:H1} to \eqref{assum:H1'} may require enlarging the constant $\beta$, for instance, replacing $\beta$ with $\alpha\!+\!\sqrt2\,\beta$, but this does not harm any subsequent argument.

\subsection[{The functional\texorpdfstring{ $\MF_{u_0}^\mu$}{} with extra variable}]{The functional\texorpdfstring{ \boldmath$\MF_{u_0}^\mu$}{} with extra variable}

As explained earlier, our semicontinuity proof is based on rewriting the functional $\MF_{u_0}^\mu$ with the help of an extra variable $x_0$. In fact, the 1-homogeneous integrand of the rewritten functional is roughly the perspective function $\fpers$ (see Section \ref{subsec:integrands}) and in technically precise language is the following function $\p$, which is properly defined on $\R^{N+1}\times\R^{N+1}=(\R \times \R^N)\times(\R \times \R^N)$:

\begin{defi}\label{defi:p}
  Given $f$ as in Assumption \ref{assum:f} with $\Om=\R^N$, we introduce $\p:(\R \times \R^N)\times(\R \times \R^N)\to [0,\infty)$ by setting
  \[
    \p((x_0,x),(\xi_0,\xi))=\p(x_0,x,\xi_0,\xi)\coleq
    \fpers(x,|\xi_0|,\xi) =
    \begin{cases}
      |\xi_0| \, f\Big( x,\frac{\xi}{|\xi_0|}\Big) & \text{if } \xi_0 \neq 0\\
      \finf(x,\xi) & \text{if } \xi_0 = 0
    \end{cases}
  \]
  for $(x_0,x),(\xi_0,\xi)\in\R\times\R^N$. 
\end{defi}

In particular, we have $\p(x_0,x,1,\xi)=f(x,\xi)$ and $\p(x_0,x,0,\xi)=\finf(x,\xi)$, and the integrand $\p$ falls into the framework of \cite{FS} in the sense recorded next.

\begin{lem}\label{lem:p}
  For $f$ as in Assumption \ref{assum:f} with $\Om=\R^N$, \eqref{assum:H1'}, \eqref{assum:H4'} and $\p$ given by Definition \ref{defi:p}, we have\textup{:}
  \begin{enumerate}[{\rm(i)}]
  \item $\p$ is positively 1-homogeneous and even in $(\xi_0,\xi)$, that is $\p^\infty=\widetilde\p=\p$.\label{item:p:i}
  \item\label{item:p:ii}$\p$ is comparable to the Euclidean norm, that is
  \[
    \alpha|(\xi_0,\xi)| \leq \p(x_0,x,\xi_0,\xi) \leq  \beta|(\xi_0,\xi)|
    \qq\text{for all }(x_0,x),(\xi_0,\xi)\in\R\times\R^N\,.
  \] 
  \item $\p$ is convex in $(\xi_0,\xi)$.\label{item:p:iii}
  \item $\p$ is continuous in $(x_0,x,\xi_0,\xi)$.\label{item:p:iv}
  \item There holds\label{item:p:v}
    \[
      \p(x_0,x,\xi_0,\xi)\ge\finf(x,\xi)
      \qq\text{for all }(x_0,x),(\xi_0,\xi)\in\R\times\R^N\,.
    \]
  \end{enumerate}
  In summary, \eqref{item:p:i}--\eqref{item:p:iv} express that, with the variables grouped into $(x_0,x)$ and $(\xi_0,\xi)$, the integrand $\p=\widetilde\p$ is admissible in the sense of \textup{\cite[Assumption 2.11]{FS}}.
\end{lem}

\begin{proof}
  Claims \eqref{item:p:i}--\eqref{item:p:iv} are straightforward consequences of the properties of $\fpers$ provided by Lemma \ref{lem:properties_rec_persp}, where \eqref{item:p:ii} draws on the adjusted assumption \eqref{assum:H1'}, and otherwise 
  only the convexity property \eqref{item:p:iii} needs further explication. Indeed, fix $(x_0,x),(\xi_0,\xi),(\xi_0',\xi')\in\R\times\R^N$ and $\lambda\in[0,1]$. Then the convexity of $\fpers(x,\,.\,,\,.\,)$ on $[0,\infty)\times\R^N$ guaranteed by Lemma \ref{lem:properties_rec_persp}\eqref{item:properties_rec_persp_i} combined with the \eqref{assum:H4'}-based monotonicity property of Lemma \ref{lem:tildef_incr} ensures
  \begin{align*}
    \p((x_0,x),\lambda(\xi_0,\xi)+(1{-}\lambda)(\xi_0',\xi'))
    &=\fpers(x,|\lambda\xi_0+(1{-}\lambda)\xi_0'|,\lambda\xi+(1{-}\lambda)\xi')\\
    &\le\fpers(x,\lambda|\xi_0|+(1{-}\lambda)|\xi_0'|,\lambda\xi+(1{-}\lambda)\xi')\\
    &\le\lambda\fpers(x,|\xi_0|,\xi)+(1{-}\lambda)\fpers(x,|\xi_0'|,\xi')\\
    &=\lambda\p(x_0,x,\xi_0,\xi)+(1{-}\lambda)\p(x_0,x,\xi_0',\xi')\,.
  \end{align*}
  This ends the proof of claim \eqref{item:p:iii}. Finally, claim \eqref{item:p:v} follows from \eqref{assum:H4'} via Lemma \ref{lem:tildef_incr}.
\end{proof}

In order to achieve the rewriting of $\MF_{u_0}^\mu$, we recall from the introduction that we use the cylinder
\[
  \Omd \coleq (0,1) \times \Om \subseteq \R^{N+1}
\]
over $\Omega\subseteq\R^N$ as new domain and, for arbitrary $w\in\BV(\Om)$, define $\wdi\in\BV(\Omd)$ with extra variable $x_0$ by setting
\begin{equation}\label{eq:wdi}
  \wdi(x_0,x) \coleq x_0 + w(x) \quad \text{ for } (x_0,x)\in\Omd\,.
\end{equation}
In similar vein, for a non-negative Radon measure $\mu$ on $\Om$, we introduce a new non-negative Radon measure $\mud$ on $\Omd$ as
\[
  \mu_\lozenge \coleq \big( \mathcal{L}^1 \resmes (0,1) \big) \otimes \mu\,.
\]
With this notation, the next proposition allows for rewriting all terms of $\MF_{u_0}^\mu$: First, it identifies $\int_{\Om}f(\,.\,,\D w)$, understood in the sense of Section \ref{subsec:func_meas}, as the $\p$-anisotropic total variation $|\D\wdi|_\p(\Omd)$. Second, it provides a corresponding rewriting of boundary terms. Third, it recasts also the terms with $\mu_\pm$ via ${\mu_\pm}_\lozenge$.

\begin{prop}\label{prop:Omd-terms}
  Consider $f$ as in Assumption \ref{assum:f} with $\Om=\R^N$, \eqref{assum:H1'}, \eqref{assum:H4'}, $\p$ given by Definition \ref{defi:p}, admissible measures $\mu_\pm$ on $\Om$, and $u_0 \in \W^{1,1}(\R^N)$. Then, for every $w\in\BV(\Om)$, we have
  \begin{align}
    |\D\wdi|_\p(\Omd)
    &=\int_\Om f(\,.\,,\D w)\,,
      \label{eq:reformulation_var_int}\\
    \int_{\partial \Omd} \p \left(\,.\,,(\wdi{-}{u_0}_\lozenge)\nud\right)\d\H^{N} 
    &=\int_{ \partial \Om } \finf(\,.\,,(w{-}u_0)\nu_{\Om})\,\d\H^{N-1} + 2 \int_{\Om} f(\,.\,,0)|w{-}u_0| \,\d\mathcal{L}^{N}\,,
      \label{eq:reformulation_var_bd}\\
    \int_{\Omd} (\wdi)^{\mp} \,\d{\mu_\pm}_\lozenge
    &= \int_{\Om} w^{\mp} \,\d\mu_\pm + \frac{\mu_\pm(\Om)}{2}\,,
      \label{eq:reformulation_meas}
  \end{align}
  where clearly \eqref{eq:reformulation_var_bd} involves the traces of\/ $\wdi{-}{u_0}_\lozenge$ on $\partial\Omd$ and of\/ $w{-}u_0$ on $\partial\Om$. In particular, in the short-hand notation of \eqref{eq:mu_wmp} and with $\mu(\Om)\coleq\mu_+(\Om){-}\mu_-(\Om)$, the equality \eqref{eq:reformulation_meas} gives
  \begin{equation}\label{eq:int_u_utilde}
    \llangle{\mu_\pm}_\lozenge\,;(\wdi)^\mp\rrangle
    = \llangle\mu_\pm\,;w^\mp\rrangle + \frac{\mu(\Om)}{2}\,.
  \end{equation}
\end{prop}

\begin{proof}
  In order to establish \eqref{eq:reformulation_var_int}, we first split into absolutely continuous and singular parts in the sense of
  \[
    |\D\wdi|_\p(\Omd)
    = \int_{\Omd} \p (\,.\,,\nabla \wdi)\,\d\mathcal{L}^{N+1}
    + \int_{\Omd} \p \left(.\,,\frac{\d \D^\s \wdi}{\d|\D^\s \wdi|}\right) \d|\D^\s \wdi|\,.
  \]
  For the absolutely continuous part, using first constancy of $\p$ in its first variable, then Fubini's theorem together with $\nabla\wdi(x_0,x)=(1,\nabla w(x))$ for $\mathcal{L}^{N+1}$-a.\@e.\@ $(x_0,x)\in\Omd$, and finally $\p(1,x,1,\xi)=f(x,\xi)$, we get
  \[\begin{aligned}
    \int_{\Omd} \p (\,.\,,\nabla \wdi)\,\d\mathcal{L}^{N+1}
    &= \int_{\Omd} \p \left((1,x), \nabla \wdi(x_0,x) \right)\,\d\mathcal{L}^{N+1}(x_0,x) \\
    &= \mathcal{L}^1\left((0,1)\right) \, \cdot \int_{\Om} \p \left(1,x,1,\nabla w(x) \right)\,\d\mathcal{L}^N(x) \\
    &= \int_{\Om} f\left(\,.\, , \nabla w \right)\,\d\mathcal{L}^N\,.
  \end{aligned}\]
  For the singular part, since $\D^\s \wdi=(0, \mathcal{L}^1 \otimes \D^\s w)$ implies $\frac{\d \D^\s \wdi}{\d|\D^\s \wdi|}(x_0,x)=\Big(0,\frac{\d \D^\s w}{\d|\D^\s w|}(x)\Big)$ for $|\D^\s \wdi|$-a.\@e.\@ $(x_0,x)\in\Omd$ and since we have $\p(1,x,0,\xi)=\finf(x,\xi)$, we may similarly rewrite
  \[\begin{aligned}
    \int_{\Omd} \p \left(.\,,\frac{\d \D^\s \wdi}{\d|\D^\s \wdi|}\right) \d|\D^\s \wdi|
    &= \int_{\Omd} \p \left((1,x),\frac{\d \D^\s \wdi}{\d|\D^\s \wdi|}(x_0,x)\right) \d|\D^\s \wdi|(x_0,x) \\
    &= \mathcal{L}^1\left((0,1)\right) \, \cdot \int_{\Om} \p\left(1,x,0,\frac{\d \D^\s w}{\d|\D^\s w|}(x)\right) \d|\D^\s w|(x) \\
    &= \int_{\Om} \finf \left(.\,,\frac{\d \D^\s w}{\d|\D^\s w|}\right) \d|\D^\s w| \,.
  \end{aligned}\]
  Combining the previous equalities and recalling Definition \ref{defi:meas_func}, we arrive at \eqref{eq:reformulation_var_int}.

  \smallskip
  
  Next we turn to the equality \eqref{eq:reformulation_var_bd} at $\partial \Omd = \left(  [0,1] \times \partial \Om \right)   \cupdot \left(   \left\{ 0\right\}  \times \Om \right) \cupdot \left( \left\{ 1\right\}  \times \Om \right)$. We initially record 
  $\wdi(x_0,x)-{u_0}_\lozenge(x_0,x)=w(x)-u_0(x)$ for $\H^N$-a.\@e.\@ $(x_0,x) \in \partial \Omd$, and observe that the inward normal $\nu_{\Omd}$ at $\partial\Omd$ equals $(1,0)\in\R\times\R^N$ on the boundary portion $\left\{0\right\} \times \Om$, whereas it equals $(-1,0)\in\R\times\R^N$ on $\left\{1\right\} \times \Om$. Then, on these two boundary portions we may employ the 1-homogeneity of $\p$ in $(\xi_0,\xi)$ and $\p(x_0,x,\pm1,0)=f(x,0)$ to compute
  \begin{align*}
    \int_{\{0\} \times \Om} \p(\,.\,,(\wdi{-}{u_0}_\lozenge)\nu_{\Omd}) \,\d\H^{N}
    &= \int_{\Om} \p(0,x,w(x){-}u_0(x),0) \,\d\mathcal{L}^{N}(x)
    = \int_{\Om} f(\,.\,,0)|w{-}u_0|\,\d\mathcal{L}^{N}\,,\\
    \int_{\{1\} \times \Om} \p(\,.\,,(\wdi{-}{u_0}_\lozenge)\nu_{\Omd}) \,\d\H^{N}
    &= \int_{\Om} \p(1,x,u_0(x){-}w(x),0) \,\d\mathcal{L}^{N}(x)
    = \int_{\Om} f(\,.\,,0)|w{-}u_0|\,\d\mathcal{L}^{N}\,.
  \end{align*}
  Moreover, for $\H^{N}$-\ae{} $(x_0,x)$ in the remaining boundary portion $[0,1] \times \partial \Om$ we observe $\nu_{\Omd}(x_0,x)  = (0,\nu_{\Om}(x))$. Then, using $\p(x_0,x,0,\xi)=f^\infty(x,\xi)$ and exploiting Lemma \ref{lem:prod_rectifiable} with $S=\partial\Om$ (in other words: the product structure of $\H^N$ on $[0,1]\times\partial\Om$) via Fubini's theorem, we also deduce
  \begin{align*}
    \int_{ [0,1] \times \partial \Om } \p(\,.\,,(\wdi{-}{u_0}_\lozenge)\nu_{\Omd}) \,\d\H^{N}
    &= \int_{ [0,1] \times \partial \Om }\p(x_0,x,0, (w(x){-}u_0(x))\nu_{\Om}(x))\,\d\H^{N}(x_0,x) \\
    &= \int_{ [0,1] \times \partial \Om }\finf(x,(w(x){-}u_0(x))\nu_{\Om}(x))\,\d\H^{N}(x_0,x) \\
    &= \mathcal{L}^1([0,1]) \, \cdot \int_{ \partial \Om } \finf(\,.\,,(w{-}u_0)\nu_{\Om})\,\d\H^{N-1} \\
    &= \int_{ \partial \Om } \finf(\,.\,,(w{-}u_0)\nu_{\Om})\,\d\H^{N-1} \,.
  \end{align*}
  By combining the equalities on the three boundary portions we arrive at \eqref{eq:reformulation_var_bd}.

  \smallskip

  Finally, \eqref{eq:reformulation_meas} follows quickly from the definitions of $\Omd$, ${\mud}_\pm$ and from Fubini's theorem. Indeed, we have
  \begin{align*} 
    \int_{\Omd} (\wdi)^\mp \,\d{\mud}_\pm
    = \mu_\pm(\Om) \int_{0}^{1} \! x_0 \,\d x_0
      + \mathcal{L}^1((0,1)) \int_{\Om} w^\mp\,\d\mu_\pm
    = \frac{\mu_\pm(\Om)}{2} + \int_{\Om}w^\mp \,\d\mu_\pm\,,
  \end{align*}
  which gives \eqref{eq:reformulation_meas}.
\end{proof}

Before closing this subsection we add two technically convenient remarks.

\begin{rem}
  The application of Proposition \ref{prop:Omd-terms} for the mirrored integrand\/ $f(x,{-}\xi)$, which is connected with $\finf(x,{-}\xi)=\widetilde{\finf}(x,\xi)$ and
  $\p(x_0,x,\xi_0,{-}\xi)=\p(x_0,x,{-}\xi_0,{-}\xi)=\widetilde\p(x_0,x,\xi_0,\xi)$, turns the equality \eqref{eq:reformulation_var_int} into
  \begin{equation}\label{eq:estim_deriv_mirr}
    |\D \wdi|_{\widetilde{\p}}(\Omd)
    = \int_{\Om} f( \,.\,, {-}\nabla w )\dx + \int_{\Om} \finf\left( .\,,{-}\frac{\d \Ds w}{\d|\Ds w|}\right) \d|\Ds w|
  \end{equation}
  for $w \in \BV(\Om)$. The equality \eqref{eq:reformulation_var_bd} is recast in an analogous fashion.
\end{rem}

\begin{rem}
  Combining \eqref{eq:reformulation_var_int} and \eqref{eq:estim_deriv_mirr} with the bound $f\ge\finf$ of \eqref{assum:H4'}, for $w \in \BV(\Om)$ we find
  \[
    |\D\wdi|_{\p}(\Omd) \geq |\D w|_{\finf}(\Om)
    \qq\qq\text{and}\qq\qq
    |\D\wdi|_{\widetilde\p}(\Omd) \geq |\D w|_{\widetilde\finf}(\Om)\,.
  \]
\end{rem}

\subsection{ICs with extra variable}

The rewriting \eqref{eq:reformulation_meas}, \eqref{eq:int_u_utilde} of the measure terms naturally brings up the question for the properties of the measures ${\mu_\pm}_\lozenge$. Indeed, in this regard a preliminary observation is that admissibility in the sense of Definition \ref{defi:mu} carries over from $\mu_\pm$ to ${\mu_\pm}_\lozenge$:

\begin{lem} \label{lem:iff_admiss_mud}
  If a non-negative Radon measure $\mu$ on $\Om$ is admissible in the sense of Definition \ref{defi:mu}, then the measure $\mud$ on $\Omd$ is also admissible. 
\end{lem}

\begin{proof}
  We need to show
  \begin{equation}\label{eq:negligible_wmu}
    \mud(Z)=0
    \qq\text{for every }\H^{N}\text{-negligible Borel set }Z\subseteq\Omd
  \end{equation}
  and
  \begin{equation}\label{eq:finite-integral_wmu}
    \int_{\Omd} v^+\,\d\mud<\infty
    \qq\text{for every non-negative }v\in\BV(\Omd)\,.
  \end{equation}

  \smallskip
  
  In order to check \eqref{eq:negligible_wmu}, we consider a Borel set $Z\subseteq\Omd$ such that $\H^N(Z)=0$. By Lemma \ref{lem:Eilenberg}, the Borel set ${}_{x_0}Z \coleq \left\{ x \in \Om :\, (x_0,x) \in Z \right\}$ satisfies $\H^{N-1}({}_{x_0}Z)=0$ for \ae{} $x_0 \in (0,1)$. Consequently, the admissibility of $\mu$ implies $\mu({}_{x_0}Z)=0$ for \ae{} $x_0\in(0,1)$, and then $\mud(Z)=(\mathcal{L}^1\otimes\mu)(Z)=0$ follows from Fubini's theorem. This completes the proof of \eqref{eq:negligible_wmu}. 

  \smallskip 

  Next we turn to the proof of \eqref{eq:finite-integral_wmu}. We start observing that the admissibility of $\mu$ in the sense of Definition \ref{defi:mu} implies by \cite[Proposition 4.1]{FS} that the (isotropic) IC holds for $\mu$ in $\Om$ with some constant $C\in{[0,\infty)}$. This in turn implies by the characterization result \cite[Theorem 7.5]{Schmidt25}\footnote{At this point we refer also to \cite[Theorem 4.7]{MeyZie77}, \cite[Theorem 5.12.4]{Ziemer89}, \cite[Theorem 3.5]{PhuTor08}, \cite[Theorem 4.4]{PhuTor17} for closely related predecessor versions of the result on all of $\R^N$ and with potential enlargement of the constant $C$ and to \cite[Theorems 4.2, 4.6]{FS} for a version for pairs of measures.} that we have
  \[
    \int_\Om\psi(x)\,\d\mu(x)
    \le C\int_\Om|\nabla\psi(x)|\dx
  \]
  for all non-negative $\psi\in\C^\infty_\c(\Om)$.
  This version of the IC with smooth test functions allows for easily incorporating the extra variable, in fact with Fubini's theorem we infer
  \[\begin{aligned}
    \int_\Omd\Psi(x_0,x)\,\d\mud(x_0,x)
    &=\int_0^1\int_\Om\Psi(x_0,x)\,\d\mu(x)\,\d x_0\\
    &\le C\int_0^1\int_\Om|\nabla_x\Psi(x_0,x)|\dx\,\d x_0
    \le C\int_\Omd|\nabla\Psi(x_0,x)|\,\d(x_0,x)
  \end{aligned}\]
  for all non-negative $\Psi\in\C^\infty_\c(\Omd)$. At this stage we deduce \eqref{eq:finite-integral_wmu} either directly by specializing \cite[Theorem 4.6]{FS} to the case of a single measure or by using \cite[Theorem 7.5]{Schmidt25} to reach \eqref{eq:finite-integral_wmu} first for $v\in\W^{1,1}(\Omd)$ and then observing that every $v\in\BV(\Omd)$ satisfies $v\le\widehat v$ for some $\widehat v\in\W^{1,1}(\Omd)$. 
\end{proof}

The decisive point for our approach is now that also our anisotropic ICs suitably carry over from $\mu_\pm$ to ${\mu_{\pm}}_\lozenge$. This is recorded next.

\begin{prop}[anisotropic ICs with extra variable]\label{prop:higher_dim}
  Consider $f$ as in Assumption \ref{assum:f} with $\Om=\R^N$, \eqref{assum:H1'}, \eqref{assum:H4'} 
  and $\p$ given by Definition \ref{defi:p}. If $\mu_\pm$ are admissible measures on $\Om$, then the $\finf$-IC for $(\mu_-,\mu_+)$ and the $\widetilde\finf$-IC for $(\mu_+,\mu_-)$, both in $\Om$ with a constant $C \in [0,\infty)$, imply the $\p$-IC for $({\mu_-}_\lozenge,{\mu_+}_\lozenge)$ and the $\widetilde\p$-IC for $({\mu_+}_\lozenge,{\mu_-}_\lozenge)$, now both in $\Omd$ and still with the same constant $C$.
\end{prop}

We recall at this point that, in the case we consider most relevant, the measures $\mu_+$ and $\mu_-$ are the positive and negative part of a signed measure or in other words are singular to each other. For now, we limit ourselves to proving Proposition \ref{prop:higher_dim} in this case, in which we can draw on a characterization of the relevant ICs with smooth test functions and can keep the reasoning comparably straightforward. A proof in full generality can be based on more cumbersome slicing arguments on the level $\BV$ test functions, but may be less relevant and is deferred to Appendix \ref{asec:sections_ICs}.

\begin{proof}[Proof of Proposition \ref{prop:higher_dim} in case $\mu_+$ and $\mu_-$ are singular to each other]
  We first observe that by Lemma \ref{lem:iff_admiss_mud} the admissibility of $\mu_\pm$ carries over to ${\mu_\pm}_\lozenge$. Moreover, by \cite[Theorem 4.2]{FS} the assumed ICs for $(\mu_-,\mu_+)$ and $(\mu_+,\mu_-)$ imply
  \begin{equation}\label{eq:f-finf-IC}
    -C\int_\Om \widetilde{\finf}(x,\nabla\psi(x))\dx
    \le\int_\Om\psi(x)\,\d(\mu_-{-}\mu_+)(x)
    \le C\int_\Om\finf(x,\nabla\psi(x))\dx
  \end{equation}
  for all non-negative $\psi\in\C^\infty_\c(\Om)$. We next apply a Fubini argument very similar to the one in the preceding proof and additionally exploit the estimate $\finf(x,\xi)\le\p(x_0,x,\xi_0,\xi)$ of Lemma \ref{lem:p}\eqref{item:p:v}. In this way we see that the right-hand estimate in \eqref{eq:f-finf-IC} induces
  \[\begin{aligned}
    \int_\Omd\Psi(x_0,x)\,\d({\mu_-}_\lozenge{-}{\mu_+}_\lozenge)(x_0,x)
    &=\int_0^1\int_\Om\Psi(x_0,x)\,\d(\mu_-{-}\mu_+)(x)\,\d x_0\\
    &\le C\int_0^1\int_\Om f^\infty(x,\nabla_x\Psi(x_0,x))\dx\,\d x_0\\
    &\le C\int_0^1\int_\Om\p(x_0,x,\partial_{x_0}\Psi(x_0,x),\nabla_x\Psi(x_0,x))\dx\,\d x_0\\
    &=C\int_\Omd\p(x_0,x,\nabla\Psi(x_0,x))\,\d(x_0,x)
  \end{aligned}\]
  for all non-negative $\Psi\in\C^\infty_\c(\Omd)$. In addition, the left-hand estimate in \eqref{eq:f-finf-IC} induces an analogous lower estimate which involves $\widetilde\p$. Finally, we exploit that $\mu_+$ and $\mu_-$ are singular to each other and thus we have the full spectrum of IC characterizations from \cite[Theorem 4.6]{FS} at our disposal. This in fact allows to move back from the previous estimates with smooth test functions $\Psi$ to the original set-based definition of our ICs and thus yields the $\p$-IC for $({\mu_-}_\lozenge,{\mu_+}_\lozenge)$ in $\Omd$ and the $\widetilde\p$-IC for $({\mu_+}_\lozenge,{\mu_-}_\lozenge)$, as claimed.
\end{proof}

\begin{rem}
  The statement of Proposition \ref{prop:higher_dim} suffices for our purposes, but in fact can be slightly improved inasmuch that just one of the two assumed ICs is enough to deduce the corresponding one in the conclusion, e.\@g.\@ the IC for $(\mu_-,\mu_+)$ implies the one for $({\mu_-}_\lozenge,{\mu_+}_\lozenge)$. However, as the equivalence results in \textup{\cite[Section 4]{FS}} are stated with combined ICs on $(\mu_-,\mu_+)$ and $(\mu_+,\mu_-)$, their verbatim application in the preceding proof gives only the above \textup{``}combined\/\textup{''} version of Proposition \ref{prop:higher_dim}. The proof of the improved version mentioned is fully analogous in principle, but would require revisiting a certain amount of arguments in \textup{\cite[Section 4]{FS}}. 
\end{rem}

\subsection{Lower semicontinuity} \label{subsec:LSC}

With the previous results at hand we are now ready for implementing a comparably short proof of Theorem \ref{thm:LSC} by reduction to our previous result in \cite[Theorem 3.5]{FS} for the case of anisotropic total variations.
 
\begin{proof}[Proof of Theorem \ref{thm:LSC}]
  We consider a sequence $(u_k)_k$ in $\BV(\Om)$ which converges to $u \in \BV(\Om)$ strongly in $\L^1(\Om)$. For the functions ${u_k}_\lozenge,\udi\in\BV(\Omd)$ given by \eqref{eq:wdi}, we observe that also $({u_k}_\lozenge)_k$ converges to $\udi$ strongly in $\L^1(\Omd)$, and we exploit Proposition \ref{prop:Omd-terms} to recast our functional $\MF_{u_0}^\mu$ on arbitrary $w\in\BV(\Om)$ as
  \begin{equation}\label{eq:F=P+rest}
    \MF_{u_0}^\mu[w] = \widehat\Phi[\wdi] - 2 \int_{\Om}  f( \,.\, , 0 )|w{-}u_0|\,\d\mathcal{L}^{N} - \frac{\mu(\Om)}{2}
    \qq\text{for }w\in\BV(\Om)
  \end{equation}
  with the abbreviations $\mu(\Omega)\coleq\mu_+(\Om){-}\mu_-(\Om)$ and
  \[
    \widehat\Phi[\wdi] \coleq |\D\wdi|_\p(\Omd) + \int_{\partial \Omd} \p \left( \,.\, ,  (\wdi{-}{u_0}_\lozenge) \nud  \right) \,\d\H^N
    +\llangle{\mu_\pm}_\lozenge\,;(\wdi)^\mp\rrangle\,.
  \]
  Here, $\widehat\Phi$ is an anisotropic total variation functional with measures of the type treated in \cite{FS}. In fact, Lemma \ref{lem:p} gives the relevant assumptions for the integrand $\p$, Lemma \ref{lem:iff_admiss_mud} ensures admissibility of $\mu_\pm$, and most importantly by Proposition \ref{prop:higher_dim} the assumed $\finf$-IC for $(\mu_-,\mu_+)$ and $\widetilde{\finf}$-IC for $(\mu_+,\mu_-)$ imply the $\p$-IC for $\big({\mu_-}_\lozenge,{\mu_+}_\lozenge\big)$ and the $\widetilde{\p}$-IC for $\big({\mu_+}_\lozenge,{\mu_-}_\lozenge\big)$ with constant $1$. Therefore, \cite[Theorem 3.5]{FS} applies for the functional $\widehat\Phi$ and guarantees its lower semicontinuity along the sequence $({u_k}_\lozenge)_k$. Since moreover we have $f(\,.\,,0)\in\L^\infty(\Om)$, the second term on the right-hand side of \eqref{eq:F=P+rest} is even continuous with respect to strong convergence in $\L^1(\Om)$. All in all, we thus conclude
  \begin{align*}
    \liminf_{k \to \infty} \MF_{u_0}^\mu[u_k] 
    &= \liminf_{k \to \infty} \widehat\Phi[{u_k}_\lozenge] - 2 \lim_{k \to \infty} \int_{\Om} f(\,.\,,0)|u_k{-}u_0|\,\d\mathcal{L}^{N} - \frac{\mu(\Om)}{2} \\
    &\geq  \widehat\Phi[u_\lozenge]  - 2 \int_{\Om} f(\,.\,,0)|u{-}u_0|\,\d\mathcal{L}^{N} - \frac{\mu(\Om)}{2} 
    = \MF_{u_0}^\mu[u]\,.
  \end{align*}
  This establishes the lower semicontinuity claim of the theorem. 
\end{proof}

\begin{rem}[on the role of assumption \eqref{assum:H4} for semicontinuity]\label{rem:H4_for_LSC}
  The proof of Theorem \ref{thm:LSC} exploits \eqref{assum:H4}, in fact its recasting \eqref{assum:H4'}, in two regards\textup{:}
  
  First, \eqref{assum:H4'} together with \eqref{assum:H2} ensures convexity of\/ $\p$ in $(\xi_0,\xi)\in\R\times\R^N$ \ka cf.\@ Lemma \ref{lem:p}\eqref{item:p:iii}\kz, and this convexity in turn allows for dealing with $\widehat\Phi$ in the preceding proof of Theorem \ref{thm:LSC} via \textup{\cite[Theorem 3.5]{FS}}. In contrast, when dropping \eqref{assum:H4} we would merely have convexity in $(\xi_0,\xi)\in[0,\infty)\times\R^N$ \ka i.\@e.\@ restricted to $\xi_0\ge0$\kz{} at our disposal and would not reach the exact framework of \textup{\cite[Theorem 3.5]{FS}}. It seems likely that this technical point can be overcome by taking Reshetnyak-type semicontinuity for measures with values in the cone $[0,\infty)\times\R^N$ as a starting point \ka cf.\@ \textup{\cite[Theorem 2.4]{BecSch13}}\kz{} and by correspondingly adapting a larger chunk of arguments from \textup{\cite{FS}}. However, not all necessary adaptations are entirely straightforward. For instance, one may no longer rely on an underlying parametric theory in full space $\R^{N+1}$ only, as provided by \cite[Section 6]{FS}. In any case, when dropping \eqref{assum:H4} we can no longer proceed by reduction to a standard anisotropic $\TV$ framework and by using \textup{\cite[Theorem 3.5]{FS}} as stated.
  
  Furthermore, \eqref{assum:H4'} in form of Lemma \ref{lem:p}\eqref{item:p:v} also enters into the verification of the ICs for ${\mu_\pm}_\lozenge$ in Proposition \ref{prop:higher_dim} and seems more or less indispensable in deriving exactly these ICs. Still, since the proof of Theorem \ref{thm:LSC} truly necessitates semicontinuity of\/ $\widehat\Phi$ only \emph{on the subset}\/ $\{\wdi:w\in\BV(\Om)\}$ of\/ $\BV(\Omd)$, one may hope to get through with weaker ICs and without need for \eqref{assum:H4}. Once more, however, this cannot be achieved by reduction to the exact framework of \textup{\cite{FS}} and rather requires revisiting the theory developed there to a more cumbersome extent.
\end{rem}

\subsection{Coercivity and existence} \label{subsec:coerc}

We now identify ICs as criteria for coercivity of the functional $\MF_{u_0}^\mu$, and then we briefly conclude the proof of the existence result in Theorem \ref{thm:exist_inhom}. Essentially these considerations parallel the very classical ones for the area case; see e.\@g.\@ \cite{Giaquinta74a,Giaquinta74b}. However, in our framework it seems worth remarking that the ICs bound the measures $\mu_\pm$ in terms of the recession function $f^\infty$ only. In this sense the intuition that coercivity should depend only on the behavior of the integrand $f(x,\xi)$ for $|\xi|\to\infty$ is confirmed. For a quite analogous discussion of coercivity with $\W^{-1,\infty}$ right-hand sides and bounds in terms of $f^\infty$, we refer to \cite[Section 1.1]{BecBulGme20}.

\begin{prop}[necessity of ICs for coercivity] \label{prop:inhom_CN_coercivity}
  We consider $u_0\in\W^{1,1}(\R^N)$ and impose Assumption \ref{assum:f} for $f$. If $\mu_\pm$ are admissible measures on $\Om$ and if $\MF_{u_0}^\mu$ is bounded from below on $\BV(\Om)$, then $(\mu_-,\mu_+)$ satisfies the $\finf$-IC in $\Om$ with constant\/ $1$, and $(\mu_+,\mu_-)$ satisfies the $\widetilde\finf$-IC in $\Om$ with constant\/ $1$.
\end{prop}

\begin{proof}
  We argue by contradiction. Suppose $(\mu_-,\mu_+)$ does not satisfy the $\finf$-IC in $\Om$ with constant $1$, that is, there exists a measurable $A\Subset\Om$ such that
  \[
    \mu_-(A^+) - \mu_+(A^1) >  \P_{\finf}(A)\,.
  \]
  Then, for $u_k \coleq k \1_{A} \in \BV(\Om)$, $k \in \N$, we compute by definition of our functional
  \begin{align*}
    \MF_{u_0}^\mu[u_k]
    &=\int_\Om f(\,.\,,0) \dx + k \P_{\finf}(A) + 
      \int_{\partial\Om}  \finf(\,.\,,-u_0 \nu_\Om)\,\d\H^{N-1}
      - k \mu_-(A^+) + k\mu_+(A^1) \\
    & = k\left( \P_{\finf}(A) -\mu_-(A^+)+\mu_+(A^1)\right)
      + \mathrm{const}(\Om,f,u_0)\,.
  \end{align*}
  and thus obtain $\lim_{k\to\infty}\MF_{u_0}^\mu[u_k]={-}\infty$. This contradicts the boundedness of $\MF_{u_0}^\mu$ from below and thus establishes the claimed IC for $(\mu_-,\mu_+)$. The IC for $(\mu_+,\mu_-)$ follows analogously (with $u_k \coleq - k \1_A$).
\end{proof}

\begin{prop}[sufficiency of ICs for coercivity] \label{prop:inhom_CS_coercivity}
  Consider $u_0\in\W^{1,1}(\R^N)$ and impose Assumption \ref{assum:f} for $f$ with $\Om=\R^N$, \eqref{assum:H4'}. If $\mu_\pm$ are admissible measures on $\Om$ such that\/ $(\mu_-,\mu_+)$ satisfies the $\finf$-IC in $\Om$ with constant $C \in [0,1)$ and $(\mu_+,\mu_-)$ satisfies the $\widetilde{\finf}$-IC on $\Om$ with constant $C \in [0,1)$, then $\MF_{u_0}^\mu$ is coercive on $\BV(\Om)$ in the sense of\/ $\MF_{u_0}^\mu[w]\ge \nu\|w\|_{\BV(\Om)}-L$ for all\/ $w\in\BV(\Om)$ with constants $\nu>0$ and $L\in\R$. Moreover, in the borderline case of ICs with $C=1$, $\MF_{u_0}^\mu$ is at least bounded from below on $\BV(\Om)$.
\end{prop}

\begin{proof}
  Both assumed ICs together yield by \cite[Remark 4.3]{FS} the inequality
  \[
    - \llangle\mu_\pm\,;w^\mp\rrangle
    \leq C \left(  |\D w|_{\finf}(\Om) + \int_{\partial \Om} \finf(\,.\,, w \nu_{\Om} )\,\d\H^{N-1}\right)
  \]
  for all $w \in \BV(\Om)$. We now use in turn assumption \eqref{assum:H4'}, the preceding inequality and the triangle inequality of \eqref{anis_eq:subadd_phi}, the lower bound in \eqref{eq:lin_growth_finf}, and Poincar\'e's inequality (cf.\@ \cite[eq.\@ (2.6)]{FS}). In this way, in the case $C<1$ we derive
  \[\begin{aligned}
    \MF_{u_0}^\mu[w]
    &\geq |\D w|_{\finf}(\Om) + \int_{\partial \Om} \finf \left( \,.\, , (w{-}u_0) \nu_\Om  \right) \,\d\H^{N-1} + \llangle\mu_\pm\,;w^\mp\rrangle \\ 
    &\geq (1{-}C) \left( |\D w|_{\finf}(\Om) + \int_{\partial \Om} \finf \left( \,.\, , w \nu_\Om  \right) \d\H^{N-1} \right)\, - \int_{\partial \Om} \finf \left( \,.\, ,  u_0 \nu_\Om  \right) \,\d\H^{N-1}\\
    &\geq (1{-}C)\alpha \left( |\D w|(\Om) + \int_{\partial \Om} |w|\,\d\H^{N-1} \right) 
    - L\\
    &\geq \nu \|w\|_{\BV(\Om)} - L
  \end{aligned}\]
  with $\nu > 0$, which depends on $C$, $\alpha$, and the Poincar\'e constant, and with $L \coleq \int_{\partial \Om} \finf \left( \,.\,,  u_0 \nu_\Om  \right) \,\d\H^{N-1}$. Clearly, in the case $C=1$ we get $\MF_{u_0}^\mu[w]\ge {-}L$ in an analogous (and in fact even slightly simpler) way.
\end{proof}

\begin{rem}[on the role of assumption \eqref{assum:H4} for coercivity]\label{rem:H4_for_coercivity}
  While the given proof of Proposition \ref{prop:inhom_CS_coercivity} exploits \eqref{assum:H4}, in fact its variant \eqref{assum:H4'}, a refined reasoning ensures coercivity in the case $C<1$ even without assuming \eqref{assum:H4} or \eqref{assum:H4'}. Indeed, one still has continuity of $\fpers$ (cf.\@ Lemma \ref{lem:properties_rec_persp}\eqref{item:properties_rec_persp_iv}) and as consequence obtains $|f(x,\xi)-\finf(x,\xi)|\le(|\xi|+1)\omega(|\xi|)$ for all $(x,\xi)\in\R^N\times\R^N$ with some $\omega\colon[0,\infty)\to[0,\infty)$ such that $\lim_{s\to\infty}\omega(s)=0$ holds. Taking into account \eqref{eq:lin_growth_finf} one deduces $f(x,\xi)\ge C_\ast\finf(x,\xi)-M$ for all $(x,\xi)\in\R^N\times\R^N$ with any fixed $C_\ast\in(C,1)$ and some corresponding $M\in\R$, and this last observation then suffices to check coercivity by a reasoning analogous to the one above.

  In contrast, the claim on boundedness of $\MF_{u_0}^\mu$ from below in the limit case $C=1$ does inevitably depend on \eqref{assum:H4}. Indeed, all our assumptions except \eqref{assum:H4} are fulfilled in $N=2$ dimensions on $\Omega\coleq\B_2(0)\Subset\R^2$ for $f(x,\xi)\coleq|\xi|+1-\sqrt{|\xi|+1}$ with $\finf(x,\xi)=|\xi|$ and for $\mu_+\colequiv0$, $\mu_-\coleq H\mathcal{L}^2\resmes\B_2(0)$ with $H(x)\coleq\frac1{|x|}$; compare \cite[Section 5]{FS} for ways of verifying the limit IC for $\mu_-$. Still, for $v_k\coleq k^2\max\{\min\{w_k,1\},0\}$ with $w_k(x)\coleq1{-}k(|x|{-}1)$, it is a matter of computation checking that $|\D v_k|(\B_2(0))=\pi(2k^2{+}k)=\int_{\B_2(0)}v_k\,\d\mu_-$ and $\MF_0^\mu[v_k]=\int_{\B_2(0)}\big[1{-}\sqrt{|\nabla v_k|{+}1}\big]\,\d\mathcal{L}^2=\pi\big[1{-}\sqrt{k^3{+}1}\big]\big(\frac2k{+}\frac1{k^2}\big)$ hold. This yields $\lim_{k\to\infty}\MF_0^\mu[v_k]={-}\infty$ and confirms that in this exemplary case $\MF_0^\mu$ is unbounded from below.
\end{rem}

With suitable lower semicontinuity and coercivity at hand, the proof of existence is now a routine matter:

\begin{proof}[Proof of Theorem \ref{thm:exist_inhom}]
  We consider a minimizing sequence $(u_k)_k$ for $\MF_{u_0}^\mu$ in $\BV(\Om)$. Since we assume $C<1$, the coercivity property of Proposition \ref{prop:inhom_CS_coercivity} implies boundedness of $(u_k)_k$ in $\BV(\Om)$, and a subsequence $\big(u_{k_\ell}\big)_\ell$ converges in $\L^1(\Om)$ to some $u \in \BV(\Om)$. By the lower semicontinuity result of Theorem \ref{thm:LSC} we conclude
  \[
    \MF_{u_0}^\mu[u]
    \le \liminf_{k \to \infty} \MF_{u_0}^\mu[u_k]
    = \inf_{w \in \BV(\Om)} \MF_{u_0}^\mu[w]\,.
  \]
  Thus, $u$ is a minimizer of $\MF_{u_0}^\mu$ in $\BV(\Om)$.
\end{proof}

\section{An example of non-existence in case of the borderline IC}
  \label{sec:counterexample}

In this section we prove the claims made in Example \ref{ex:anis_area} by a reasoning in parts analogous to \cite[Section 5.3]{FS}. We start with an auxiliary lemma which verifies a suitable anisotropic IC.

\begin{lem}\label{lem:IC-p_circ}
  We assume that $\p \colon \R^2 \to {[0,\infty)}$ is positively 1-homogeneous and convex and that it satisfies $\p(\xi) > 0$  for all $\xi \in \R^2 \setminus \{0\}$. Then we have
  \begin{equation}\label{eq:IC-p_circ}
    \int_A\frac1{\p^{\circ}(x)}\dx \leq 
    \P_{\widetilde\p}(A)
  \end{equation}
  for every $A\subseteq\R^2$ such that $|A|<\infty$. Moreover, equality occurs in \eqref{eq:IC-p_circ} if and only if\/ $|A\,\triangle\,\{\p^\circ<r\}|=0$ holds for some $r\in{[0,\infty)}$.
\end{lem}

\begin{proof}
  We first verify the auxiliary equality
  \begin{equation}\label{claim_per_pcirc}
    \int_{\left\{ \p^\circ < r  \right\}}\frac1{\p^{\circ}(x)}\dx\,
    = \P_{\widetilde\p}(\left\{ \p^\circ < r  \right\})
    \qq\text{for all }r\in[0,\infty)\,.
  \end{equation}
  Indeed, taking into account the homogeneity of $\p$, we can recast the standard coarea formula (see e.\@g.\@ \cite[Section 3.4.3]{EG92}) in form of the anisotropic coarea formula
  \[	
    \int_U g\,\p(\nabla H)\dx
    =\int_{-\infty}^\infty\int_{U\cap\{H=t\}}g\,\p(\nu_H)\,\d\H^1\dt
  \]
  for any Lipschitz function $H\colon U\to\R$ such that $\nabla H\neq0$ \ae{} in $U$ and any Borel function $g\colon U\to[0,\infty)$ on open $U \subseteq \R^2$. Here, $\nu_H\coleq\frac{\nabla H}{|\nabla H|}$ is defined \ae{} in $U$ and is consistent with the notation of Definition \ref{defi:anis_TV} for the Radon-Nikod\'ym derivative. We now combine the \ae{} equality $\p(\nabla \p^\circ)=1$ from Lemma \ref{lem:rel_p_gradient} and the preceding formula with $H=\p^\circ$ and $g=1/\p^\circ$ on $U = \{\p^\circ<r\}$. In this way we infer
  \begin{equation}\label{eq:IC-p_circ-1}\begin{aligned}
    \int_{\{\p^\circ<r\}}\frac1{\p^\circ}\dx 
    = \int_{\{\p^\circ<r\}} \frac1{\p^\circ} \, \p(\nabla\p^\circ) \dx
    = \int_0^r \frac1t \int_{\{\p^\circ=t\}} \p(\nu_{\p^\circ}) \,\d\H^1 \dt\,.
  \end{aligned}\end{equation}
  Next, taking into account convexity and homogeneity of $\p^\circ$, we deduce that $\{\p^\circ<t\}$ is a bounded convex set with $\partial^\ast\{\p^\circ<t\}=\partial\{\p^\circ<t\}=\{\p^\circ=t\}$ for all $t>0$, and moreover we record that $\nu_{\p^\circ}=-\nu_{\{\p^\circ<t\}}$ holds $\H^1$-\ae{} on $\{\p^\circ=t\}$ for \ae{} $t>0$ at least; compare e.\@g.\@ \cite[eqn (2.19)]{FS} for the last property. These observations together with the 1-homogeneity of $\p^\circ$ allow to recognize
  \begin{equation}\label{eq:IC-p_circ-2}
    \int_{\{\p^\circ=t\}} \p(\nu_{\p^\circ}) \,\d\H^1
    =\P_{\widetilde\p}(\{\p^\circ<t\})
    =\P_{\widetilde\p}\Big(\frac tr\{\p^\circ<r\}\Big)
    =\frac tr\,\P_{\widetilde\p}(\{\p^\circ<r\})
    \qq\text{for \ae{} }t>0\,.
  \end{equation}
  The combination of the previous chains of equalities then straightforwardly yields the auxiliary claim \eqref{claim_per_pcirc}. 

  \medskip

  Now we consider $A \subseteq \R^2$ such that $0<|A|<\infty$, and we fix $r\in(0,\infty)$ such that $|\{\p^\circ<r\}|=|A|$. In view of $|A\setminus\{\p^\circ<r\}|=|\{\p^\circ<r\}\setminus A|$ we infer
  \[\begin{aligned}
    \int_A \frac1{\p^{\circ}} \dx
    &\leq \frac1r \, |A\cap\{\p^\circ\ge r\}|
      + \int_{A\cap\{\p^\circ<r\}} \frac1{\p^{\circ}} \dx \\
    & = \frac1r \, |\{\p^\circ<r\}\setminus A|
      \, + \int_{A\cap\{\p^\circ<r\}} \frac1{\p^{\circ}} \dx
    \leq \int_{\{\p^\circ<r\}} \frac1{\p^{\circ}}\dx\,,
  \end{aligned}\]
  where the condition for turning both inequalities into equalities is precisely $|A\,\triangle\,\{\p^\circ<r\}|=0$. Moreover, the equality \eqref{claim_per_pcirc} and the anisotropic isoperimetric inequality of Theorem \ref{thm:anis_isop_ineq} yield
  \[
    \int_{\{\p^\circ<r\}} \frac1{\p^{\circ}}\dx
    = \P_{\widetilde\p}(\{\p^\circ<r\})
    \leq \P_{\widetilde\p}(A)
  \]
  with equality in particular in case $|A\,\triangle\,\{\p^\circ<r\}|=0$. The combination of these observations then implies both the claimed inequality \eqref{eq:IC-p_circ} and the characterization of its equality cases.
\end{proof}

\begin{rem}\label{rem:IC-p_circ_mirr}
  Replacing $\p$ with the mirrored integrand $\widetilde\p$, from Lemma \ref{lem:IC-p_circ} we obtain also
  \[
    \int_A\frac1{{\widetilde\p}^{\circ}(x)}\dx
    \leq \P_\p(A)\,,
  \]
  under the same assumptions and with equality precisely in case $|A\,\triangle\,\{{\widetilde\p}^\circ<r\}|=0$.
\end{rem}

\begin{proof}[Proof of the claims from Example \ref{ex:anis_area}] 
  By Lemma \ref{lem:IC-p_circ} in the modified version of Remark \ref{rem:IC-p_circ_mirr}, the measure $\mu_-=(1/{\widetilde\p}^\circ)\mathcal{L}^2$ satisfies the $\p$-IC in $\R^2$ with the borderline constant $1$. In view of \cite[Theorem 4.2]{FS} we equivalently recast this IC as
  \begin{equation} \label{eq:IC_ex_anis_area}
    \int_{\R^2} \frac{w}{{\widetilde\p}^{\circ}}\dx
    \leq |\D w|_\p(\R^2) \quad \text{ for all } w \in \BV_\c(\R^2)\,.
  \end{equation}
  Furthermore, we record that $f$ given by $f(\xi)\coleq\sqrt{1+\p(\xi)^2}$ falls under Assumption \ref{assum:f} with $\finf(\xi)=\p(\xi)$ and $\fpers(t,\xi)=\sqrt{t^2+\p(\xi)^2}$. For the corresponding functional of measures, we observe the \emph{strict} inequality
  \[
    \int_\Om\sqrt{1+\p^2(\D w)}>|\D w|_\p(\Om)
    \qq\qq\text{for all }w\in\BV(\Om)\,,
  \]
  and consequently via the IC in \eqref{eq:IC_ex_anis_area} we find 
  \[
    \A_\p[w]
    =\int_{\R^2}\sqrt{1+\p(\D\ol{w})^2}
      -\int_{\R^2}\frac{\ol{w}}{{\widetilde\p}^{\circ}}\dx
    >|\D\ol{w}|_\p(\R^2)
      -\int_{\R^2}\frac{\ol{w}}{{\widetilde\p}^{\circ}}\dx 
    \geq 0
  \] 
  for all $w \in \BV(\Om)$, where $\ol{w}$ denotes the extension of $w$ to all of $\R^2$ with value $0$ outside $\Om$. Recalling $\Om=\{{\widetilde\p}^\circ<1\}$, we now define a sequence of functions $u_k\in\W^{1,1}_0(\Om)$ by setting\footnote{Alternatively, the example can be built with $u_k(x)\coleq k g({\widetilde\p}^\circ(x))$ for any fixed decreasing $\C^1$ function $g \colon [0,1] \to [0,\infty)$ such that $g(1)=0$. The extremality property of $\ol{u_k}$ can then be checked by a coarea argument.}
  \[
    u_k(x)\coleq k(1-{\widetilde\p}^{\circ}(x))\,.
  \]
  Moreover, since $\Om$ is the ${\widetilde\p}^\circ$-unit ball in $\R^2$, in analogy with the isotropic case (e.\@g.\@ by a computation analogous to the one in \eqref{eq:IC-p_circ-1} and \eqref{eq:IC-p_circ-2}) we get $\P_\p(\Om)=2|\Om|$. We then combine the definition of $u_1$, the equality case of Remark \ref{rem:IC-p_circ_mirr}, the preceding observation, and Lemma \ref{lem:rel_p_gradient} in computing
  \[
    \int_\Om\frac{u_1}{{\widetilde\p}^\circ}\dx
    =\int_\Om\frac1{{\widetilde\p}^\circ}\dx-|\Om|
    =\P_\p(\Om)-|\Om|
    =|\Om|
    =\int_\Om\widetilde\p(\nabla{\widetilde\p}^\circ)\dx
    =\int_\Om\p(\nabla u_1)\dx\,.
  \]
  Thus, we conclude that $\ol{u_1}$ and in fact all $\ol{u_k}$ are extremals for the IC \eqref{eq:IC_ex_anis_area}. In turn, this extremality property and the homogeneity of $\p$ yield
  \[\begin{aligned}
    \A_\p[u_k] 
    & = \int_\Om \sqrt{1+\p(k\nabla u_1)^2}\dx
      - k \int_\Om \frac{u_1}{{\widetilde\p}^\circ}\dx\\
    & = \int_\Om \big( \sqrt{1+k^2\p(\nabla u_1)^2} -k\p(\nabla u_1) \big)\dx  \\
    & = \int_\Om \frac1{\sqrt{1+k^2\p(\nabla u_1)^2}+k\p(\nabla u_1)} \dx\\
    & \leq  \int_\Om \frac1{1+k\p(\nabla u_1)} \dx \xrightarrow[k \to \infty]{} 0\,,
  \end{aligned}\]
  where the final convergence results from the dominated convergence theorem and crucially exploits the observation that $\nabla u_1 = -\nabla{\widetilde\p}^\circ\neq0$ and hence $\p(\nabla u_1)\neq0$ hold \ae{} in $\Om$. Collecting the previous findings, we have shown
  \[
    \inf_{\BV(\Om)}\A_\p=0<\A_\p[w]
    \qq\text{for all }w \in \BV(\Om)\,.
  \]
  Therefore, the minimum of $\A_\p$ is not attained.
\end{proof}

\section{Construction of recovery sequences}\label{sec:rec_seq}

In this section, we work out the proof of Theorem \ref{thm:rec_seq}. The implementation follows \cite[Section 8]{FS} and merely requires comparably minor adaptations. So, we here give an account on the general strategy of proof, restate relevant auxiliary results in the present framework, and indicate the necessary adaptations. For full details of some technical procedures, however, we still refer the reader to \cite[Section 8]{FS}.

Before going into the details, we recall once more the general assumption that $\Om$ is a bounded open set with Lipschitz boundary in $\R^N$. In addition, as in Section \ref{sec:LSC_exist}, we assume also here that $f$ is defined on $\R^N\times\R^N$ and satisfies Assumption \ref{assum:f} to the extent relevant for Theorem \ref{thm:rec_seq} (i.\@e.\@ lower bound in \eqref{assum:H1} weakened to non-negativity and \eqref{assum:H4} dropped) with $\Om=\R^N$.

This said, in analogy with \cite[Proposition 8.1]{FS} we initially record that a first type of recovery sequence --- indeed a sequence in $\W^{1,1}(\Om)$, but not yet in $\W^{1,1}_{u_0}(\Om)$ --- is straightforwardly available from \cite[Proposition 4.4]{FS} or in slightly different framework from \cite[Lemma 4.1]{LeoCom24_v5}.

\begin{prop}[recovery sequences with unconstrained boundary values]\label{prop:free_rec_seq}
  We impose on $f$ the above assumptions and consider admissible measures $\mu_\pm$ on $\Om$ such that $\mu_+$ and $\mu_-$ are singular to each other. Then, for every $u\in\BV(\Om)$, there exists a sequence $(w_k)_k$ in $\W^{1,1}(\Om)$ such that $(w_k)_k$ converges to $u$ area-strictly in $\BV(\Om)$ with
  \[
    \lim_{k\to\infty}\MF_{u_0}^\mu[w_k]
    =\MF_{u_0}^\mu[u]
  \]
  for every $u_0\in\W^{1,1}(\R^N)$. More specifically, for the single terms, we achieve
  \begin{gather}
    \lim_{k\to\infty}\int_\Om f(\,.\,,\nabla w_k)\dx
    =\int_\Om f(\,.\,,\D u)\,,
      \label{eq:free_rec_seq_1}\\
    \lim_{k\to\infty}\int_{\partial\Om}\finf(\,.\,,(w_k{-}u_0)\nu_\Om)\,\d\H^{N-1}
    =\int_{\partial\Om}\finf(\,.\,,(u{-}u_0)\nu_\Om)\,\d\H^{N-1}\,,
      \label{eq:free_rec_seq_2}\\
    \lim_{k\to\infty}\int_\Om w_k^\ast\,\d\mu_-=\int_\Om u^+\,\d\mu_-
    \qq\qq\text{and}\qq\qq
    \lim_{k\to\infty}\int_\Om w_k^\ast\,\d\mu_+=\int_\Om u^-\,\d\mu_+\,.
      \label{eq:free_rec_seq_3}
  \end{gather}
\end{prop}

\begin{proof}
  We rely on the approximation result of \cite[Proposition 4.4]{FS} and the observation of \cite[Remark 4.5]{FS} that this result remains valid even with \emph{area-}strict convergence. This gives existence of a sequence $(w_k)_k$ in $\W^{1,1}(\Om)$ such that $(w_k)_k$ converges to $u$ area-strictly in $\BV(\Om)$ with \eqref{eq:free_rec_seq_3}. Then Theorem \ref{thm:Resh_cont_inhom_BV} yields \eqref{eq:free_rec_seq_1}, and via \cite[Theorem 3.88]{AFP00} we deduce first convergence of the traces in $\L^1(\partial\Om\,;\H^{N-1})$ and then \eqref{eq:free_rec_seq_2}. Finally, the combination of \eqref{eq:free_rec_seq_1}, \eqref{eq:free_rec_seq_2}, \eqref{eq:free_rec_seq_3} entails $\lim_{k\to\infty}\MF_{u_0}^\mu[w_k]=\MF_{u_0}^\mu[u]$.
\end{proof}

The subsequent auxiliary lemma resembles \cite[Lemma 8.2]{FS} and allows for adjusting the boundary values of a recovery sequence at least in case the prescribed boundary datum $u_0$ is in $\L^\infty(\R^N)$. The precise statement employs once more the notation $\ol{u}\coleq\1_\Om u+\1_{\R^N\setminus\ol{\Om}}u_0\in\BV(\R^N)$ for the extension of $u\in\BV(\Om)$ via the values of the given datum $u_0\in\W^{1,1}(\R^N)$.

\begin{lem}[from unconstrained boundary values to $\L^\infty$ boundary values] \label{lem:recovery-W-Wu0}
  We again impose on $f$ the above assumptions and consider admissible measures $\mu_\pm$ on $\Om$. Then, for every $w\in\W^{1,1}(\Om)$ and every $u_0 \in \W^{1,1}(\R^N)\cap\L^\infty(\R^N)$, there exists a sequence $(v_k)_k$ in $\W^{1,1}_{u_0}(\Om)$, thus $\MF_{u_0}^{0}[v_k]=\mathrm{F}^{0}[v_k]=\int_\Om f(\,.\,,\nabla v_k)\dx$, such that $(\ol{v_k})_k$ converges to $\ol{w}$ area-strictly in $\BV(\pOm)$, on any open $\pOm\subseteq\R^N$ such that $\Om\Subset\pOm$, $|\pOm|<\infty$, with
  \[
    \lim_{k\to\infty}\mathrm{F}^\mu[v_k]=\MF_{u_0}^\mu[w]\,.
  \]
  More specifically, for the single terms, we achieve
  \begin{gather*}
    \lim_{k \to \infty}\int_\Om f(\,.\,,\nabla v_k)\dx=\int_\Om f(\,.\,,\nabla w)\dx+\int_{\partial\Om}\finf(\,.\,,(w{-}u_0)\nu_\Om)\,\d\H^{N-1}\,,\\
    \lim_{k\to\infty}\int_\Om v_k^\ast\,\d\mu_-=\int_\Om w^\ast\,\d\mu_-
    \qq\qq\text{and}\qq\qq
    \int_\Om v_k^\ast\,\d\mu_+=\int_\Om w^\ast\,\d\mu_+\,.
  \end{gather*}
\end{lem}

\begin{proof}
  We follow the reasoning previously developed for \cite[Lemma 8.2]{FS}. In rough summary, this reasoning involves: a result on strict approximation of the arbitrary $w\in\W^{1,1}(\Om)$ by functions $u_\ell\in\W^{1,1}_{u_0}(\Om)$ with prescribed boundary datum $u_0$, the usage of truncations $(u_\ell)^M$ of $u_\ell$ at levels $M\ge\|u_0\|_{\L^\infty(\R^n)}$, the passage to the limit with the $\mu$-independent terms via the Reshetnyak continuity theorem and with $\int_\Om\big((u_\ell)^M\big)^\ast\d\mu_\pm$ via the dominated convergence theorem, and finally the choice $v_k\coleq(u_{\ell_k})^{M_k}$ for suitable $M_k,\ell_k\to\infty$. In fact, the main deviation from \cite{FS} is that here we involve general functionals of measures in the sense of Section \ref{subsec:func_meas} instead of just anisotropic total variations. However, we can take $\ol{u_\ell}$ even as \emph{area-}strict approximations of $\ol{w}$ in $\BV(\pOm)$ (see \cite[Lemma B.2]{Bildhauer03} or \cite[Theorem 1.2]{Schmidt15}), and then we can rely on the rewriting of \eqref{eq:F-enlarge-domain} and can still apply the Reshetnyak continuity theorem on the enlarged domain $\pOm$ if we only use this theorem in the inhomogeneous version of Theorem \ref{thm:Resh_cont_inhom_BV} rather than the homogeneous version of Theorem \ref{thm:Resh_cont_hom_BV}. With these minor adaptations, the above-mentioned strategy carries over to the present situation, and we refer the reader to the proof of \cite[Lemma 8.2]{FS} for further details of the implementation.
\end{proof}

From the preceding auxiliary results we obtain the desired recovery sequences in case $u_0\in\L^\infty(\R^N)$:

\begin{proof}[Proof of Theorem \ref{thm:rec_seq} in case $u_0\in\W^{1,1}(\R^N)\cap\L^\infty(\R^N)$]
  We use the sequence $(w_k)_k$ of Proposition \ref{prop:free_rec_seq} and record in particular that also $(\ol{w_k})_k$ converges to $\ol{u}$ area-strictly in $\BV(\pOm)$, on any open $\pOm\subseteq\R^N$ such that $\Om\Subset\pOm$, $|\pOm|<\infty$. In view of $u_0\in\L^\infty(\R^N)$ we can apply Lemma \ref{lem:recovery-W-Wu0} to determine, for each $k$, a sequence $(v_{k,\ell})_\ell$ in $\W^{1,1}_{u_0}(\Om)$ such that $(\ol{v_{k,\ell}})_\ell$ converges to $\ol{w_k}$ area-strictly in $\BV(\pOm)$, $\pOm$ as before, with $\lim_{\ell\to\infty}\mathrm{F}^\mu[v_{k,\ell}]=\MF_{u_0}^\mu[w_k]$ and also the other conclusions of Lemma \ref{lem:recovery-W-Wu0} valid in corresponding versions. Then we obtain all claims of Theorem \ref{thm:rec_seq} for $u_k\coleq v_{k,\ell_k}\in\W^{1,1}_{u_0}(\Om)$ by choosing, for each $k$, a suitably large $\ell_k$, and by putting together the convergence properties of Proposition \ref{prop:free_rec_seq} and Lemma \ref{lem:recovery-W-Wu0}.
\end{proof}

\begin{rem}[on the $\L^\infty$ constraints]
  The basic reason for requiring some $\L^\infty$ control in the preceding is
  that this control enables the application of the dominated convergence theorem
  in the proof of Lemma \ref{lem:recovery-W-Wu0} and thus allows for deducing
  convergence of $\mu_\pm$-measure terms from $\mu_\pm$-\ae{} convergence of
  precise representatives. Essentially we can produce the required control in
  the proof of Lemma \ref{lem:recovery-W-Wu0} itself by passing to the
  $\L^\infty$ truncations $(u_\ell)^M$, but in doing so we can preserve
  prescribed boundary values $u_0$ only if, at least for $u_0$, an extra
  $\L^\infty$ bound is assumed and we take $M\ge\|u_0\|_{\L^\infty}$. We
  refer once more to \cite[Section 8]{FS} for full details.
\end{rem}

It remains to establish Theorem \ref{thm:rec_seq} in the general case without the extra assumption $u_0\in\L^\infty(\R^N)$. To this end, we will approximate an arbitrary $u_0\in\W^{1,1}(\R^N)$ by $u_{0,k}\in\W^{1,1}(\R^N)\cap\L^\infty(\R^N)$ and will then rely on the next lemma to slightly perturb competitors $z_k$ with $z_k=u_{0,k}$ at $\partial\Om$ into $u_k$ with $u_k=u_0$ at $\partial\Om$. To avoid ambiguity, from here on we upgrade the notation for the extension of $u$ via $u_0$ from $\ol{u}$ to $\ol{u}^{u_0}$.

\begin{lem}\label{lem:approx-u0_inhom}
  We impose on $f$ the general assumptions of this section and consider admissible measures $\mu_\pm$ on $\Om$. If a sequence $(u_{0,k})_k$ converges in $\W^{1,1}(\R^N)$ to $u_0$ and a sequence $(z_k)_k$ in $\W^{1,1}(\Om)$ is such that $z_k\in\W^{1,1}_{u_{0,k}}(\Om)$ for each $k$, then there exists a sequence $(u_k)_k$ in $\W^{1,1}_{u_0}(\Om)$ such that $(u_k{-}z_k)_k$ converges to $0$ in $\W^{1,1}(\Om)$ and consequently also $\big(\ol{u_k}^{u_0}{-}\ol{z_k}^{u_{0,k}}\big)_k$ converges to $0$ in $\W^{1,1}(\R^N)$ with
  \begin{gather*}
    \lim_{k\to\infty}\bigg(\int_\Om f(\,.\,,\nabla u_k)\dx-\int_\Om f(\,.\,,\nabla z_k)\dx\bigg)=0\,,\\
    \lim_{k\to\infty}\bigg(\int_\Om u_k^\ast\,\d\mu_--\int_\Om z_k^\ast\,\d\mu_-\bigg)=0
    \qq\text{and}\qq
    \lim_{k\to\infty}\bigg(\int_\Om u_k^\ast\,\d\mu_+-\int_\Om z_k^\ast\,\d\mu_+\bigg)=0\,.
  \end{gather*}
\end{lem}
  
\begin{proof}
  All claims except the convergence with general integrand $f$ are provided by \cite[Lemma 8.3]{FS}\footnote{The statement of \cite[Lemma 8.3]{FS} is partially made for measures which satisfy an isotropic IC in $\Om$, possibly with large constant. By \cite[Proposition 4.1]{FS} the admissibility of $\mu_\pm$, as assumed in the present Lemma \ref{lem:approx-u0_inhom}, does ensure this type of IC for $\mu_\pm$, and thus all conclusions of \cite[Lemma 8.3]{FS} indeed apply here.} (which is based on the straightforward choice $u_k\coleq z_k-u_{0,k}+u_0$). However, since \ae{} on $\Om$ we have the upper bound $f(\,.\,,\nabla u_k)-f(\,.\,,\nabla z_k)
  \le\finf(\,.\,,\nabla u_k-\nabla z_k)
  \le\beta|\nabla u_k-\nabla z_k|$ together with an analogous lower bound, the $\W^{1,1}(\Om)$-convergence $u_k{-}z_k\to0$ implies this remaining convergence as well.
\end{proof}

Finally, we are ready to verify the existence of recovery sequences in $\W^{1,1}_{u_0}(\Om)$ in full generality.

\begin{proof}[Proof of Theorem \ref{thm:rec_seq} for general $u_0\in\W^{1,1}(\R^N)$.]
  Given $u_0\in\W^{1,1}(\R^N)$, we choose a sequence $(u_{0,k})_k$ in $\W^{1,1}(\R^N)\cap\L^\infty(\R^N)$ such that $(u_{0,k})_k$ converges to $u_0$ in $\W^{1,1}(\R^N)$. We then follow closely the proof of \cite[Theorem 8.4]{FS}. First, we apply Theorem \ref{thm:rec_seq} for the arbitrary given $u\in\BV(\Om)$ and any of the chosen $u_{0,k}$ as the boundary datum (for which the theorem is already established) to determine, for each $k$, a sequence $(y_{k,\ell})_\ell$ in $\W^{1,1}_{u_{0,k}}(\Om)$ such that $\big(\ol{y_{k,\ell}}^{u_{0,k}}\big)_\ell$ converges to $\ol{u}^{u_{0,k}}$ area-strictly in $\BV(\pOm)$, on any open $\pOm\subseteq\R^N$ such that $\Om\Subset\pOm$, $|\pOm|<\infty$, with
  \begin{gather*}
    \lim_{\ell\to\infty}\int_\Om f(\,.\,,\nabla y_{k,\ell})\dx
    =\int_\Om f(\,.\,,\D u)+\int_{\partial\Om}\finf(\,.\,,(u{-}u_{0,k})\nu_\Om)\,\d\H^{N-1}\,,\\
    \lim_{\ell\to\infty}\int_\Om y_{k,\ell}^\ast\,\d\mu_-
    =\int_\Om u^+\,\d\mu_-
    \qq\qq\text{and}\qq\qq
    \lim_{\ell\to\infty}\int_\Om y_{k,\ell}^\ast\,\d\mu_+
    =\int_\Om u^-\,\d\mu_+\,.
  \end{gather*}
  We then choose, for each $k$, a suitably large $\ell_k$, set $z_k\coleq y_{k,\ell_k}\in\W^{1,1}_{u_{0,k}}(\Om)$, and take into account both the convergence $u_{0,k}\to u_0$ in $\W^{1,1}(\R^N)$ and the convergence of traces $u_{0,k}\to u_0$ in $\L^1(\partial\Om\,;\H^{N-1})$. In this way, we can achieve that $\big(\ol{z_k}^{u_{0,k}}\big)_k$ converges to $\ol{u}^{u_0}$ area-strictly in $\BV(\pOm)$, $\pOm$ as before, with
  \begin{gather}
    \lim_{k\to\infty}\int_\Om f(\,.\,,\nabla z_k)\dx
    =\int_\Om f(\,.\,,\D u)+\int_{\partial\Om}\finf(\,.\,,(u{-}u_0)\nu_\Om)\,\d\H^{N-1}\,,
      \label{eq:appr1}\\
    \lim_{k\to\infty}\int_\Om z_k^\ast\,\d\mu_-
    =\int_\Om u^+\,\d\mu_-
    \qq\qq\text{and}\qq\qq
    \lim_{k\to\infty}\int_\Om z_k^\ast\,\d\mu_+
    =\int_\Om u^-\,\d\mu_+\,.
      \label{eq:appr2}
  \end{gather}
  By applying Lemma \ref{lem:approx-u0_inhom} we find a new sequence $(u_k)_k$ in $\W^{1,1}_{u_0}(\Om)$, which may be seen as a perturbation of $(z_k)_k$, such that the same convergence properties remain valid. Spelled out, this means that $\big(\ol{u_k}^{u_0}\big)_k$ converges to $\ol{u}^{u_0}$ area-strictly in $\BV(\pOm)$ and that formulas \eqref{eq:appr1}, \eqref{eq:appr2} hold verbatim in the same way with $z_k$ replaced by $u_k$. In particular, the combination of these convergences gives $\lim_{k\to\infty}\mathrm{F}^\mu[u_k]=\MF_{u_0}^\mu[u]$, and thus all claims of Theorem \ref{thm:rec_seq} are established.
\end{proof}

\begin{appendix}

\section{Codimension-one sections and general ICs}
  \label{asec:sections_ICs}

In this section we return to the proof of Proposition \ref{prop:higher_dim} in the full generality of measures $\mu_+$ and $\mu_-$ not necessarily singular to each other. However, before carrying out the main argument in this regard, we collect preliminary estimates which involve codimension-one sections of $\BV$ functions. Also in these considerations we generally assume that $\Om$ is a bounded open set with Lipschitz boundary in $\R^N$.

\subsection[Some estimates for codimension-one sections of \texorpdfstring{$\BV$}{BV} functions]{\boldmath Some estimates for codimension-one sections of $\BV$ functions}

We start with a small lemma of general measure theory, here arranged in a form suitable for our purposes.

\begin{lem}\label{lem:ae-conv}
  For an arbitrary measure space $(\mathcal{X},\mathscr{A},\mu)$, consider a sequence $(h_k)_k$ in $\L^1(\mathcal{X}\,;\mu)$ such that $h_k\ge0$ on $\mathcal{X}$ for all\/ $k\in\N$ and $h\in\L^1(\mathcal{X}\,;\mu)$. If\/ $\liminf_{k\to\infty}h_k\ge h$ is valid $\mu$-\ae{} on $\mathcal{X}$ and moreover $\lim_{k\to\infty}\int_\mathcal{X}h_k\,\d\mu=\int_\mathcal{X}h\,\d\mu$ holds, then there exists a subsequence $\big(h_{k_\ell}\big)_\ell$ such that also $\lim_{\ell\to\infty}h_{k_\ell}=h$ is valid\/ $\mu$-\ae{} on $\mathcal{X}$.
\end{lem}

\begin{proof}
  Since $\lim_{k\to\infty}(h_k{-}h)_-=0$ holds $\mu$-\ae{} on $\mathcal{X}$ and since we have $0\le(h_k{-}h)_-\le h_+\in\L^1(\mathcal{X}\,;\mu)$, we may apply the dominated convergence theorem to deduce $\lim_{k\to\infty}\int_\mathcal{X}(h_k{-}h)_-\,\d\mu=0$. As a consequence, we infer $\lim_{k\to\infty}\int_\mathcal{X}(h_k{-}h)_+\,\d\mu=\lim_{k\to\infty}\big[\int_\mathcal{X}h_k\,\d\mu-\int_\mathcal{X}h\,\d\mu+\int_\mathcal{X}(h_k{-}h)_-\,\d\mu\big]=0$, that is, $(h_k{-}h)_+$ converge to $0$ in $\L^1(\mathcal{X}\,;\mu)$. A standard result in measure theory then gives a subsequence $\big(h_{k_\ell}\big)_\ell$ such that first $\lim_{\ell\to\infty}(h_{k_\ell}{-}h)_+=0$ holds $\mu$-\ae{} on $\mathcal{X}$ and all in all also $\lim_{\ell\to\infty}h_{k_\ell}=h$ holds $\mu$-\ae{} on $\mathcal{X}$.
\end{proof}

Our main observations and estimates for sections of $\BV$ functions follow.

\begin{lem} \label{lem:var_anisotropies}
  Consider an arbitrary $V \in \BV(\Omd)$. Then, for \ae{} $x_0\in(0,1)$, there hold\textup{:}
  \begin{gather}
    V(x_0,.\,)\in \BV(\Om)\,,
      \label{eq:slice-BV}\\
    V(x_0,.\,)^\pm(x)=V^\pm(x_0,x)
      \text{ for }\H^{N-1}\text{-\ae{} }x\in\Om\,.
      \label{eq:slice-v-pm}
  \end{gather}
  Moreover, with $f$ and $\p$ as in Section \ref{sec:LSC_exist}, and with the notation $\D_x V(x_0,.\,)$ for the derivative measure of\/ $V(x_0,.\,)\in \BV(\Om)$, the $\finf$-anisotropic total variations $|\D_x V(x_0,.\,)|_{\finf}(\Om)$ are measurable in $x_0\in(0,1)$, and we have
  \begin{align}
  \int_{0}^{1} |\D_x V(x_0,.\,)|_{\finf}(\Om) \,\d x_0
    &\leq |\D V|_{\p}(\Omd)\,,
    \label{eq:var_anisotropies}\\
  \int_{0}^{1} \int_{\partial \Om} \finf(\,.\,,V(x_0,.\,) \nu_{\Om})\,\d\H^{N-1} \,\d x_0
    &\leq \int_{\partial \Omd} \p(\,.\,,V \nu_{\Omd})\,\d\H^N\,.
    \label{eq:var_anisotropies_bd}
  \end{align}
\end{lem}

Actually, \eqref{eq:slice-BV} and the isotropic version of \eqref{eq:var_anisotropies} (including the relevant measurability claim) are covered by \cite[Appendice, Teorema 3.3]{Miranda64}, and in our subsequent proof we will revisit some of the arguments provided there in order to establish our further claims.

\begin{proof}[Proof of Lemma \ref{lem:var_anisotropies}]
  To check \eqref{eq:slice-BV}, we first deduce from $V\in\L^1(\Omd)$ and Fubini's theorem that $V(x_0,.\,)\in\L^1(\Om)$ holds for \ae{} $x_0\in(0,1)$. Then we argue by approximation. Indeed, the $\BV$-version of the Meyers-Serrin theorem
  (see e.\@g.\@ \cite[Teorema 1]{AnzGiaq78} or \cite[Theorem 3.9]{AFP00}) gives a sequence $(V_k)_k$ in $\W^{1,1}(\Omd)$ such that $V_k$ converges to $V$ strictly in $\BV(\Omd)$. In view of
  \begin{equation}\label{eq:L1-slices}
    \lim_{k\to\infty} \int_{0}^{1} \|V_k(x_0,.\,)-V(x_0,.\,)\|_{\L^1(\Om)} \,\d x_0
    = \lim_{k\to\infty} \|V_k -V\|_{\L^1(\Omd)}
    = 0\,,
  \end{equation}
  we may choose a subsequence $\big(V_{k_\ell}\big)_\ell$ such that $V_{k_\ell}(x_0,.\,)$ tends to $V(x_0,.\,)$ in $\L^1(\Om)$ for \ae{} $x_0 \in (0,1)$. Then Fatou's lemma, Fubini's theorem, and the strict convergence give
  \[
    \int_0^1\liminf_{\ell\to\infty}\int_\Om|\nabla_xV_{k_\ell}(x_0,x)|\dx \,\d x_0
    \le\liminf_{\ell\to\infty}\int_\Omd|\nabla V_{k_\ell}(x_0,x)|\,\d(x_0,x)
    =|\D V|(\Omd)\,,
  \] 
  and by semicontinuity of the total variation we deduce the claim \eqref{eq:slice-BV}. In fact, we get $V(x_0,.\,)\in\BV(\Om)$ with $|\D_xV(x_0,.\,)|(\Om)\le\liminf_{\ell\to\infty}\int_\Om|\nabla_xV_{k_\ell}(x_0,x)|\dx<\infty$ for \ae{} $x_0\in(0,1)$.

  \smallskip

  Next we establish the claimed measurability of $|\D_x V(x_0,.\,)|_{\finf}(\Om)$ in $x_0\in(0,1)$ (plus another useful auxiliary assertion). We first recall that the corresponding measurability of the isotropic quantity $|\D_xV(x_0,.\,)|(\Om)$ is guaranteed already by \cite[Proposizione 3.2]{Miranda64}\footnote{The proof of \cite[Proposizione 3.2]{Miranda64} draws crucially on the usage of specific strict approximations which can be obtained, in our notation,  away from $\partial\Om$ by mollification of $V$ with respect to the variable $x\in\R^N$ only.}. Then we exploit that strict convergence of $(\nabla V_k)\mathcal{L}^{N+1}$ to $\D V$ in $\RM(\Omd,\R^{N+1})$ induces strict convergence of $(\nabla_xV_k)\mathcal{L}^{N+1}$ to $\D_xV$ in $\RM(\Omd,\R^N)$, and by slight modification of the previous reasoning we deduce
  \begin{equation}\label{eq:pre-strict-conv-slices}\begin{aligned}
    \int_0^1|\D_xV(x_0,.\,)|(\Om)\,\d x_0
    &\le\liminf_{\ell\to\infty}\int_0^1\int_\Om|\nabla_xV_{k_\ell}(x_0,x)|\dx\,\d x_0\\
    &\le\limsup_{\ell\to\infty}\int_0^1\int_\Om|\nabla_xV_{k_\ell}(x_0,x)|\dx\,\d x_0
    =|\D_xV|(\Omd)
  \end{aligned}\end{equation}
  for a suitable subsequence of every sequence $(V_k)_k$ in $\W^{1,1}(\Omd)$ which converges strictly to $V$ in $\BV(\Omd)$. Additionally we now bring in that $|\D_xV|(\Omd)\le\int_0^1|\D_xV(x_0,.\,)|(\Om)\,\d x_0$ holds by \cite[Teorema 3.3]{Miranda64}\footnote{In fact, \cite[Teorema 3.3]{Miranda64} gives even equality $|\D_xV|(\Omd)=\int_0^1|\D_xV(x_0,.\,)|(\Om)\,\d x_0$. From the proof of this equality, the argument for `$\ge$' has already been revisited in form of \eqref{eq:pre-strict-conv-slices}, while the presently relevant inequality `$\le$'  follows quite straightforwardly from the distributional characterization of the total variation and Fubini's theorem.}. This improves the preceding chain of inequalities to a chain of equalities. In terms of the auxiliary functions $h_k(x_0)\coleq\int_\Om|\nabla_xV_k(x_0,x)|\dx$ and $h(x_0)\coleq|\D_xV(x_0,.\,)|(\Om)$, for which we already know $h_k,h\in\L^1((0,1))$ and $\liminf_{\ell\to\infty}h_{k_\ell}(x_0)\ge h(x_0)$ for \ae{} $x_0\in(0,1)$, this means $\lim_{\ell\to\infty}\int_0^1h_{k_\ell}(x_0)\,\d x_0=\int_0^1h(x_0)\,\d x_0$. Possibly passing to yet another subsequence, we then achieve by Lemma \ref{lem:ae-conv} the convergence $\lim_{\ell\to\infty}h_{k_\ell}(x_0)=h(x_0)$ for \ae{} $x_0\in(0,1)$, in other words
  \begin{equation}\label{eq:strict-conv-slices}
    \lim_{\ell\to\infty}\int_\Om|\nabla_xV_{k_\ell}(x_0,x)|\dx=|\D_xV(x_0,.\,)|(\Om)
    \qq\text{for \ae{} }x_0\in(0,1)\,.
  \end{equation}
  As this confirms strict convergence in $\BV(\Om)$, at the present stage the continuity property of Theorem \ref{thm:Resh_cont_hom_BV} applies and gives
  \[
    \lim_{\ell\to\infty}\int_\Om\finf(x,\nabla_xV_{k_\ell}(x_0,x))\dx
    =|\D_xV(x_0,.\,)|_\finf(\Om)
    \qq\text{for \ae{} }x_0\in(0,1)\,.
  \]
  In particular, we may now read off the claimed measurability of $|\D_xV(x_0,.\,)|_\finf(\Om)$ in $x_0\in(0,1)$.

  \smallskip
  
  With measurability of $|\D_xV(x_0,.\,)|_\finf(\Om)$ in $x_0\in(0,1)$ at hand, the proof of \eqref{eq:var_anisotropies} is mostly an anisotropic version of the initial reasoning for \eqref{eq:slice-BV}. Indeed, we take right the same $\big(V_{k_\ell})_\ell$ and then use in turn the lower semicontinuity of Theorem \ref{thm:Resh_LSC_hom_BV}, Fatou's lemma and Fubini's theorem, the estimate $\finf(x,\xi)\le \p(x_0,x,\xi_0,\xi)$ of Lemma \ref{lem:p}\eqref{item:p:v}, and finally the continuity property of Theorem \ref{thm:Resh_cont_hom_BV}. In this way we infer
  \[\begin{aligned}
    \int_0^1|\D_xV(x_0,.\,)|_\finf(\Om)\,\d x_0    &\le\int_0^1\liminf_{\ell\to\infty}\int_\Om\finf(x,\nabla_xV_{k_\ell}(x_0,x))\dx\,\d x_0\\
    &\le\liminf_{\ell\to\infty}\int_\Omd\finf(x,\nabla_xV_{k_\ell}(x_0,x))\,\d(x_0,x)\\
    &\le\liminf_{\ell\to\infty}\int_\Omd\p(x_0,x,\partial_{x_0} V_{k_\ell}(x_0,x),\nabla_xV_{k_\ell}(x_0,x))\,\d(x_0,x)
    =|\D V|_\p(\Omd)\,,
  \end{aligned}\]
  which proves \eqref{eq:var_anisotropies}.

  \smallskip

  To prove \eqref{eq:var_anisotropies_bd}, we exploit once more the product structure of $\H^N$ on $[0,1]\times\partial\Om$ provided by Lemma \ref{lem:prod_rectifiable}. Furthermore, we make use first of $\p(x_0,x,0,\xi)=\finf(x,\xi)$ and then of the inclusion $\left([0,1]\times\partial\Om\right) \subset \partial\Omd$ with $\nu_{\Omd}(x_0,x)=(0,\nu_{\Om}(x))$ for $\H^N$-\ae{} $(x_0,x) \in \left([0,1]\times\partial\Om \right)$ and of $\p \geq 0$. This yields
  \[\begin{aligned}
    \int_{0}^{1} \int_{\partial \Om}  \finf(\,.\,,V(x_0,.\,)\nu_{\Om})\,\d\H^{N-1} \,\d x_0
    &= \int_{ [0,1] \times \partial \Om } \finf(x, V(x_0,x)\nu_{\Om}(x))\,\d\H^{N}(x_0,x) \\
    &= \int_{   [0,1] \times \partial \Om }\p(x_0,x, V(x_0,x)(0,\nu_{\Om}(x)))\,\d\H^{N}(x_0,x) \\
    &\leq \int_{\partial \Omd} \p(\,.\,,V \nu_{\Omd})\,\d\H^N
  \end{aligned}\]
  and leaves us precisely with \eqref{eq:var_anisotropies_bd}. 

  \smallskip

  Finally, we turn to \eqref{eq:slice-v-pm} and in a first step verify this claim for $V\in\W^{1,1}(\Omd)$. By the Meyers-Serrin theorem and the $1$-capacitary results in \cite[Sections 4 and 10]{FedZie72} (compare also \cite[Theorem 4]{Meyers70} and \cite[Proposition 1.2, Section 6]{DalMaso83}) there exists a sequence $(V_k)_k$ in $\W^{1,1}(\Omd)\cap\C^\infty(\Omd)$ such that $V_k$ converge to $V$ in $\W^{1,1}(\Omd)$ and moreover $V_k=V_k^\ast$ converge to $V^\ast$ pointwise $\H^N$-\ae{} in $\Omd$. By Lemma \ref{lem:Eilenberg} this implies, for \ae{} $x_0\in(0,1)$, the $\H^{N-1}$-\ae{} convergence of $V_k(x_0,.\,)$ to $V^\ast(x_0,.\,)$ in $\Om$. Next, by a reasoning analogous to \eqref{eq:L1-slices}, for a subsequence $(V_{k_\ell})_\ell$, we may moreover assume, again for \ae{} $x_0\in(0,1)$, convergence of $V_{k_\ell}(x_0,.\,)$ to $V(x_0,.\,)$ in $\W^{1,1}(\Om)$ and then also $\H^{N-1}$-\ae{} convergence of $V_{k_\ell}(x_0,.\,)=V_{k_\ell}(x_0,.\,)^\ast$ to $V(x_0,.\,)^\ast$ in $\Om$. Comparing the $\H^{N-1}$-\ae{} convergences, we infer, still for \ae{} $x_0\in(0,1)$, the $\H^{N-1}$-\ae{} equality $V^\ast(x_0,.\,)=V(x_0,.\,)^\ast$ in $\Om$, which confirms \eqref{eq:slice-v-pm} in this case. In a second step we generalize to arbitrary $V\in\BV(\Omd)$ by a similar, but slightly more involved approximation argument. Indeed, we apply \cite[Theorem 3.3]{CarDalLeaPas88} to find strict approximations of $V$ from above and exploit the fact that by \cite[Theorem 3.2]{Lahti17} these approximations necessarily are $\H^N$-\ae{} approximations of $V^+$ as well (compare e.\@g.\@ \cite[Proof of Lemma 2.23]{FS} for similar reasoning). In more technical terms, both the results together in fact guarantee existence of a sequence $(V_k)_k$ in $\W^{1,1}(\Omd)$ such that $V_k$ converge to $V$ strictly in $\BV(\Omd)$ with $V_k\ge V$ \ae{} in $\Omd$ and moreover $V_k^\ast$ converge to $V^+$ pointwise $\H^N$-\ae{} in $\Omd$. By Lemma \ref{lem:Eilenberg}, the last convergence carries over to the sections as before. Moreover, recalling that we derived \eqref{eq:strict-conv-slices} for a subsequence of an arbitrary strictly convergent sequence, for \ae{} $x_0\in(0,1)$, we furthermore achieve strict convergence of $V_{k_{\ell_m}}(x_0,.\,)$ to $V(x_0,.\,)$ in $\BV(\Om)$, and in view of $V_k(x_0,.\,)\ge V(x_0,.\,)$ by another application of \cite[Theorem 3.2]{Lahti17} we can get $\H^{N-1}$-\ae{} convergence of $V_{k_{\ell_m}}(x_0,.\,)^\ast$ to $V(x_0,.\,)^+$ in $\Om$ as well. Comparing the convergences, we deduce \eqref{eq:slice-v-pm} for the case of the approximate upper limit. Finally, \eqref{eq:slice-v-pm} for the approximate lower limit follows by applying the result obtained to ${-}V$.
\end{proof}

\subsection{ICs with extra variable for general pairs of measures}

With Lemma \ref{lem:var_anisotropies} at hand, we are ready to carry on the ICs from $(\mu_-,\mu_+)$ and $(\mu_+,\mu_-)$ to $({\mu_-}_\lozenge,{\mu_+}_\lozenge)$ and $({\mu_+}_\lozenge,{\mu_-}_\lozenge)$ even for $\mu_+$ and $\mu_-$ not necessarily singular to each other.

\begin{proof}[Proof of Proposition \ref{prop:higher_dim} in the general case]
  We first recall that by Lemma \ref{lem:iff_admiss_mud} the admissibility of $\mu_\pm$ carries over to ${\mu_\pm}_\lozenge$. Now, instead of working with ICs for sets as in Definition \ref{defi:IC}, we rather employ \cite[Theorem 4.2]{FS} and work with the equivalent reformulation of the ICs for $\BV$ functions. Then it will be enough to show that 
  \begin{equation}\label{eq:higher_dim2a}
    -\llangle\mu_\pm\,;v^\mp\rrangle
    \le C\left(|\D v|_{\finf}(\Om)
      + \int_{\partial\Om}\finf(\,.\,,v\nu_\Om)\,\d\H^{N-1}\right)
    \qq\text{for all non-negative }v\in\BV(\Om)
  \end{equation}
  implies
  \begin{equation}\label{eq:higher_dim2b}
    -\llangle{\mu_\pm}_\lozenge\,;V^\mp\rrangle
    \leq C\left(|\D V|_{\p}(\Omd)
      + \int_{\partial\Omd}\p(\,.\,,V\nu_{\Omd})\,\d\H^{N}\right)
    \qq\text{for all non-negative }V\in\BV(\Omd)
  \end{equation}
  (whereby the analogous implication with $(\mu_+,\mu_-,\finf,\p)$ replaced by $(\mu_-,\mu_+,\widetilde\finf,\widetilde\p)$ is also valid).

  Therefore, we now suppose that \eqref{eq:higher_dim2a} holds, and we consider an arbitrary non-negative $V\in\BV(\Omd)$. From \eqref{eq:slice-BV} we have $V(x_0,.\,)\in\BV(\Om)$ for \ae{} $x_0\in(0,1)$. We can thus proceed by using first Fubini's theorem along with \eqref{eq:slice-v-pm}, then \eqref{eq:higher_dim2a} with $v=V(x_0,.)$, and finally \eqref{eq:var_anisotropies} and \eqref{eq:var_anisotropies_bd}. In this way we deduce
  \[\begin{aligned}
    -\llangle{\mu_\pm}_\lozenge\,;V^\mp\rrangle
    &= \int_{0}^{1} \big({-}\llangle\mu_\pm\,;V(x_0,.\,)^\mp\rrangle\big) \,\d x_0\\
    & \leq C  \int_{0}^{1}  \left(  |\D V(x_0,\,.\,)|_{\finf}(\Om)+\int_{\partial \Om}\finf(\,.\,, V(x_0,\,.\,) \nu_\Om)\,\d\H^{N-1} \right) \,\d x_0\\
    & \leq C \left( |\D V|_{\p}(\Omd)+ \int_{\partial \Omd} \p(\,.\,,V \nu_{\Omd})\,\d\H^N\right)
  \end{aligned}\]
  and arrive at \eqref{eq:higher_dim2b}, as required.
\end{proof}

\end{appendix}

\bibliographystyle{amsplain}

\begin{thebibliography}{99}
\phantomsection\addcontentsline{toc}{section}{References}

\bibitem{AFTL97}	
A.~Alvino, V.~Ferone, G.~Trombetti, and P.-L.~Lions, \emph{{Convex symmetrization and applications}}, Ann. Inst. Henri Poincaré, Anal. Non Linéaire {{\bf14}} (1997), 275--293.

\bibitem{AFP00}
L.~Ambrosio, N.~Fusco, and D.~Pallara, \emph{{Functions of Bounded Variation and Free Discontinuity Problems}}, Oxford University Press (2000).

\bibitem{Anz86}
G.~Anzellotti, \emph{On the minima of functionals with linear growth}, Rend. Semin. Mat. Univ. Padova {\bf75} (1986), 91--110.

\bibitem{AnzGiaq78}
G.~Anzellotti and M.~Giaquinta, \emph{{Funzioni $\BV$ e tracce}}, Rend. Semin. Mat. Univ. Padova {\bf60} (1978), 1--21.
	
\bibitem{BecBulGme20}
L.~Beck, M.~Bulíček, and F.~Gmeineder, \emph{{On a Neumann problem for
  variational functionals of linear growth}}, Ann. Sc. Norm. Super. Pisa,
  Cl. Sci. (5) {{\bf21}} (2020), 695--737.

\bibitem{BecSch13}
L.~Beck and T.~Schmidt, \emph{On the Dirichlet problem for variational
  integrals in $\BV$}, J. Reine Angew. Math. {\bf674} (2013), 113--194. 

\bibitem{BecSch15}
L.~Beck and T.~Schmidt, \emph{Convex duality and uniqueness for
  $\BV$-minimizers}, J. Funct. Anal. {{\bf268}} (2015), 3061--3107.

\bibitem{Bildhauer03}
M.~Bildhauer, \emph{Convex variational problems. Linear, nearly linear and anisotropic growth conditions}, Lect. Notes Math. 1818, Springer, Berlin (2003).

\bibitem{BomGiu73}
E.~Bombieri and E.~Giusti, \emph{{Local estimates for the gradient of
  non-parametric surfaces of prescribed mean curvature}}, Comm. Pure Appl.
  Math. {{\bf26}} (1973), 381--394.

\bibitem{Busemann49}
H.~Busemann, \emph{{The isoperimetric problem for Minkowski area}}, Am. J.
  Math. {{\bf71}} (1949), 743--762.

\bibitem{CarDalLeaPas88}
M.~Carriero, G.~Dal~Maso, A.~Leaci, and E.~Pascali, \emph{{Relaxation of
  the non-parametric Plateau problem with an obstacle}}, J. Math. Pures
  Appl. (9) {{\bf67}} (1988), 359--396.

\bibitem{CarLeaPas85}
M.~Carriero, A.~Leaci, and E.~Pascali, \emph{{Integrals with respect to a
  Radon measure added to area type functionals\textup{:} semi-continuity
  and relaxation}}, Atti Accad. Naz. Lincei, VIII. Ser., Rend. Cl. Sci.
  Fis. Mat. Nat. {{\bf78}} (1985), 133--137.
  
\bibitem{CarLeaPas86}
M.~Carriero, A.~Leaci, and E.~Pascali, \emph{{Semicontinuity and relaxation
  for functionals that are sum of integrals of area type and of integrals
  with respect to a Radon measure}}, Rend. Accad. Naz. Sci. Detta XL, V.
  Ser., Mem. Mat. {{\bf10}} (1986), 1--31.

\bibitem{CarLeaPas87}
M.~Carriero, A.~Leaci, and E.~Pascali, \emph{{On the semicontinuity and the
  relaxation for integrals with respect to the Lebesgue measure added to
  integrals with respect to a Radon measure}}, Ann. Mat. Pura Appl. (4)
  {{\bf149}} (1987), 1--21.

\bibitem{ComLeo25}
G.E.~Comi and G.P. Leonardi, \emph{{Measures in the dual of
  $\BV$\textup{:} perimeter bounds and relations with divergence-measure fields}}, Nonlinear Anal., Theory Methods Appl., Ser. A
  {{\bf251}} (2025), 28 pages.

\bibitem{DaiTruWan12}
Q.~Dai, N.S.~Trudinger, and X.-J.~Wang, \emph{{The mean curvature
  measure}}, J. Eur. Math. Soc. {{\bf14}} (2012), 779--800.

\bibitem{DaiWanZho15}
Q.~Dai, X.-J.~Wang, and B.~Zhou, \emph{{The signed mean curvature
  measure}}, in: Feehan, Paul M. N. (ed.) et al., Analysis, complex
  geometry, and mathematical physics: in honor of Duong H. Phong.
  Contemp. Math. {{\bf644}} (2015), 23--32.

\bibitem{DalMaso83}
G.~Dal~Maso, \emph{{On the integral representation of certain local
  functionals}}, Ric. Mat. {{\bf32}} (1983), 85--113.

\bibitem{Delladio91}
S.~Delladio, \emph{{Lower semicontinuity and continuity of functions of
  measures with respect to the strict convergence}}, Proc. R. Soc. Edinb.,
  Sect. A {{\bf119}} (1991), 265--278. 

\bibitem{Duzaar93}
F.~Duzaar, \emph{{On the existence of surfaces with prescribed
  mean curvature and boundary in higher dimensions}}, Ann. Inst.
  Henri Poincaré, Anal. Non Linéaire {{\bf10}} (1993), 191--214.

\bibitem{DuzSte96}
F.~Duzaar and K. Steffen, \emph{{Existence of hypersurfaces with
  prescribed mean curvature in Riemannian manifolds}}, Indiana
  Univ. Math. J. {{\bf45}} (1996), 1045--1093.

\bibitem{EG92}
L.C.~Evans and R.F.~Gariepy, \emph{{Measure Theory and Fine Properties of
  Functions}}, CRC Press (1992).

\bibitem{Fed69}
H.~Federer, \emph{Geometric Measure Theory}, Springer (1969).

\bibitem{FedZie72}
H.~Federer and W.P.~Ziemer, \emph{{The Lebesgue set of a function whose distribution derivatives are $p$-th power summable}}, Indiana Univ. Math. J. {{\bf22}} (1972), 139--158.

\bibitem{FS}
E.~Ficola and T.~Schmidt, \emph{{Lower semicontinuity and existence results for anisotropic $\TV$ functionals with signed measure data}}, J. Funct. Anal. \textbf{290} (2026), 78 pages.

\bibitem{Gerhardt74}
C.~Gerhardt, \emph{{Existence, regularity, and boundary behaviour of
		generalized surfaces of prescribed mean curvature}}, Math. Z. {{\bf139}}
(1974), 173--198.

\bibitem{Giaquinta74a}
M.~{Giaquinta}, \emph{{Sul problema di Dirichlet per le superfici a curvatura
  media assegnata}}, {in: Symp. math. 14, Geom. simplett., Fis. mat., Teor.
  geom. Integr. Var. minim., Convegni 1973 (1974), 391--396}.

\bibitem{Giaquinta74b}
M.~{Giaquinta}, \emph{{On the Dirichlet problem for surfaces of prescribed mean
  curvature}}, Manuscr. Math. {{\bf12}} (1974), 73--86.

\bibitem{GMS79}
M.~Giaquinta, G.~Modica, and J.~Souček, \emph{{Functionals with linear growth
  in the calculus of variations\,\,--\,\,I}}, Commentat. Math. Univ. Carol.
  {{\bf20}} (1979), 143--156. 

\bibitem{Giusti78a}
E.~Giusti, \emph{{On the equation of surfaces of prescribed mean curvature.
  Existence and uniqueness without boundary conditions}}, Invent. Math.
  {{\bf46}} (1978), 111--137.

\bibitem{GmeRaiVSc21}
F.~Gmeineder, B.~Raită, and J.~Van Schaftingen, \emph{{On limiting trace
  inequalities for vectorial differential operators}}, Indiana
  Univ. Math. J. {{\bf70}} (2021), 2133--2176.
  
\bibitem{GofSer64}
C.~Goffman and J.~Serrin, \emph{{Sublinear functions of measures and
  variational integrals}}, Duke Math. J. {{\bf31}} (1964), 159--178.

\bibitem{KriRin10a}
J.~Kristensen and F.~Rindler, \emph{{Characterization of generalized
  gradient Young measures generated by sequences in $\W^{1,1}$ and\/
  $\BV$}}, Arch. Ration. Mech. Anal. {{\bf197}} (2010), 539--598; correction in Arch. Ration. Mech. Anal. {{\bf203}} (2012), 693--700.

\bibitem{KriRin10b}
J.~Kristensen and F.~Rindler, \emph{{Relaxation of signed integral
  functionals in $\BV$}}, Calc. Var. Partial Differ. Equ. {{\bf37}} (2010),
  29--62.

\bibitem{Lahti17}
P.~Lahti, \emph{{Strict and pointwise convergence of $\BV$ functions in metric
  spaces}}, J. Math. Anal. Appl. {{\bf455}} (2017), 1005--1021.

\bibitem{LeoCom24_v5}
G.P.~Leonardi and G.E.~Comi, \emph{{The prescribed mean curvature measure
equation in non-parametric form}}, Preprint, \href{https://arxiv.org/abs/2302.10592v5}{arXiv:2302.10592v5} (2024).

\bibitem{MerSegTro08}
A.~Mercaldo, S.~Segura de León, and C.~Trombetti, \emph{{On the
  behaviour of the solutions to $p$-Laplacian equations as $p$
  goes to $1$}}, Publ. Mat., Barc. {{\bf52}} (2008), 377--411.

\bibitem{Meyers70}
N.G.~Meyers, \emph{{A theory of capacities for potentials of functions in
  Lebesgue classes}}, Math. Scand. {{\bf26}} (1970), 255--292.

\bibitem{MeyZie77}
N.G.~Meyers and W.P.~Ziemer, \emph{{Integral inequalities of Poincar\'e and
  Wirtinger type for $\BV$ functions}}, Am. J. Math. {{\bf99}} (1977), 1345--1360.

\bibitem{Miranda64}
M.~Miranda, \emph{{Superfici cartesiane generalizzate ed insiemi di perimetro
  localmente finito sui prodotti cartesiani}}, Ann. Sc. Norm. Super. Pisa,
  Sci. Fis. Mat., III. Ser. {{\bf18}} (1964), 515--542.

\bibitem{Miranda74}
M.~Miranda, \emph{{Dirichlet problem with $\L^1$ data for the non-homogeneous
  minimal surface equation}}, Indiana Univ. Math. J. {{\bf24}} (1974), 227--241.

\bibitem{Pallara91}
D.~Pallara, \emph{{On the lower semicontinuity of certain integral
  functionals}}, Rend. Accad. Naz. Sci. XL, V. Ser., Mem. Mat.
  {{\bf15}} (1991), 57--70.

\bibitem{PhuTor08}
N.C.~Phuc and M.~Torres, \emph{{Characterizations of the existence and
  removable singularities of divergence-measure vector fields}}, Indiana Univ.
  Math. J. {{\bf57}} (2008), 1573--1597.

\bibitem{PhuTor17}
N.C.~Phuc and M.~Torres, \emph{{Characterizations of signed measures in the
  dual of\/ $\BV$ and related isometric isomorphisms}}, Ann. Sc. Norm. Super.
  Pisa, Cl. Sci. (5) {{\bf17}} (2017), 385--417.

\bibitem{Reshetnyak68}
Y.G.~Reshetnyak, \emph{{Weak convergence of completely additive vector
  functions on a set}}, Sib. Math. J. {{\bf9}} (1968), 1039--1045; translated from Russian original in: Sib. Mat. Zh. {{\bf9}} (1968), 1386--1394. 

\bibitem{SchSch16}
C.~Scheven and T.~Schmidt, \emph{{$\BV$ supersolutions to equations of
  $1$-Laplace and minimal surface type}}, J. Differ. Equations {{\bf261}} (2016), 1904--1932.

\bibitem{SchSch18}
C.~Scheven and T.~Schmidt, \emph{{On the dual formulation of obstacle
  problems for the total variation and the area functional}}, Ann.
  Inst. Henri Poincaré, Anal. Non Linéaire {{\bf35}} (2018), 1175--1207.

\bibitem{Schmidt15}
T.~Schmidt, \emph{{Strict interior approximation of sets of finite perimeter
  and functions of bounded variation}}, Proc. Am. Math. Soc. {{\bf143}}
  (2015), 2069--2084.

\bibitem{Schmidt25}
T.~Schmidt, \emph{{Isoperimetric conditions, lower semicontinuity, and
  existence results for perimeter functionals with measure data}},
  Math. Ann. {{\bf391}} (2025), 5729--5807.

\bibitem{Sim83}
L.~Simon, \emph{{Lectures on Geometric Measure Theory}}, Proceedings of the Centre for Mathematical Analysis, Australian National University, Canberra (1983).

\bibitem{Steffen76a}
K.~Steffen, \emph{{Isoperimetric inequalities and the problem of Plateau}},
  Math. Ann. {{\bf222}} (1976), 97--144.

\bibitem{Steffen76b}
K.~Steffen, \emph{{On the existence of surfaces with prescribed mean
  curvature and boundary}}, Math. Z. {{\bf146}} (1976), 113--135.

\bibitem{Ziemer95}
W.K. Ziemer, \emph{{The nonhomogeneous minimal surface equation involving a
  measure}}, Pac. J. Math. {{\bf167}} (1995), 183--200.

\bibitem{Ziemer89}
W.P. Ziemer, \emph{{Weakly Differentiable Functions. Sobolev Spaces and
  Functions of Bounded Variation}}, Springer (1989).

\end{thebibliography}

\end{document}